\documentclass[11pt,authoryear,a4paper]{article}
\usepackage{amsmath,amsthm,amsfonts}
\usepackage[pdftex]{graphicx}
\usepackage{booktabs}
\usepackage[usenames, dvipsnames]{color}
\usepackage{natbib}  
\usepackage{bm}
\usepackage{multirow}
\usepackage{enumerate}
\usepackage[labelformat=simple]{subcaption}
\usepackage[figuresleft]{rotating}
\usepackage{xcolor}
\usepackage{tikz-cd}
\usepackage{svg}
\usepackage{pdflscape}

\usepackage{url}
\usepackage[section]{placeins}

\usepackage{fullpage}
\usepackage{setspace}

\usepackage{authblk}
\usepackage{algorithm}
\usepackage{algpseudocode}
\usepackage{mathtools}
\usepackage{hyperref}
\hypersetup{
    linkcolor=blue,
    citecolor=blue,
    filecolor=magenta,      
    colorlinks=true, % This is giving an error
    urlcolor=cyan,
    pdftitle={},    pdfpagemode=FullScreen
    }

\newcommand{\II}{\mathbb{I}}
\newcommand{\NN}{\mathbb{N}}
\newcommand{\ZZ}{\mathbb{Z}}

\newcommand{\RR}{\mathbb{R}}

\newcommand{\EE}{\mathbb{E}}

\newcommand{\llceil}{\left\lceil}
\newcommand{\rrceil}{\right\rceil}
\newcommand{\llfloor}{\left\lfloor}
\newcommand{\rrfloor}{\right\rfloor}

\newtheorem{definition}{Definition}
\newtheorem{theorem}{Theorem}
\newtheorem{corollary}{Corollary}
\newtheorem{lemma}{Lemma}
\newtheorem{proposition}{Proposition}

\newtheorem{remark}{Remark}

\theoremstyle{definition}

\makeatletter
\renewcommand\AB@affilsepx{, \protect\Affilfont}
\makeatother

\begin{document}

\title{A Bi-Objective Markov Decision Process Design Approach to Redundancy Allocation with Dynamic Maintenance for a Parallel System}

\author[1]{Luke Fairley}
\author[1]{Rob Shone}
\author[1,2]{Peter Jacko}
\author[3]{Jefferson Huang}
\affil[1]{Lancaster University}\affil[2]{Berry Consultants}\affil[3]{Naval Postgraduate School}
\date{\today} %{\color{red}CONFIDENTIAL DRAFT --- PLEASE DO NOT CIRCULATE}}

\maketitle

\begin{abstract}
    The reliability of a system can be improved by the addition of redundant elements, giving rise to the well-known \emph{redundancy allocation problem} (RAP). We propose a novel extension to the RAP called the Bi-Objective Integrated Design and Dynamic Maintenance Problem (BO-IDDMP) which allows for future dynamic maintenance decisions to be incorporated. This leads to a problem with first-stage redundancy design decisions and second-stage sequential maintenance decisions under uncertainty. To the best of our knowledge, this is the first use of a continuous-time Markov Decision Process Design framework to formulate a problem with non-trivial dynamics, as well as its first use alongside bi-objective optimization. A general heuristic optimization methodology for bi-objective MDP Design problems is developed, and then applied to the BO-IDDMP. The efficiency and accuracy of our methodology are demonstrated against an exact mixed-integer linear programming solver. The heuristic is shown to be orders of magnitude faster in the majority of cases, and in only 2 out of 84 cases produces a solution that is dominated by the exact method. The inclusion of dynamic maintenance policies is shown to yield stronger and better-populated Pareto fronts, allowing more flexibility for the decision-maker. The impacts of varying parameters unique to our problem are also investigated.
\\ \ \\
    %\textbf{Keywords}: reliability, two-stage optimisation, bi-objective optimisation, matheuristics, Markov decision processes\\
    \textbf{Keywords:} Maintenance, Reliability, Markov decision processes, integer programming, bi-objective optimization\\
    % \textbf{Area of Review:} Stochastic Models
\end{abstract}
\section{Introduction}

Reliability is a feature of utmost importance in critical systems such as those in manufacturing, telecommunication, power generation and distribution, aircraft control, space exploration and satellite systems \citep{KIM201864}. Within such systems, no matter how complex, one can often identify a component or set of components that are critical to the operation of the system. These components together can be represented as a \emph{series system}, or a system that fails overall if only one of its components fails. Reliability can also be linked to the operational costs of a system, where the failure of some component may require the usage of a less desirable and more costly alternative. It is a natural goal to improve the reliability of the system, and this can be achieved in various different ways. One could consider investing in research and development to either manufacture new versions of critical components with higher component-wise reliability, or to entirely redesign the system such that the final result is more reliable overall. However, {\color{black}for many applications, particularly those using discrete component choices (e.g., most electronics),} such an approach would be time-consuming, very costly, and would have diminishing returns. An alternative is to consider the introduction of redundancy into the system, whereby system-critical components are duplicated so that if one fails, the system can switch to using a duplicate component. This gives rise to the \emph{redundancy allocation problem} (RAP).

In this work, we merge the design aspect of the RAP with a dynamic maintenance problem formulated as a Markov decision process (MDP) to create an \emph{Integrated Design and Dynamic Maintenance Problem} (IDDMP). We believe this is the first MDP Design problem that both reflects an application and has non-trivial stochastic dynamics. Additionally, we consider the first bi-objective MDP (BO-MDP) Design problem, where to the best of our knowledge all prior work has been single-objective. On top of this, while most bi-objective RAP literature considers a second objective based on one-off installation costs, or total accumulated cost in finite mission time, our second objective is to minimize the long-run average cost over an infinite horizon, where these costs arise from usage and repair costs. Due to the complexity and novelty of the resulting model, we focus on single-subsystem problems in this paper, as a first step towards application to more general models. %whereas most RAP literature considers multiple subsystems.

\subsection{Redundancy Allocation Problem}
\label{paper1:intro:rap}
In the broadest sense, the RAP is an NP-hard problem \citep{CHERN1992309} in which some of the decision variables represent the installation of redundant components in some system, with knapsack-style constraints on the selected components (cost, weight, volume etc), and where the objective (or one of the objectives) is to maximize reliability, or equivalently to minimize the probability of system failure. This is similar to but distinct from reliability allocation, where the structure (number and arrangement of components) is fixed, but each component has a range of alternatives from which to choose \citep{majety1999reliability}. For any component in the original design which is a candidate for further redundancy, we say that it may be replaced by a \emph{subsystem} of parallel components. The RAP has seen many variations throughout the literature. It has seen single-objective \citep{CHERN1992309} and multi-objective \citep{ZOULFAGHARI2014biObjRepairable, KAYEDPOUR2017repairableMarkovFiniteHorizon, tavana2018multistate, LINS2011manyCost} variations, where secondary objectives often represent the installation costs of the components. The types of systems considered range between series-parallel systems \citep{CHERN1992309} which only require one component in each subsystem to be working to fully function, $k$-out-of-$n$ systems which require some $k$ copies of each component to be operational simultaneously \citep{KESHAVARZGHORABAEE2015kOutOfN}, and complex systems with complicated networked relationships between components \citep{park2020complexInforms}. Redundancy strategies may follow a hot, warm, or cold standby strategy, or a mixture of the three. { Under hot standby, all components are considered to be \emph{active} and therefore are subject to the same time-to-failure distributions as if they were being used. Under warm standby, all components are considered to be active, but with a lower rate of failure if they are not being used. Finally, under cold standby only the component currently being used is active, and the rates of failure for all other components are zero. When the active component fails, the system successfully switches to an idle component with some probability, where this is called \emph{perfect switching} if the probability is always one.} System failure occurs when both the active component fails and the switching fails. A mixed strategy keeps some components always active and others on standby \citep{REIHANEH20221112}. Components may be modeled as all non-repairable \citep{kulturelkonak2003nonRep}, all repairable \citep{KAYEDPOUR2017repairableMarkovFiniteHorizon}, or a mix of the two \citep{ZOULFAGHARI2014biObjRepairable}. The RAP may consider a finite \emph{mission-time} \citep{kulturelkonak2003nonRep, KAYEDPOUR2017repairableMarkovFiniteHorizon} or long-run average reliability \citep{CHERN1992309}. On top of the standard allocation decisions, other decisions to be made may include the allocation of repair teams to each subsystem \citep{KAYEDPOUR2017repairableMarkovFiniteHorizon, LINS2011manyCost}, or the order in which component switching occurs \citep{REIHANEH20221112}. Components in the RAP may be binary-state, meaning that they are either fully working or fully damaged with no in-between, or they can be multi-state, allowing for a finer-grain representation of the state of repair of any given component \citep{tavana2018multistate}. The majority of the {\color{black}more recent} literature uses metaheuristics such as genetic algorithms to optimize their respective models; however work has been done using Dantzig-Wolfe decomposition to create column-generation-based heuristics \citep{ZIA2010} or exact solution algorithms via branch-and-price \citep{REIHANEH20221112}.

The specific type of RAP that we are interested in is a bi-objective repairable model for a series-parallel system. \cite{ZOULFAGHARI2014biObjRepairable} proposes a bi-objective model to balance between cost and reliability, using a mixture of repairable and non-repairable components. Repairable components are actively maintained but the cost of performing these repairs is not considered. They optimize their model using a genetic algorithm. \cite{KAYEDPOUR2017repairableMarkovFiniteHorizon} proposes a bi-objective model that balances the availability of the system by a certain mission time against the combined cost of the components and of the repair workers allocated to each subsystem, where the number of repair workers allocated to each subsystem is an additional decision variable. Repairs are carried out actively (i.e. never delayed when repair workers are available), but a repair worker can only repair one component at a time. Component failure and repair times are assumed to be exponential. A \emph{continuous-time Markov chain} (CTMC) model is used to model the stochastic dynamics of this system, and to evaluate its expected availability by the mission time. The model is optimized using NSGA-II, a genetic algorithm specific to multi-objective optimization. \cite{tavana2018multistate} also considers a bi-objective model balancing availability by a certain mission time and installation cost. This is a multi-state model where every component goes through multiple states of degradation before failure. The dynamics of each component are again modeled using a CTMC. There is no limitation on the number of ongoing repairs, nor are there any associated costs. This model is also optimized using NSGA-II.  Finally, \cite{LINS2011manyCost} considers a bi-objective model balancing reliability and cost; however, the cost term is more complicated. In addition to installation cost (called acquisition cost in their work), they also consider operational cost, corrective maintenance cost, repair team cost, and penalty for system failure. Here, operational costs are incurred per unit time from every healthy component under a warm standby model. Corrective maintenance costs are a lump sum paid every time a component is repaired, repair teams are paid per unit time, and there is a penalty incurred per unit time of system failure. Components are always repaired as soon as possible. The model is optimized using a multi-objective genetic algorithm alongside discrete event simulation.

{ Previous literature has explored the integration of the RAP for series-parallel systems with future maintenance decisions as a nonlinear two-stage stochastic programming problem with recourse. \cite{bei2017designAndMaintenanceRAP} introduces such a two-stage problem where first-stage redundancy allocation decisions (allowing for different types of component per subsystem) are taken with uncertainty with respect to \emph{future usage stresses}, which affect the assumed Weibull-distributed lifetimes of the installed components. After the random future usage stress has been realized, second-stage decisions then determine the \emph{preventative maintenance interval} for each subsystem, which is the maximum length of time until a perfect preventative maintenance action is simultaneously and instantaneously applied to all components in a subsystem. If the subsystem fails before this point, an expensive emergency perfect repair must be performed. These perfect maintenance actions allow the stochastic dynamics of the system to be modelled as a renewal process, allowing for the {\color{black}straightforward} calculation of expected maintenance costs per unit time, which in turn are used in the objective function (alongside the installation cost of the components when adjusted by a capital inflation factor). Due to the nonlinear mixed-integer nature of the problem, the authors explore black-box methods (namely a commercial solver called NOMAD) to solve instances of their model, as well as integer rounding based on a continuous relaxation.  \cite{zhu2018designAndMaintenanceSequential} extend this by replacing the one-off usage stresses with stress \emph{functions} that increase over time. The second-stage decisions determine (for each subsystem) the number of times minimal preventative repair will be performed, and how long each of the preventative maintenance intervals will be. First-stage decisions are simplified as there is only one component type per subsystem. To optimize their model, the authors use a decomposition approach where NOMAD is used to explore first-stage decisions, and some theoretical work allows the optimal second-stage decision for any given first stage to be solved by MATLAB's ``fsolve'' function. \cite{BEI2019designAndMaintenanceRiskAverse} return to the simpler one-off stress formulation, and consider a more risk-averse problem formulation that introduces the conditional value-at-risk (CVaR) risk measure. Their solution methodology is similar, with NOMAD used to explore first-stage decisions and a combination of theoretical work and MATLAB's fsolve to optimize second-stage decisions. To the best of our knowledge, these papers are the only previous studies that integrate redundancy allocation with maintenance, and there is no literature on integrating redundancy allocation with dynamic maintenance. This is because, in general, there is very little literature on the integration of MDPs and first-stage design decisions into unified optimization models, let alone any literature that develops suitable, scalable solution methodologies. 

Two recent, thorough literature reviews on RAP by \cite{devi2023RAPlitReview} and \cite{guanCoit2025RAPlitReview} demonstrate that the problem is well established in the literature and very well studied. However, \cite{devi2023RAPlitReview} notes that the majority of the literature focuses on {\color{black}simple numeric examples not motivated by any real system}, and there is a need to look towards real-world applications. Nevertheless, real-world applications have been considered by a number of authors, especially in the energy sector. \cite{kuo2001optimalReliabilityDesign} explores applications such as redundant electrical components in the control system of an aircraft, and redundant piping in a gas pipeline. \cite{ANAND1994pwrCooling} considers redundancy allocation for a pressurized water reactor cooling loop. \cite{SULE2019SafetyCriticalEnergy} considers redundancy allocation with multiple objectives for safety-critical energy systems, directly considering the Sinopec Luoyang Petrochemical Plant in a case study. They consider balancing the number of redundant components against the overall risk of system failure. \cite{ling2021energySystems} explores redundancy allocation for energy systems, with attention drawn to a case study considering biomass power plants. We note that these types of energy systems are repairable, and these repairs can be carried out whilst the functional components continue to operate. As such, the dynamic maintenance strategies that we explore in this work are applicable to these settings.
}
\subsection{Markov Decision Processes}
While some of the RAP literature includes Markov processes to model stochastic dynamics \citep{KAYEDPOUR2017repairableMarkovFiniteHorizon, tavana2018multistate}, to our knowledge none of the RAP literature makes use of MDPs. These are a framework used for modeling sequential decision-making problems under uncertainty, first introduced by \cite{bellman1957markovian}. In essence, MDPs are Markov chains that allow for decisions to be made at each state. For small to moderate problems that can be solved exactly, they are typically solved using one of three techniques: value iteration \citep{bellman1957markovian}, policy iteration \citep{howard1960dynamic}, or linear programming \citep{epenoux1963mdpLP}. A comprehensive account of the theory of discrete-time MDPs is given by \cite{Puterman1994MarkovDP}. Larger problems are solved using approximate methods such as reinforcement learning \citep{sutton2018reinforcement} or approximate dynamic programming \citep{powell2011adp}. An extension of the MDP is the continuous-time MDP, or CTMDP, a thorough exposition of which is given by \cite{xianping2009continuous}. Common applications of MDPs include inventory management \citep{bertsekas2012dynamic}, queuing control \citep{adusumilli2010queue}, and medical applications such as patient admission to hospitals \citep{nunes2009markovpatient} and overflow decisions \citep{dai2019inpatient}. 

MDPs have also been used to model maintenance problems, as small decisions (whether or not to repair, the extent of the repair, etc) can be made sequentially as the system degrades and is repaired over time. Some of this literature includes the condition-based preventative maintenance of systems, such as the work of \cite{chen2005conditionbased} or \cite{amari2006conditionbased}. More recent work considers \emph{partially-observable MDPs} (POMDPs) for joint inspection and maintenance optimization \citep{GUO2022pomdp} or maintenance planning alone \citep{DEEP2023pomdp}. More relevant to the RAP is the application of MDPs to the maintenance of $k$-out-of-$n$ redundancy systems by \cite{flynn2004heuristic}. Their work considers repairs with known deterministic lump-sum costs, a penalty for failure, and instantaneous repairs. The problem is modeled as a discrete-time MDP. Under these simplifying assumptions, they show that the optimal policy is a \emph{Critical Component Policy} under reasonably light assumptions, meaning that every component is either always repaired straight away or never repaired at all, resulting in a very simple policy. However, we find that this does not hold for more the model proposed in this work. Similarly, \cite{ozekici1988replacement} considers periodic age-based replacement of components in multi-component systems, and models this as an MDP.

\subsection{MDP Design Problem}
Rarely seen is the fusion of design decisions, such as those found in the RAP, with the dynamic operation of the system resulting from those design decisions, as modeled by an MDP. The main framework for this is the little-studied \emph{MDP Design} problem. This is a two-stage problem in which the first stage decides upon the system design and the second stage optimizes the resulting system by modeling it as an MDP. To our knowledge, the earliest example of this framework is due to \cite{dimitrov2009mdpDesign}, where for every state-action pair there is a vector cost associated with that state-action pair being available in the second-stage MDP. These vector costs must satisfy knapsack-style constraints. They also extend this model to include uncertainty in whether or not a design decision (i.e. allowing for a certain state-action pair) will truly be realized with that state-action pair being available. \cite{DIMITROV2013mdpDesignMalaria} apply this model to the problem of malaria intervention. An MDP Design framework was independently reformulated by \cite{BROWN2024mdpDesignLetters}. In addition to formulating a \emph{bilevel} problem where design decisions are taken in the first stage before operating the resulting MDP, this work also considers the possibility of uncertainty in the dynamics and costs of the associated MDP. For example, if a design decision reflects the installation of some component in a system, there may be uncertainty in the reliability or costs associated with that component in the second stage. As such, they present a scenario-based approach to deal with this uncertainty. Suggested applications for this model include reliability, inventory management, and queuing design and control. They apply an off-the-shelf bi-level solver to a range of small randomly generated problems that do not reflect any specific application.

To the best of our knowledge, the three papers cited above form the entirety of the MDP Design literature as it stands, leaving much room for future development. We believe this represents an exciting opportunity, as MDP Design problems can be used to find improvements to the designs of real-world systems that will subsequently require sequential decisions to be made under uncertainty. In more general terms, this approach can allow for the integration of strategic and operational decisions in many different contexts. Relevant examples could include (i) an infrastructure system that must first be designed and deployed, then dynamically maintained as it experiences wear-and-tear over time; (ii) a service system (e.g., a hospital or transport hub) consisting of facilities that must first be designed and built before having their queues or waiting lists managed on a day-to-day basis; (iii) a dynamic variant of a location routing problem \citep{drexl2015locationRoutingsurvey} in which depots must be strategically located before vehicle routes are managed on a daily basis \citep{pillac2013dvrpSurvey}. Further development of the general area of MDP Design problems will help to address these sorts of important real-world problems.

% \subsection{Matheuristics}

% Matheuristics are approaches which combine heuristic methods such as problem relaxations or metaheuristics with some exact optimization methods \citep{maniezzo2021matheuristics}. A review of matheuristics with application to routing problems is provided by \cite{archetti2014matheuristicRouting}, who identify three core types of matheuristic: decomposition-based, where subproblems are generated which can be solved exactly; improvement-based, which uses mathematical programming to improve a candidate solution found by another heuristic; and column-generation-based, which are methods used for problems with the specific structure necessary for column generation to be applicable. 

% Matheuristics have seen limited application to multi-objective optimization in the literature when compared to standard metaheuristics. Some examples include \cite{cantu2021matheurBiObjMeta, vahedi2023matheurBiObjMeta}, which both incorporate mathematical programming into a population-based metaheuristic framework, and \cite{tautenhain2019matheurNonMeta}, which doesn't rely upon a metaheuristic framework. In particular, population-based matheuristics seem to be rare in general, and we are not aware of any population-based matheuristics that don't rely on a metaheuristic framework.

\subsection{Problem Overview}
\begin{figure}[tbp]
    \centering
    \includegraphics[width=\textwidth]{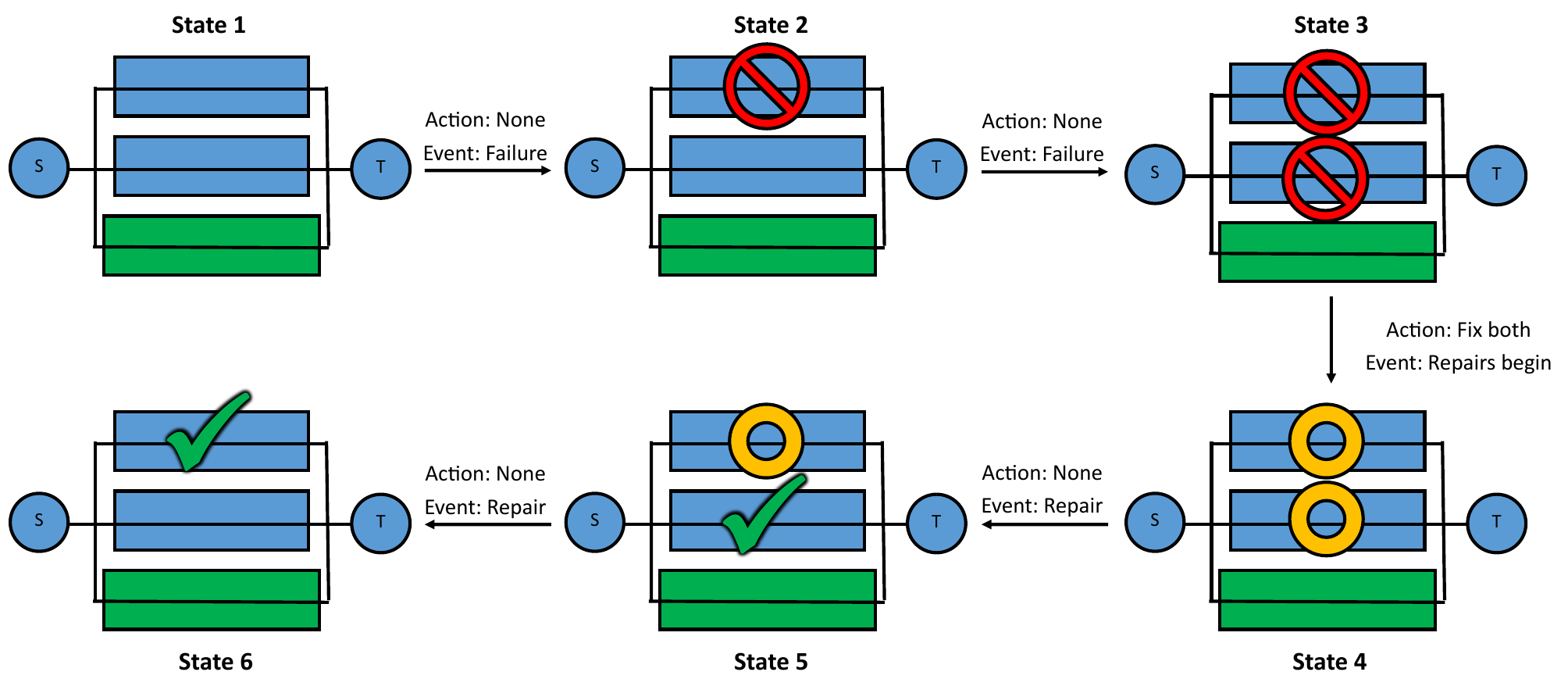}
    \caption{Sample trajectory through the dynamic maintenance problem}
    \label{fig:sampleTraj}
\end{figure}
We present a high-level overview of our problem, aided by a simple example (the detailed problem formulation follows in \autoref{paper1:formulation}). Consider a set of candidate components to be installed in parallel in a single subsystem. Each candidate component incurs usage cost at a constant rate when it is the component with the cheapest usage cost out of all healthy components. In other words, when the component we are currently using fails, we switch to a healthy component with the cheapest usage cost. Also, when a cheaper component comes back online, we switch back to using that component. We assume a {hot} standby strategy with perfect and instantaneous switching, with no or negligible idling costs. {That is, usage costs are only incurred from the one component being used, and no costs are incurred from the idling-yet-active components.} Each component also incurs repair costs at a constant rate while being repaired. These two costs contribute to the long-run operational costs, forming one of our objectives. {The explicit modelling of usage costs in this way is a novel contribution to the RAP literature which allows us to, for example, model a type of ``backup'' component that is more reliable yet less power-efficient or less efficient in the processing of its inputs, making it more costly to run.}  Each component also has a failure rate and a repair rate corresponding to their respective failure and repair completion events. Finally, components also have attributes such as installation costs and weights, which are subject to constraints such as installation budgets and maximum total weights, respectively.

Our first-stage decisions are the quantities of each type of component to be installed. Our second stage decisions are dynamic maintenance decisions on which currently damaged components to start repairing, given the current overall state of the subsystem. \autoref{fig:sampleTraj} illustrates a sample trajectory through this model, considering a subsystem with three components, two of which are the same type, as indicated by the colors (available online). We start in state 1 with all components healthy and no maintenance decisions to be made, so we wait until the first failure event, which moves us to state 2. In this state, we can choose between two actions: fix this component, or do not fix this component. This is the type of situation where we may argue that it is unnecessary to repair this component straight away, since we have two further levels of redundancy. We assume that no repair action is taken, so we now wait until the next failure, taking us to state 3. Now, only one component remains healthy, so the next failure event would cause system failure. There are four possible actions, as each damaged component may or may not be put into repair. In this sample trajectory we decide to repair both, leading us to state 4. Now there are no further actions to take, and we simply wait for the next event, which is the completion of repairs on the second component in state 5. From state 5, there are still no actions to be taken, so we simply wait for the next event, which is the completion of repairs on the first component, leading us to state 6. State 6 is the same as state 1, with all components healthy. States 2 and 3 are most illustrative of the type of decisions we wish to make in the dynamic maintenance problem, where we want to decide which repairs to commence for any combination of components being healthy, damaged, or currently being repaired.

{\color{black}
As a motivating example, consider the case of small-scale power generation. Suppose the owner of some system wishes to generate their own energy to power the system. Available to them are a range of generators all of which generate enough energy for the system, but which differ in terms of costs due to fuel consumption (either due to fuel costs or efficiency), maintenance costs, time-to-failure, repair time, installation cost, and so on. We assume that both long-run operational costs and reliability are important to the system owner, and that they have limits in terms of their upfront installation budget and physical constraints such as weight and size, but that these constraints are relaxed enough to purchase and install multiple redundant generators for the sake of backup. We assume that these backup generators are kept in active (or hot) standby, i.e. they are kept idling (using a negligible amount of fuel) so that when one generator fails, an available generator with the cheapest usage costs can quickly and reliably take over. As such, the power generation fails only if all generators are simultaneously non-operational. Due to the presence of redundant generators, the owner is also interested in avoiding unnecessary maintenance costs, especially as they hope to operate this system well into the future, and maintenance will be needed repeatedly over a long period of time. On the other hand, the owner does not want to sacrifice too much in terms of reliability in the pursuit of saving on maintenance costs. As such, the problem is to simultaneously design a power system with redundancy and determine a maintenance strategy for it, in such a way as to strike a balance between two different objectives.
}

\subsection{Our Contributions}

%Over the course of our investigation, we find that our MDP design problem grows rapidly in complexity for moderately sized instances, hence we propose a bespoke heuristic to provide a set of high-quality solutions in a fraction of the time taken by an exact \emph{mixed-integer linear programming} (MILP) solver. We also find that this heuristic more easily finds a better-populated set of solutions due to a lower sensitivity to certain parameters.

{In this paper we make the following contributions:
\begin{itemize}
    \item In \autoref{paper1:formulation}, we introduce a novel formulation of an RAP for parallel systems as a bi-objective MDP Design problem, allowing for dynamic decisions to be made regarding maintenance. This is a new contribution to the RAP literature, and is also the first bi-objective MDP Design problem in the literature more generally. It is the first MDP Design formulation that allows for the feasibility of state-action pairs in the second-stage MDP to be affected by multiple first-stage design decisions, and for each first-stage design decision to impact the feasibility of multiple state-action pairs. This extends the formulation of \cite{dimitrov2009mdpDesign}, which only allowed for a one-to-one correspondence between state-action pairs and design variables.
    \item In \autoref{paper1:formulation}, we also propose the novel inclusion of usage costs incurred only when a component is being used, bearing some similarity to the way that costs are incurred in unreliable facility location problems \citep{lim2010facility, lim2012facility} and unreliable $p$-median problems \citep{OHANLEY201363}. Alongside this, we consider component-wise (not repair-team-wise) repair costs.
    \item In \autoref{paper1:methodology}, we introduce a heuristic for approximately solving bi-objective MDP Design problems, which is important due to the intractability of MDP Design problems of the form presented in \autoref{paper1:formulation}.  
    \item In \autoref{paper1:study}, we verify the performance of our heuristic methodology against a standard exact MILP solver, and show that our heuristic performs well in terms of the quality and completeness of the Pareto fronts produced, whilst being much more computationally efficient. We also investigate the effects of varying parameters unique to our model, namely usage costs, repair costs, and the rates at which failure and repair events occur. This includes showing the benefits of using dynamic maintenance policies, which to our knowledge have never been considered in the RAP literature. Specifically, dynamic maintenance policies produce a better populated Pareto front, allowing better flexibility to decision-makers, and can also dominate solutions that are Pareto-optimal under the assumption of a non-dynamic policy.
\end{itemize}}
 % Additionally, we  The resulting model suffers greatly from the curse of dimensionality, with the number of decision variables growing exponentially in the number of component types. To counteract this,  Our computational study demonstrates that the heuristic compares very favorably to the exact method, with the exact method seldom finding a solution not found by the heuristic. Our study then investigates the impact of varying certain parameters of our model, namely usage costs, repair costs, and event rates. Throughout the computational study, we demonstrate the strengths of considering dynamic maintenance policies, showing that they provide a better-populated Pareto front of solutions, and can even produce solutions that are strictly better in both objectives than solutions without dynamic policies.

\autoref{paper1:tab:litReview} highlights our contributions against the relevant prior literature. The structure of the paper is as follows: in \autoref{paper1:formulation} we formulate the problem by formalizing the model shown in \autoref{fig:sampleTraj} as a bi-objective CTMDP model, and then integrating it into a design problem. \autoref{paper1:methodology} builds towards and explains our heuristic methodology for finding an approximate Pareto front of solutions to this problem. \autoref{paper1:study} then undertakes a computational study comparing our heuristic to the exact method, as well as exploring the effect of varying the parameters in our model that are less often seen in the RAP literature, such as usage costs, repair costs and event rates. \autoref{paper1:conclusion} concludes, and proofs are deferred to the supplementary material.

\begin{landscape}
\begin{table}[]
\color{black}
\begin{tabular}{l|cc|ccc|cc|l}
\cline{2-8}
                                             & \multicolumn{2}{c|}{\begin{tabular}[c]{@{}c@{}}Objective \\ Type\end{tabular}} & \multicolumn{3}{c|}{Maintenance Decisions}                                                                                                                                                                 & \multicolumn{2}{c|}{Operational Costs}                                                                                                                                      &                                                                                             \\ \hline
\multicolumn{1}{|l|}{Paper}                  & \multicolumn{1}{l|}{Single}                 & \multicolumn{1}{l|}{Bi}          & \multicolumn{1}{l|}{\begin{tabular}[c]{@{}l@{}}Repair \\ Team \\ Allocation\end{tabular}} & \multicolumn{1}{l|}{\begin{tabular}[c]{@{}l@{}}Static \\ Schedule\end{tabular}} & \multicolumn{1}{l|}{Dynamic} & \multicolumn{1}{l|}{\begin{tabular}[c]{@{}l@{}}Component-wise\\ Maintenance\end{tabular}} & \multicolumn{1}{l|}{\begin{tabular}[c]{@{}l@{}}Component \\ Usage\end{tabular}} & \multicolumn{1}{l|}{\begin{tabular}[c]{@{}l@{}}Solution \\ Methodology\end{tabular}}        \\ \hline
\multicolumn{1}{|l|}{\cite{LINS2011manyCost}} & \multicolumn{1}{c|}{}                       & $\checkmark$                     & \multicolumn{1}{c|}{$\checkmark$}                                                         & \multicolumn{1}{c|}{}                                                           &                              & \multicolumn{1}{c|}{$\checkmark$}                                                         & $\checkmark$                                                                    & \multicolumn{1}{l|}{\begin{tabular}[c]{@{}l@{}}Genetic\\ Algorithm with\\ DES\end{tabular}} \\ \hline
\multicolumn{1}{|l|}{\cite{ZOULFAGHARI2014biObjRepairable}}       & \multicolumn{1}{c|}{}                       & $\checkmark$                     & \multicolumn{1}{c|}{}                                                                     & \multicolumn{1}{c|}{}                                                           &                              & \multicolumn{1}{c|}{}                                                                     &                                                                                 & \multicolumn{1}{l|}{\begin{tabular}[c]{@{}l@{}}Genetic \\ Algorithm\end{tabular}}           \\ \hline
\multicolumn{1}{|l|}{\cite{KAYEDPOUR2017repairableMarkovFiniteHorizon}}         & \multicolumn{1}{c|}{}                       & $\checkmark$                     & \multicolumn{1}{c|}{$\checkmark$}                                                         & \multicolumn{1}{c|}{}                                                           &                              & \multicolumn{1}{c|}{}                                                                     &                                                                                 & \multicolumn{1}{l|}{NSGA-II}                                                                \\ \hline
\multicolumn{1}{|l|}{\cite{bei2017designAndMaintenanceRAP}}               & \multicolumn{1}{c|}{$\checkmark$}           &                                  & \multicolumn{1}{c|}{}                                                                     & \multicolumn{1}{c|}{$\checkmark$}                                               &                              & \multicolumn{1}{c|}{$\checkmark$}                                                         &                                                                                 & \multicolumn{1}{l|}{\begin{tabular}[c]{@{}l@{}}Black-Box;\\ Integer Rounding\end{tabular}}  \\ \hline
\multicolumn{1}{|l|}{\cite{tavana2018multistate}}            & \multicolumn{1}{c|}{}                       & $\checkmark$                     & \multicolumn{1}{c|}{}                                                                     & \multicolumn{1}{c|}{}                                                           &                              & \multicolumn{1}{c|}{}                                                                     &                                                                                 & \multicolumn{1}{l|}{NSGA-II}                                                                \\ \hline
\multicolumn{1}{|l|}{\cite{zhu2018designAndMaintenanceSequential}}               & \multicolumn{1}{c|}{$\checkmark$}           &                                  & \multicolumn{1}{c|}{}                                                                     & \multicolumn{1}{c|}{$\checkmark$}                                               &                              & \multicolumn{1}{c|}{$\checkmark$}                                                         &                                                                                 & \multicolumn{1}{l|}{\begin{tabular}[c]{@{}l@{}}Decomposition\\ with black-box\end{tabular}} \\ \hline
\multicolumn{1}{|l|}{\cite{BEI2019designAndMaintenanceRiskAverse}}               & \multicolumn{1}{c|}{$\checkmark$}           &                                  & \multicolumn{1}{c|}{}                                                                     & \multicolumn{1}{c|}{$\checkmark$}                                               &                              & \multicolumn{1}{c|}{$\checkmark$}                                                         &                                                                                 & \multicolumn{1}{l|}{\begin{tabular}[c]{@{}l@{}}Decomposition\\ with black-box\end{tabular}} \\ \hline
\multicolumn{1}{|l|}{This work}              & \multicolumn{1}{c|}{}                       & $\checkmark$                     & \multicolumn{1}{c|}{}                                                                     & \multicolumn{1}{c|}{}                                                           & $\checkmark$                 & \multicolumn{1}{c|}{$\checkmark$}                                                         & $\checkmark$                                                                    & \multicolumn{1}{l|}{\begin{tabular}[c]{@{}l@{}}Novel\\ heuristic\end{tabular}}              \\ \hline
\end{tabular}
\caption{Comparison of prior works and this work}
\label{paper1:tab:litReview}
\end{table}
\end{landscape}

%%%%%%%%%%%%%%%%%%%%%%%%%
\section{Problem Formulation}
\label{paper1:formulation}
In this section we formalize our problem by combining bi-objective optimization, continuous-time Markov decision processes and the MDP Design framework to provide a novel bi-objective CTMDP Design problem. {\autoref{paper1:formulation:dynamic} motivates our choice of using a bi-objective CTMDP to model this problem, describes the CTMDP in generality, then uses this framework to model our dynamic maintenance problem.} \autoref{paper1:formulation:int} then extends this model to include design variables subject to knapsack-style constraints, and links the permitted states of the CTMDP to the design decisions made.

\subsection{Dynamic Maintenance Problem}
\label{paper1:formulation:dynamic}

We first focus on how to model our \emph{Bi-Objective Dynamic Maintenance Problem} (BO-DMP) as a CTMDP, assuming a fixed design initially. {To begin, we discuss why bi-objectivity and dynamic maintenance are important to consider. In a single-objective formulation where the goal is just to maximize system reliability, the notion of dynamic maintenance---and the inclusion of operational costs---does not make sense, as the goal is only to improve system reliability, regardless of how frequent or expensive maintenance actions are. However, the solution to such a problem may be \emph{overly} reliable if the constraints on the problem are quite relaxed, leading to an optimal system design that is very reliable, but costly to maintain and operate. A bi-objective formulation allows us to find a set of solutions that balance between operational costs and system reliability. More specifically, a Pareto front of solutions allows a decision maker to investigate how reasonable sacrifices in reliability can lead to a system with reduced long-run operational costs. This consideration of operational costs is also where the idea of using dynamic maintenance arises, as it leads to the question of whether repairing components immediately is always worthwhile, or whether it may sometimes be better to wait until the system is in a worse state overall. In the latter case, delays in repairs and a less aggressive maintenance strategy may enable accumulated cost savings over the lifespan of the system. We now follow with the mathematical modelling of our problem.}  

We start by defining the CTMDP in general and introduce the relevant notation, and then define the states, actions, transitional behavior and costs of the BO-DMP. This is followed by a linear programming (LP) model which can easily be solved to optimality for realistic, moderately-sized instances.

Formally, an CTMDP is a 4-tuple:$$\mathcal{C} = \left\langle \mathcal{S}, (\mathcal{A}(s))_{s\in\mathcal{S}}, \left(q(s,a,s')\right)_{s,a,s'\in \mathcal{S},\mathcal{A}(s),\mathcal{S}}, \left(c(s,a)\right)_{s,a\in \mathcal{S},\mathcal{A}(s)}  \right\rangle,$$
where $\mathcal{S}$ is the state space, $\mathcal{A}(s)$ is the feasible action space for state $s\in\mathcal S$, $q(s,a,s')$ gives the transition rate from state $s$ to state $s'$ when taking action $a$, and $c(s,a)$ is the vector cost rate when taking action $a$ in state $s$. This definition is similar to that of \cite{Puterman1994MarkovDP}.

We restrict attention to deterministic stationary (DS) policies $\mu:\mathcal{S}\to\mathcal{A}$, giving maps from states to actions. For a fixed policy $\mu$, we can define the $T$-horizon cost-to-go or value function, where $T$ is some positive constant, $J^\mu_T : \mathcal{S}\to\RR$:
$$J_T^\mu(s) = \EE_\mu\left\{\left.\int_0^T c\left(\,S(t),\,\mu(S(t))\,\right) \text{d}t\,\right|\,S(0) = s\right \},$$
where $S(t)$ is a random variable representing the state of the process at time $t$. This represents the expected cost accumulated over $T$ units of time when starting the process in some state $s$ and controlling the process using policy $\mu$. Using this, we can then define the long-run average cost $g^\mu$ as follows:
$$g^\mu(s) = \lim_{T\to\infty} \frac{1}{T}\,J_T^\mu(s).$$

In the case of so-called \emph{unichain} MDPs, the long-run average cost $g^\mu(s)$ under any policy $\mu$ is independent of the starting state $s$. Additionally, under the more relaxed assumption that the MDP is \emph{weakly communicating}, the long-run cost $g^*$ under the optimal policy is independent of starting state \citep{Puterman1994MarkovDP}. We restrict our attention to such problems and later justify this with \autoref{theorem:communicat}. In the single objective case, we wish to find a policy $\mu^*$ that minimizes $g^{\mu^*}$. In the multi-objective case, we wish to find non-dominated policies $\mu$ where there does not exist a policy $\mu'$ such that $g^{\mu'} \leq g^{\mu}$ component-wise and $g^{\mu'}_k < g^{\mu}_k$ for some entry (objective) $k$. 

% {\color{Orange} The relative cost-to-go under a policy $\mu$ when the initial state is $s$ is defined as follows:
% $$h_\mu(s) = \lim_{T\to\infty}\left\{J_\mu^T(s) - g_\mu T\right\}.$$ 
% In the single-objective case, the policy $\mu^*$ that minimises $g_{\mu^*}$ solves the Bellman Equations:
% $$h^*(s) + g^* t(s,\mu^*(s)) = \min_{\mu\in\mathcal{M}}\left\{c(s,\mu(s))t(s,\mu(s)) + \EE_{s'\sim p(s,\mu(s),\cdot)}h_\mu(s')\right\}, $$
% for all $s\in\mathcal{S}$, where $h^* = h_{\mu^*}$, $g^* = g_{\mu^*}$, and $t(s,a)$ is the expected sojourn time from state $s$ under action $a$.}

For the BO-DMP, we define the state $\mathbf{s}\in\ZZ^{N\times2}$ of the CTMDP as an $(N\times 2)$-dimensional matrix representing the number of components of each type in either a repairing or damaged condition, where $N$ is the number of component types. Each row $s_i$ for $s$ is of the form
$$s_{i} = [s_{i1},s_{i2}],$$
where $s_{i1}$ is the number of components of type $i$ currently being repaired, and $s_{i2}$ is the number of components of type $i$ currently damaged and not being repaired. We assume that the total number of copies of component $i$ is known and given as $M_i$, so the number of healthy copies of component $i$ in any given state is $M_i - s_{i1} - s_{i2}$, which we denote by $s_{i0}$ for notational convenience. It is simple to see that $s_i$ can take on $M_i(M_i + 1)/2$ possible values, as we require $s_{i1} + s_{i2} \leq M_i$.
As such, the size of the state space is $|\mathcal{S}| = \prod_{i=1}^N M_i(M_i + 1)/2$, so our state space grows quadratically in the number of copies of any given component, and exponentially in the number of types of a component with at least 1 component installed.

The actions $\mathbf{a}\in\ZZ^N$ are vectors whose components $a_{i}$ are decision variables determining how many components of type $i$ to put into repair. For each state, we define the feasible action set $\mathcal{A}(\bm{s}) = \{\bm{a}\in\ZZ^N:0\leq a_i \leq s_{i2}\}$, as we can only repair components up to the number of damaged components of a specific type. For ease of notation, we write $\bm{s}\oplus\bm{a}\in\ZZ^{N\times2}$\ to denote the post-decision state (that is, the state immediately after the action is taken, but before any new events are observed, following \cite{powell2011adp}), which satisfies $[\bm{s}\oplus\bm{a}]_i = [s_{i1} + a_i, s_{i2} - a_i]$. We assume that these actions are \emph{impulsive}, meaning the effect of taking an action is instantaneous. For clarity, this does not mean that the repair itself is instantaneous, but that the process of starting a repair is instantaneous. We achieve this by enforcing the following relations when $\bm{a}\neq\bm{0}$:
\begin{align*}
    q(\bm{s},\bm{a}, \bm{s'}) &= q(\bm{s\oplus a, 0, s'}), \text{ for } \bm{s' \neq s, s' \neq s\oplus a},\\
    q(\bm{s,a,s\oplus a}) &= 0,\\
    q(\bm{s,a,s}) &= -\sum_{\bm{s'\neq s}} q(\bm{s,a,s'}),\\
    c(\bm{s}, \bm{a}) &= c(\bm{s\oplus a, 0}).
\end{align*}
This treats the state-action pair in such a way that it emulates the behavior of being in the post-action state under the zero action. Due to the separate treatment of non-zero actions, we can more succinctly write $q(\bm{s},\bm{s'}) := q(\bm{s},\bm{0},\bm{s'})$ and $c(\bm{s}) := c(\bm{s}, \bm{0})$.

Outside of the immediate effects of actions, we also have two types of event which cause a state transition:
\begin{itemize}
    \item \textbf{Failure} - The failure of a component, which occurs at a constant failure rate whenever the component is active. The interarrival times of such events are assumed to be independent and exponentially distributed with rate $\alpha_{i} > 0$ for component $i$. 
    \item \textbf{Successful Repair} - The event that changes the condition of a repairing component so that it becomes a healthy component. We assume that ongoing repairs on each component are carried out simultaneously and independent of each other by multiple repair workers, and therefore repair completion times are independent and exponentially distributed with rate $\tau_i > 0$ for component $i$.
\end{itemize}

We can then define the transition rates $q$ under the zero action as follows:

$$q(\bm{s}, \bm{s'}) = \left\{
\begin{array}{cc}
     s_{i0}\alpha_i, &\bm{s'} = \bm{s} + \bm{e}_{i2},\\
     s_{i1}\tau_i , &\bm{s'} = \bm{s} - \bm{e}_{i1},\\
     -\sum_{i=1}^N\left(s_{i1}(\tau_i + \alpha_i) + s_{i0}\alpha_i\right) & \bm{s'} = \bm{s},\\
     0, & \text{otherwise,}
\end{array}\right.$$
where $\bm{e}_{ij}$ is a matrix with dimensions the same as the state matrix with zeros everywhere except index $ij$, where it is 1. The first line gives the rate of a healthy component of type $i$ failing and becoming damaged, and the second line gives the rate of a repairing component of type $i$ completing its repair and becoming healthy. The third line then gives the overall negative sojourn rate of the state $\bm{s}$, and the fourth line shows that all other transitions have rate 0.

The cost rates of the BO-DMP should incorporate utilization costs, repair costs, and system availability.  The cost rates are therefore made up of three components:
\begin{itemize}
    \item \textbf{Component Utilization} - Each component $i$ has an associated cost $c_i$ per unit time to use the component, incurred only when it is the cheapest one available. The system is utilized optimally by always using the cheapest healthy component {(recall that using a component does not increase its failure rate, so using the cheapest component does not increase the chance of needing to pay repair costs). Therefore we always incur usage costs $c^u(\bm{s})=\min_{i:s_{i0}>0}\left\{c_i\right\}$.}
    
    \item \textbf{System Failure} - If the system contains no healthy components, we incur a failure cost at rate 1 (i.e. this serves as an indicator variable for system failure).
    
    \item \textbf{Repair Costs} - If component $i$ is currently under repair then we incur a cost at rate $r_i$ per unit time. We denote this by $c^r(\bm{s})=\sum_{i=1}^Nr_is_{i1}$.
\end{itemize}

Whilst utilization and repair costs can simply be added together due to being in comparable units, the availability of the system may not have a direct monetary value. As such, we treat system availability as a separate cost/objective, giving two separate objectives. Symbolically, we have an operational cost rate
$c^o(\mathbf{s}) = c^u(\bm{s}) + c^r(\bm{s})$, and a failure cost rate
$c^f(\bm{s}) = \prod_{i=1}^N \II\{s_{i0} = 0\}$.

We can model the BO-CTMDP using linear programming formulations, which when converted to a single objective problem can be solved using any LP solver. However, the LP formulation depends on certain properties of the CTMDP. The simplest LP model of an MDP is most applicable if the MDP is \emph{unichain}, meaning that for any DS policy over the MDP, the induced Markov chain has at most one recurrent class. A more general quality is that of \emph{communicating}, meaning that for every pair of states $s,s'\in\mathcal{S}$, there exists a policy $\mu$ that {\color{black}transitions} from state $s$ to state $s'$ with non-zero probability in finitely many steps. A further generalization of this is \emph{weakly communicating}, which allows for the existence of a (possibly empty) subset of states which are transient under every policy, and therefore are excluded from the choices of $s'$ in the original condition. For weakly communicating MDPs, we can use the unichain LP formulation to obtain an optimal solution, with the caveat that the solution will not specify actions for the transient states under this optimal policy. A thorough explanation of the differences between optimizing unichain and multichain MDPs is given by \cite{Puterman1994MarkovDP}. We {provide a theoretical result} with \autoref{theorem:communicat}, demonstrating that the MDP we have formulated is weakly communicating, and thus the simpler LP formulation is suitable.

\begin{theorem}
\label{theorem:communicat}
    DMP is weakly communicating for all $N>0$. However, there exist $N$ and $M_i$ for $i=1,...,N$ such that DMP is not unichain.
\end{theorem}
\begin{proof}
    See Supplementary Material, Section A.1.
\end{proof}

As a result of \autoref{theorem:communicat}, we can use the most simple LP formulation of the CTMDP, following \cite{xianping2009continuous}. Our first formulation is called the \emph{Bi-Objective Dynamic Maintenance Problem}, or BO-DMP:
\begin{align}
    &\text{(BO-DMP)}&\min_\pi\,g^o = &\sum_{\bm{s}\in \mathcal{S}}\sum_{\bm{a}\in \mathcal{A}(\bm{s})}c^o(\bm{s},\bm{a})\pi(\bm{s},\bm{a}) \label{mo-dmp-obj1}\\
    &&\min_\pi\,g^f = &\sum_{\bm{s}\in \mathcal{S}}\sum_{\bm{a}\in \mathcal{A}(\bm{s})}c^f(\bm{s},\bm{a})\pi(\bm{s},\bm{a})\label{mo-dmp-obj2}\\
    &&\text{s.t.} & \sum_{\bm{s}\in \mathcal{S}}\sum_{\bm{a}\in \mathcal{A}(\bm{s}')}q(\bm{s},\bm{a},\bm{s}')\pi(\bm{s},\bm{a}) = 0,\,\forall \bm{s}'\in \mathcal{S},\label{mo-dmp-balance}\\
    &&& \sum_{\bm{s}\in \mathcal{S}}\sum_{\bm{a}\in \mathcal{A}(\bm{s})}\pi(\bm{s},\bm{a}) = 1\label{mo-dmp-sumto1},\\
    &&& \pi(\bm{s},\bm{a}) \geq 0,\,\forall \bm{s}\in \mathcal{S}, \bm{a}\in \mathcal{A}(s). \label{mo-dmp-pos}
\end{align}

Each decision variable $\pi(\bm{s,a})$ in the above formulation is a state-action frequency, which represents the proportion of time spent in state $\bm{s}$ and taking action $\bm{a}$. Objectives (\ref{mo-dmp-obj1}) and (\ref{mo-dmp-obj2}) are the operational costs and system failure objectives, respectively. Constraints (\ref{mo-dmp-balance}) are the CTMDP equivalent of the well-known balance equations associated with standard CTMCs. Constraints (\ref{mo-dmp-sumto1}) and (\ref{mo-dmp-pos}) ensure that $\pi$ is a valid probability distribution. This rather general formulation of a bi-objective CTMDP requires additional assumptions that we are using a stationary policy, and that the CTMDP is weakly communicating. While we have proven the latter for our case, we must make the argument for the former. While it is known that DS optimal policies always exist for the single-objective case, and therefore attention can be limited to such policies, the same doesn't hold for multi-objective MDPs. Following \cite{WHITE1982momdpNonStationary}, non-stationary policies may offer Pareto-optimal solutions that strike a unique balance between objectives that cannot be emulated by DS policies. However, their unique objective values can be emulated by a stochastic stationary policy \citep{roijers2013momdp_survey}. These Pareto-optimal stochastic policies will correspond to points on the facets of the polytope defined by \eqref{mo-dmp-balance}-\eqref{mo-dmp-pos}, which in turn means they can be interpreted as stochastic interpolations of the basic feasible solutions (i.e. DS policies). However, in practice, these stochastic policies may be considered unintuitive, untrustworthy, or difficult to interpret. This is especially the case for areas such as reliability, where the suggestion of making random maintenance decisions could be treated with suspicion. Additionally, for problems of even a moderate size, the DS policies alone provide a good amount of flexibility with respect to balancing between the two objectives. As such, we will limit ourselves to finding DS policies that are non-dominated with respect to stochastic stationary policies, so that DS policies that are dominated by some stochastic stationary policy will be ignored. To ensure that we only find such solutions, we will scalarize our bi-objective program using the mixed objectives method, yielding a single-objective problem with no additional constraints. As such, a solution to a scalarized problem as returned by a DP method or LP solver will be a basic feasible solution, i.e. a DS policy. To do this, we introduce a penalty parameter $p$ to introduce a virtual monetary cost associated with system failure, and add this to our first objective. This yields the problem we call $p-$DMP:
%The above formulation ensures that the resulting efficient solutions are stationary, and therefore depend only on the current state, and not the time elapsed. However, it is known that stochastic stationary policies can be efficient solutions for multi-objective MDPs \citep{roijers2013momdp_survey}. Whilst such policies form a broader class than that of deterministic stationary policies, they could be considered overly complex, unintuitive, or untrustworthy due to their random nature. This is especially damaging in our case of a reliability problem, especially if the target system is of significant importance. Therefore, to limit our goals to deterministic stationary policies, we wish to reduce our multi-objective problem to one with a single objective, ensuring deterministic policies. To do this, we introduce a penalty parameter $p$ to introduce a virtual monetary cost associated with system failure, and add this to our first objective. This yields the problem we call $p-$DMP:

\begin{align}
    &\text{($p-$DMP)}&\min_\pi\, &g = \sum_{\bm{s}\in \mathcal{S}}\sum_{\bm{a}\in \mathcal{A}(\bm{s})}\left(c^o(\bm{s,a}) + pc^f(\bm{s,a})\right)\pi(\bm{s,a}) \label{p-dmp-obj1}\\
    &&\text{s.t. } &\text{(\ref{mo-dmp-balance}) - (\ref{mo-dmp-pos})}.
\end{align}
This problem provides a scalarized version of BO-DMP, and is equivalent to a weighted sum single objective. We can solve this problem for different values of $p$ to obtain a set of efficient solutions that correspond to DS policies. As $p-$DMP is a single-objective LP, it can be solved easily using any LP solver as long as the state-action space is of a manageable size. 

\subsection{Integrated Design and Dynamic Maintenance Problem}
\label{paper1:formulation:int}
In the \emph{Integrated Design and Dynamic Maintenance Problem} (IDDMP), we wish to simultaneously optimize both the design variables of the redundant system and the dynamic policy used to maintain it, with respect to knapsack-style constraints $Ax\leq b$ (with all entries in $A$ and $b$ non-negative). To do this, we must first construct a state space that contains the maximum number of copies of each component with respect to the constraints:
$$\mathcal{S}^{\max} = \bigtimes_{i=1}^N \mathcal{S}_i,$$
where $\mathcal{S}_i =  \left\{(s_1,s_2)\in\NN_0^2:s_1 + s_2 \leq M_i\right\}$ gives the maximal one-component state space for component $i$ and $M_i = \min_i\{b_j/a_{ji}\}$ gives the maximum number of copies of component $i$ with respect to the knapsack constraints. Using this, we may extend $p$-DMP to include design decision variables as follows:
\begin{align}
    &\text{($p-$IDDMP)}&\min_{\pi,x}\, &g = \sum_{\bm{s}\in \mathcal{S}^{\max}}\sum_{\bm{a}\in \mathcal{A}(\bm{s})}\left(c^o(\bm{s},\bm{a}) + pc^f(\bm{s},\bm{a})\right)\pi(\bm{s},\bm{a}) \label{p-iddmp-obj}\\
    &&\text{s.t. } &\text{(\ref{mo-dmp-balance}) - (\ref{mo-dmp-pos})} \label{p-iddmp-const}\\
    &&& Ax\leq b \label{p-IDDMP-knap}\\
    &&& \sum_{\bm{s}\in \mathcal{S}^{\max},\bm{a}\in\mathcal{A}(s): s_{i2} - a_i \leq M_i - j}\pi(\bm{s},\bm{a}) \leq x_{ij} \text{ for }i=1,2,...N,\,j=1,...,M_i\label{p-IDDMP-link}\\
    &&&x_{i(j+1)} - x_{ij} \leq 0 \text{ for }i=1,...,N,\,j=1,...,M_i - 1\ \label{p-IDDMP-sym}\\ 
    &&& x_{ij}\in\{0,1\} \text{ for }i=1,2,...N,\,j=1,...,M_i \label{p-IDDMP-bin}.
\end{align}

We introduce binary variables $x_{ij}$ to represent our design decisions, where $x_{ij} = 1$ represents the decision to install the $j$th copy of component $i$. Objective function (\ref{p-iddmp-obj}) and constraints (\ref{p-iddmp-const}) are carried over from $p$-DMP, with $\mathcal{S}$ replaced by $\mathcal{S}^{\max}$. Constraints (\ref{p-IDDMP-knap}) are the knapsack-style constraints over the design variables. Constraints (\ref{p-IDDMP-link}) ensure that the probability of having more than $j$ healthy or repairing copies of component $i$ is limited by $x_{ij}$. This means that the probabilities are unconstrained if $x_{ij} = 1$, or are all forced to be zero if $x_{ij} = 0$. This works on the principle that having some number of damaged components that cannot ever be repaired is equivalent to not having those components at all. Constraints (\ref{p-IDDMP-sym}) enforce the logic that if you do not install the $j$th copy of component $i$, then you cannot install the $(j+1)$th copy. Constraints (\ref{p-IDDMP-bin}) ensure that the $x_{ij}$ variables are binary. We note that BO-DMP can be extended in a similar way to obtain BO-IDDMP, but we omit the formulation for brevity. 

While the formulation of $p$-IDDMP appears to be quite general and applicable to other problems, we must take note of its limitations. It requires that the large underlying MDP over the state space $\mathcal{S}^{\max}$ is weakly communicating, and one would have to adapt the multichain LP of \cite{Puterman1994MarkovDP} to the design setting to get a more general framework, which may require additional care. Also, our model relies on the assumption that we are in the long-run average optimality setting. In the discounted setting, the sum of the variables is bounded above by $M_\gamma = 1/
(1-\gamma)$, where $\gamma$ is the discount factor. This large constant would have to be integrated into constraint \eqref{p-IDDMP-link}, leading to a far weaker constraint in the linear relaxation. Further discussion of a variant of this big-M constraint in the context of discounted MDP Design problems is given by \cite{BROWN2024mdpDesignLetters}.
\begin{table}[tbp]
    \centering
    \begin{tabular}{cc}
    \hline
     \textbf{Notation} & \textbf{ Description }  \\
    \cmidrule(lr){1-1}
    CTMDP & \\
    \cmidrule(lr){1-1}
    $\bm{s}$ & State vector\\
    $\bm{a}$ & Action vector\\
    $\mu$ & Policy\\
    $g_\mu$ & Long-run average cost under policy $\mu$\\
    $q(\bm{s,a,s'})$ & Transition rate from state $\bm{s}$ to $\bm{s'}$ under action $\bm{a}$\\
    $c(\bm{s})$ & Cost rate of state $\bm{s}$\\
    \cmidrule(lr){1-1}
    Problem &\\
    Parameters  & \\
    \cmidrule(lr){1-1}
    $N$ & Number of types of component\\
    $M_i$ & Maximum number of copies of component type $i$\\
    $\alpha_i$ & Rate of failure for component $i$\\
    $\tau_i$ & Rate of repair for component $i$\\
    $c_i$ & Usage cost rate of component $i$\\
    $r_i$ & Repair cost rate of component $i$\\
    $p$ & The penalty incurred for system failure\\
    \hline
    \end{tabular}
    \caption{Table of Notation}
    \label{tab:notation}
\end{table}
We recognize that $\mathcal{S}^{\max}$ grows quadratically in the number of allowed copies of any given component, and grows exponentially in the number of component types. As such, the number of continuous state-action frequency variables $\pi$ and the number of equality constraints from the balance equations grow very rapidly with the problem size (more specifically, the number of component types and the leniency of the knapsack constraints). {However, many state-action pairs will be infeasible under any design, as they represent a design that violates the knapsack constraints. As such, a more careful construction of $\mathcal{S}^{\max}$ and of $\pi$ will yield a much reduced set of variables. Further details can be found in the supplementary material Section A.2.} The problem is also complicated by the presence of binary decision variables, leading to a very large-scale mixed integer linear program (MILP). As such, we cannot hope to solve this problem directly for problems even of moderate size, and must design some bespoke solution methodology. \autoref{paper1:methodology} follows with the introduction of our heuristic methodology for approximating a set of Pareto-optimal solutions for BO-IDDMP. 
The notation introduced in this section is summarized in \autoref{tab:notation}.

{
\section{Methodology}
\label{paper1:methodology}
In this section we present a {novel yet simple heuristic solution methodology for BO-MDP Design problems}. \autoref{paper1:methodology:app} introduces the method, which we call the Approximate Pareto Population (APP) method, in {full generality, and provides a simple bound on performance.} \autoref{paper1:methodology:InterpApplic} provides an interpretation of how this algorithm works in general and in context of our problem, and explores how it might be applicable to other problems. \autoref{paper1:methodology:dop} analyses the main sub-problem that arises when we apply APP to BO-IDDMP, and \autoref{paper1:methodology:alg} gives an overview and illustration of the algorithm.
\subsection{Approximate Pareto Population}
\label{paper1:methodology:app}
APP is a novel heuristic framework for BO-MDP Design problems.  It is a population- and decomposition-based method that breaks down a BO-MDP Design problem into smaller sub-problems. We provide the method in full generality here as it could be applicable to other problem settings as a heuristic or benchmark, however the remainder of the paper will focus on its application to BO-IDDMP. The general form for a BO-MDP Design problem can be represented as a master problem MP:
\begin{align*}
    (\text{MP}) \quad\min_{x,\mu} & f^1(x,\mu)\\
                \min_{x,\mu} & f^2(x,\mu)\\
                \text{s.t. } & x\in X, \mu\in\mathcal{M}(x).
\end{align*}
% \begin{align*}
%     &(\text{MP}) &\min_{x,y}   & f^1(x,y)\\
%     &            &\min_{x,y}   & f^2(x,y)\\
%     &            &\text{s.t. } & Ax + By \leq b\\
%     &            &             & Cx + Dy = c\\
%     &            &             & x\in X, y\in Y.
% \end{align*}
Here, $x$ is the decision vector with feasible region $X$, and we assume that for each decision there is an associated MDP (or CTMDP) $\mathcal{C}(x)$ over which there is a full space of (potentially stochastic and non-stationary) policies $\mathcal{M}(x)$. The two objective functions $f^1$ and $f^2$ both take the design and policy as inputs. 
% While this is the problem setting, we limit our heuristic to finding DS policies for the reasons mentioned in \autoref{paper1:formulation:dynamic}.

The key idea of APP is that we first generate a finite set of candidate designs $X^*$ which strike different balances between the two objectives. Suppose that for each design $x$ we have an associated set of heuristic solutions $H_x\subset\mathcal{M}(x)$. By restricting MP to these heuristic solutions, we obtain our first subproblem SP1.
\begin{align*}
    (\text{SP1})\quad \min_{x,\mu}   & f^1(x,\mu)\\
                \min_{x,\mu}   & f^2(x,\mu)\\
                \text{s.t. } & x\in X, \mu\in H_x. 
\end{align*}

{The quality of solutions found by SP1 is of course determined by the chosen heuristics $H_x$. Based on the quality of these heuristics, we can bound the performance of SP1.
\begin{proposition}
    Let $f_w(x,\mu)=wf^1(x,\mu)+(1-w)f^2(x,\mu)$ for $w\in[0,1]$, and suppose that for all $x$ and $w$, $H_x$ is $(1+\varepsilon)$-optimal for $\min_{\mu\in\mathcal{M}(x)}\left\{f_w(x,\mu)\right\}$, i.e. $\min_{\mu\in H_x}\left\{f_w(x,\mu)\right\}\leq (1+\varepsilon)\min_{\mu\in \mathcal{M}(x)}\left\{f_w(x,\mu)\right\}$. Then the optimal solution of SP1 is $(1+\varepsilon)$-optimal for MP.
\end{proposition}
\begin{proof}
    Let $(x^*,\mu^*)$ be an optimal solution for $\min_{x\in X,\mu\in\mathcal{M}(x)}\left\{f_w(x,\mu)\right\}$ (i.e. the mixed-objective version of MP). Then there exists a $\mu\in H_{x^*}$ such that $f_w(x^*,\mu)\leq (1+\varepsilon)f_w(x^*,\mu^*)$, and this $(x^*,\mu)$ is feasible for SP1. Therefore the optimal solution of SP1 is $(1+\varepsilon)$-optimal for MP under a mixed-objective scalarization.
\end{proof}}

The design solutions of SP1 form a non-dominated front that we denote by $X^*$. This set could in general be infinite if $X$ is infinite (e.g. if $X$ contains continuous variables). {However, we focus on the finite case here}. The second stage of APP iterates over all candidate designs $x\in X^*$ and produces a Pareto front of solutions with the decision variables fixed. For any first-stage decision $x$, we call this sub-problem SP2$(x)$.
\begin{align*}
    (\text{SP2}(x)) \quad&\min_{\mu\in\mathcal{M}(x)}f^1(x,\mu)\\
                    &\min_{\mu\in\mathcal{M}(x)}f^2(x,\mu)
\end{align*}
% \begin{align*}
%     &(\text{SP2$(x)$}) &\min_{y}   & f^1(x,y)\\
%     &            &\min_{y}   & f^2(x,y)\\
%     &            &\text{s.t. } & By \leq b - Ax\\
%     &            &             & Dy = c - Cx\\
%     &            &             & y\in Y.  
% \end{align*}
We denote the Pareto front of SP2$(x)$ by $\mathcal{M}^*(x)$. {SP2 both closes the $(1+\varepsilon)$-optimality gaps of the solutions found in SP1, and also explores how well each design solution performs for different weightings of the two objectives, (or, in other words, explores different parts of the Pareto front). However, this is not always guaranteed to yield improvements.
\begin{remark}
    Consider an absolute worst-case scenario where, for each $w$, $H_x$ attains the $(1+\varepsilon)$ bound for the true optimal design $x^*_w$, and SP1 finds a solution $(x_w,\mu_w)$ such that $\mu_w$ is truly optimal for $x_w$ but attains the same value as found for the optimal design, i.e. $f_w(x_w,\mu_w)=(1+\varepsilon)f_w(x^*_w,\mu^*_w)$. As solvers generally only return one of the optimal solutions, in the worst case the sub-optimal design $x_w$ is returned for each value of $w$ instead of $x_w^*$, which cannot be improved by SP2 for the weight $w$. 
\label{sec3remark}\end{remark}
Of course, the scenario described in Remark  \ref{sec3remark} is a worst-case scenario that one would not consider to be at all common, but it does imply that we cannot make any claims about guaranteed improved bounds from SP2. However, it is clear that whenever SP2 can be easily solved, it makes sense to try it. Even in cases where SP2 is not easily solvable, it is still worth applying approximate methods such as ADP or RL to seek improvements on the solutions found in SP1, especially if bi-objective RL is known to work well for fixed designs.}

Our overall population of candidate solutions can then be denoted as $P=\{(x,\mu):x\in X^*, \mu\in\mathcal{M}^*(x)\}$. Our final approximate Pareto front can then be determined as the set of non-dominated solutions in $P$ with respect to the objectives $f^1$ and $f^2$. {Despite the somewhat weak performance guarantees of this methodology, especially with respect to SP2, we believe that this framework is an important first step to obtaining good solutions to an otherwise hard and novel class of integrated problems, and can comfortably serve as a heuristic in its own right, or a benchmark against which future methods can be compared.} With our general methodology outlined, we now offer an interpretation of this algorithm.

\subsection{Interpretation and Applicability}
\label{paper1:methodology:InterpApplic}
In the case of solving BO-MDP Design problems, APP first solves a bi-objective design problem that is restricted to a set of heuristic policies $H_x$ to control the sequential problem, and then solving the true dynamic problem using the designs found in the first stage. We acknowledge that the problems SP1 and SP2 may be challenging to solve in their own right, and could require approximate methods such as metaheuristics or reinforcement learning, respectively. By design, the algorithm does of course produce a non-dominated set of solutions (with respect to itself) that interpolate between the two single-objective heuristic solutions, and a real-world decision maker can ultimately decide whether or not the solutions provide a good trade-off between the objectives for real-world deployment, regardless of whether or not they are truly Pareto-optimal.

{We now turn our attention back to the IDDMP. A reasonable choice for a set of heuristic policies is those policies that always repair some subset of the installed components as soon as they become damaged, and entirely neglect the other components, along the same lines as the critical component policies of \cite{flynn2004heuristic}. However, in our problem, finding a design $x$ and a critical component policy $\mu$ that always repairs a subset of the components represented by $x$ will yield an identical outcome to choosing a design $x'$ that represents the components actively maintained by $\mu$, and choosing to actively maintain all components. As such, for any design $x$, we simplify the problem by allowing only one heuristic maintenance policy (so $|H_x|=1$), which is the one that always maintains all active components.} We can then interpret SP1 as finding different combinations of components that strike different balances between operational costs and system reliability, under the assumption that all components are actively maintained. Then, for each design $x$ that comes out of this, SP2($x$) will find dynamic maintenance policies for design $x$. It may sometimes happen that we find two designs $x^{(1)}$ and $x^{(2)}$ such that $x^{(1)}$ performs better than $x^{(2)}$ with respect to the first objective (cost) but worse than $x^{(2)}$ with respect to the second objective (reliability) under the fully active maintenance policy, but the design $x^{(2)}$ yields a `lazier' maintenance policy that performs better than the fully active policy under $x^{(1)}$ with respect to both objectives. That is to say a lazy maintenance policy on an intrinsically more reliable design may be both cheaper and more reliable in the long-run than the fully active maintenance on a cheaper, less reliable system. This shows the benefit of heuristically solving the problem in one large bi-objective sweep, as opposed to heuristically solving instances of the mixed-objective $p$-IDDMP problem directly. \autoref{paper1:methodology:dop} now follows with an in-depth exploration of SP1 for BO-IDDMP, and demonstrates how this problem can be linearized and solved using standard MILP solvers.
 
% APP becomes a desirable approach if we can choose $H$ such that SP1 is relatively easy to solve and provides a diverse range of candidate solutions in its Pareto front. It is also desirable for $H$ to satisfy the condition that, for any fixed $x$, the solution $y = H(x)$ is always Pareto-optimal with respect to $f^1$ and $f^2$.
}
\subsection{Design-Only Problem}
\label{paper1:methodology:dop}
When applying APP to BO-IDDMP, we noted previously that a natural choice of the heuristic mapping $H$ is the one that maps any design $x$ to the ``fully active'' maintenance policy which places any installed components into repair as soon as they become damaged. This mapping is desirable in part because it allows SP1 to find both single-objective optimal solutions of BO-IDDMP: the empty solution for minimizing cost, and the most reliable design being actively maintained for minimizing failure rate. We call this sub-problem the \emph{Bi-Objective Design-Only Problem} (BO-DOP). Not only is BO-DOP useful within the APP framework, but it's also useful in its own right as a point of comparison against BO-IDDMP, as we can compare the solution sets between the two problems. BO-DOP can be considered the static or non-dynamic version of the problem, as typically seen in the repairable RAP literature. The comparison between BO-DOP and BO-IDDMP will be carried out in \autoref{paper1:study}.

If a component of type $i$ is always repaired as soon as it becomes damaged, then it is healthy with probability $p_i = \tau_i/(\tau_i + \alpha_i)$ and repairing with probability $q_i = 1-p_i = \alpha_i/(\tau_i + \alpha_i)$.
%\begin{align*}
    %$p_i = \tau_i/(\tau_i + \alpha_i), 
    %q_i = 1-p_i = \alpha_i/(\tau_i + \alpha_i)$.
%\end{align*}
Therefore, in BO-DOP, any copy $j$ of component type $i$ under design decision $x_{ij}$ is healthy with probability $p_ix_{ij}$, non-healthy (i.e. repairing or non-existent) with probability $1-p_ix_{ij}$, and repairing with probability $q_ix_{ij}$. The two objectives are therefore as follows:
\begin{align}
    g^o_{\text{DOP}}(x) &=\sum_{i=1}^N\sum_{j=1}^{M_i}  r_iq_ix_{ij} +\sum_{i=1}^N\sum_{j=1}^{M_i}c_ip_ix_{ij}\left(\prod_{k<i}\prod_{l=1}^{M_k}(1 - p_kx_{kl})\right)\left(\prod_{l<j}(1 - p_ix_{il})\right) \label{BO-DOP-obj1}\\
    g^f_{\text{DOP}}(x) & = \prod_{i=1}^N\prod_{j=1}^{M_i}(1-p_ix_{ij}). \label{BO-DOP-obj2}
\end{align}

Objective (\ref{BO-DOP-obj1}) appears quite complex, so we provide an explanation. The first double-summation states that for every component of type $i$ and copy $j$, we incur repair cost-rate $r_i$ with probability $q_ix_{ij}$. The second summation arises from the fact that for every component type $i$ and copy $j$, we incur from it usage cost $c_i$ with the probability that component $i$ copy $j$ is healthy ($p_ix_{ij})$ and all cheaper components are non-healthy (expressed by the double product) and the first $j-1$ copies of component type $i$ are all non-healthy (expressed by the single product). Objective (\ref{BO-DOP-obj2}) is simply the probability that all components are non-healthy. BO-DOP can be expressed as follows:
\begin{align*}
    \text{(BO-DOP)}\quad \min_x\,      & g^o_{\text{DOP}}(x)\\
                    \min_x\,      & g^f_{\text{DOP}}(x)\\
                    \text{s.t. } & (\ref{p-IDDMP-knap}), (\ref{p-IDDMP-sym}), (\ref{p-IDDMP-bin}),
\end{align*}   
where we recall Constraints (\ref{p-IDDMP-knap}), (\ref{p-IDDMP-sym}), (\ref{p-IDDMP-bin}) are the knapsack, symmetry breaking and binary constraints, respectively. As currently formulated, BO-DOP is a nonlinear integer program, and would be very difficult to solve exactly. In order to solve the problem, we consider a single-objective, $\varepsilon-$constrained and slightly augmented problem which allows for linearization. We choose the system failure objective to be treated as an $\varepsilon$-constraint. The $\varepsilon-$constraint method is a standard scalarization method for bi-objective problems, and is chosen over the weighted-sum approach to recover a single-objective problem because the weighted-sum approach is known to miss non-supported efficient solutions in the case of multi-objective mixed integer programs \citep{ehrgott2005multicriteria}. Additionally, our preliminary experimentation with the weighted-sum approach to BO-DOP led to numerical instability when compared with the $\varepsilon-$constraint method. Therefore, we focus on the latter method. We augment the problem by adding a small penalty for system failure into the operational cost objective function, using an additional parameter $\delta \geq 0$. As such, we call the resulting problem $\varepsilon-\delta-$DOP.
\begin{align}
    &(\varepsilon-\delta-\text{DOP}) &\min_x & \sum_{i=1}^N\sum_{j=1}^{M_i}  r_iq_ix_{ij} \nonumber\\ &&&+\sum_{i=1}^N\sum_{j=1}^{M_i}c_ip_ix_{ij}\left(\prod_{k<i}\prod_{l=1}^{M_k}(1 - p_kx_{kl})\right)\left(\prod_{l<j}(1 - p_ix_{il})\right) \nonumber\\
    &&&+(1+\delta)c_N\prod_{i=1}^N\prod_{j=1}^{M_i}(1-p_ix_{ij})\label{e-DOP-obj}\\
    &&\text{s.t. }& (\ref{p-IDDMP-knap}), (\ref{p-IDDMP-sym}), (\ref{p-IDDMP-bin}), \label{e-DOP-other}\\
    &&&\sum_{i=1}^N\sum_{j=1}^{M_i}\log(q_i)x_{ij} \leq \varepsilon. \label{e-DOP-epsilon}
\end{align}

Objective (\ref{e-DOP-obj}) is $g^o_{\text{DOP}}$ with an additional penalty term $(1+\delta)c_N$ for system failure. This term ensures that the objective function considers system failure to be at least as bad as using the most expensive component type{, which is required for the linearization method employed in the next step}. Constraints (\ref{e-DOP-other}) are the the knapsack, symmetry breaking and binary constraints used in BO-DOP. Constraint (\ref{e-DOP-epsilon}) is the (log) $\varepsilon-$constraint, ensuring that the log failure rate (LFR) falls below some target value $\varepsilon$.  

{As it stands, $\varepsilon-\delta-$DOP is a heavily nonlinear integer program, and is in fact a supermodular optimization problem (as shown in the supplementary material, Section B.1), making it very hard to solve directly.} Fortunately, the objective function in $\varepsilon-\delta-$DOP is very similar to that of the well-studied Unreliable p-Median Problem. In relation to this problem, \cite{OHANLEY201363} developed an exact linearization method known as \emph{probability chain}, and we shall apply this method to our problem in order to obtain a formulation that can be solved exactly using standard MILP solvers. The resulting problem, $\varepsilon-\delta-$DOP-PC, is given as follows:
\begin{align}
     &(\varepsilon-\delta-\text{DOP-PC}) &\min_{x,y,z} &\sum_{i=1}^N\sum_{j=1}^{M_i} (r_iq_ix_{ij} + c_iy_{ij}) + (1 + \delta) c_N z_{N,M_N}, \nonumber\\
    &&\text{s.t. }& (\ref{e-DOP-other}) - (\ref{e-DOP-epsilon}) \nonumber \\
    &&& y_{11} + z_{11} = 1, \label{e-d-dop-pc-1}\\
    &&& y_{i(j+1)} + z_{i(j+1)} = z_{ij}, \text{ for all } i = 1,...,N, j = 1,...,M_i-1,\label{e-d-dop-pc-2} \\
    &&&y_{i1} + z_{i1} = z_{(i-1)M_{i-1}}, \text{ for all } i = 2,...,N,\label{e-d-dop-pc-3} \\
    &&& y_{ij} \leq p_ix_{ij}, \text{ for all }i=1,...,N, j = 1,...,M_i\label{e-d-dop-pc-4}\\
    &&& y_{i(j+1)} \leq p_{i}z_{ij}, \text{ for all }i = 1,...,N, j = 1,...,M_i-1\label{e-d-dop-pc-5}\\
    &&&y_{(i+1)1} \leq p_{i+1}z_{iM_i}, \text{ for all }i=1,...,N-1,\label{e-d-dop-pc-6}\\
    &&&y_{ij}, z_{ij} \geq 0.
\end{align}
{For clarity of exposition, we say component $(i,j)$ \emph{precedes} $(i',j')$ if $i<i'$ or $i=i'$ and $j<j'$. The variable $y_{ij}$ represents the probability that component $(i,j)$ is healthy and must be used (i.e. all preceding components are damaged), and the variable $z_{ij}$ is the probability that component $(i,j)$ would be used if it were not damaged, i.e. the probability that component $(i,j)$ and all preceding components are damaged. To achieve this, constraints \eqref{e-d-dop-pc-1}
 - \eqref{e-d-dop-pc-3} ensure that $y_{ij}$ and $z_{ij}$ always sum to the $z$ value of the preceding component, constraint \eqref{e-d-dop-pc-4} ensures $y_{ij}$ is bounded above by $p_i$ if component $(i,j)$ is installed, or 0 otherwise. Constraints \eqref{e-d-dop-pc-5} and \eqref{e-d-dop-pc-6} then ensure that $y_{ij}$ is bounded above by $p_iz_{i,j-1}$ if $j>1$, or $p_iz_{i-1,M_{i-1}}$ if $j=1$. The penalty term in front of $z_{N,M_N}$, which outweighs the costs for each $y_{ij}$, ensures that these bounds are attained at optimality if the component is installed. As such, these constraints together ensure that the $y$ and $z$ variables take on the desired values at optimality:
\begin{align}
    y_{ij} &= \left(\prod_{k<i}\prod_{l=1}^{M_k}(1 - p_kx_{kl})\right)\left(\prod_{l<j}(1 - p_ix_{il})\right)p_ix_{ij}, \\
    z_{ij} &= \left(\prod_{k<i}\prod_{l=1}^{M_k}(1 - p_kx_{kl})\right)\left(\prod_{l\leq j}(1 - p_ix_{il})\right).
\end{align}
The rigorous proof of equivalence between $\varepsilon-\delta-$DOP-PC and $\varepsilon-\delta-$DOP follows immediately from Proposition 1 of \cite{OHANLEY201363}, where the only difference is that our formulation has extra indices for repeated components. As such, the nonlinear integer optimization problem $\varepsilon-\delta-$DOP with $\sum_{i=1}^NM_i$ integer variables is transformed into a mixed-integer linear program with the same number of integer variables, and $2\sum_{i=1}^N M_i$ continuous variables which take on all of the nonlinearities from the original problem. This results in a much easier problem to solve in general.}

Problem $\varepsilon-\delta-$DOP-PC is a MILP that can be easily solved by any MILP solver. We can then re-solve this MILP for different values of $\varepsilon$ to obtain a set of efficient solutions to BO-DOP. To assist the $\varepsilon-$constrained method, it is also worthwhile to consider the two problems that only consider one objective, whilst ignoring the other. These problems give the extreme points of our Pareto front. In the case of minimizing operational costs, the (trivial) optimal solution is to install no components and incur no costs, hence leaving the system permanently failed. In the case of minimizing failure probability, we call the MILP that minimizes failure probability while neglecting operational costs the Failure Design-Only Problem (F-DOP), and formulate it as follows:
\begin{align}
    \text{(F-DOP)} \quad\min_x &\sum_{i=1}^N \log(q_i)x_i\\
    \text{s.t. }& Ax\leq b\\
    & x_i\in\NN_0 \text{ for }i=1,...,N.
\end{align}
This allows us to obtain the solution to the bi-objective problem that gives all preference towards the objective of system reliability, and also provides a bound for our $\varepsilon$ values. The models $\varepsilon-\delta-$DOP-PC and F-DOP can both be used to provide solutions to BO-DOP.

\subsection{Algorithm Outline}
\label{paper1:methodology:alg}
\begin{figure}[tbp]
    \centering
    \begin{subfigure}{.5\textwidth}
        \centering
        \includegraphics[width=\textwidth]{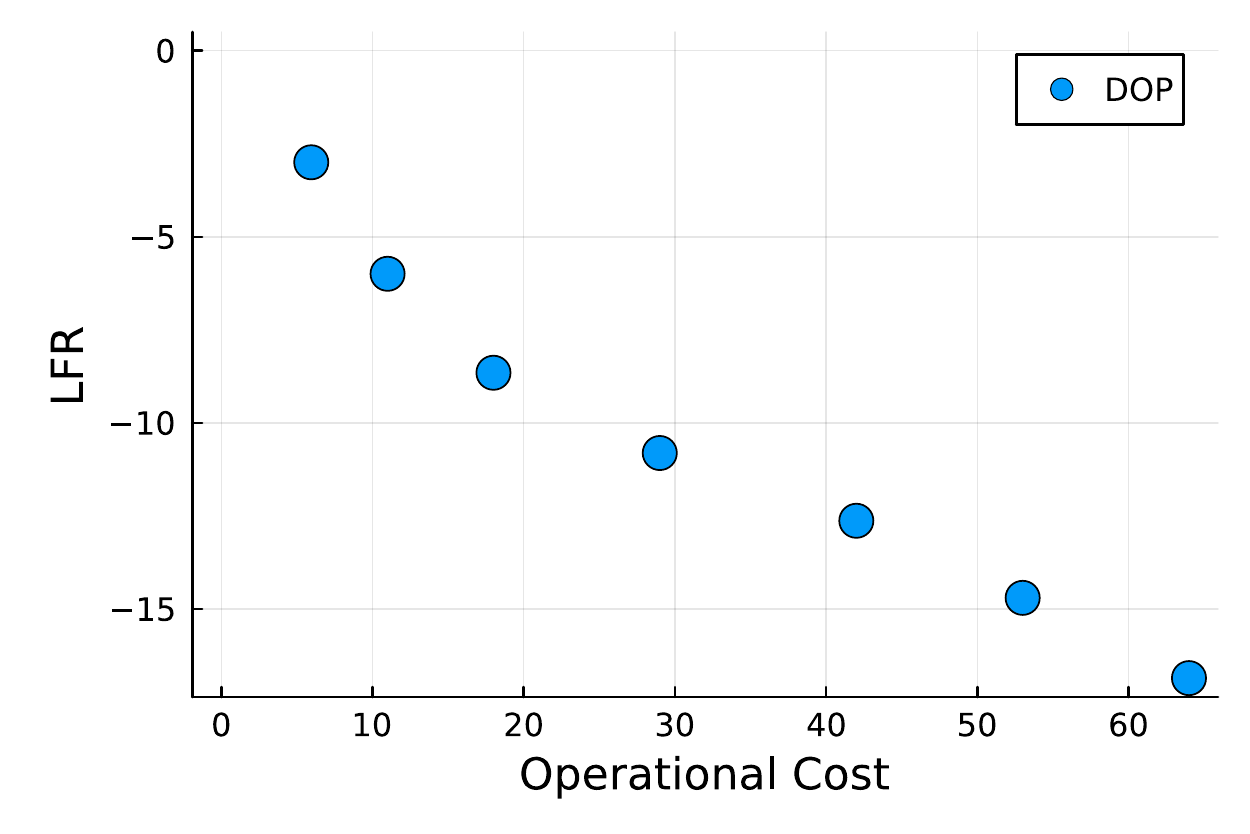}
        \subcaption{}
    \end{subfigure}%
    \begin{subfigure}{.5\textwidth}
        \centering
        \includegraphics[width=\textwidth]{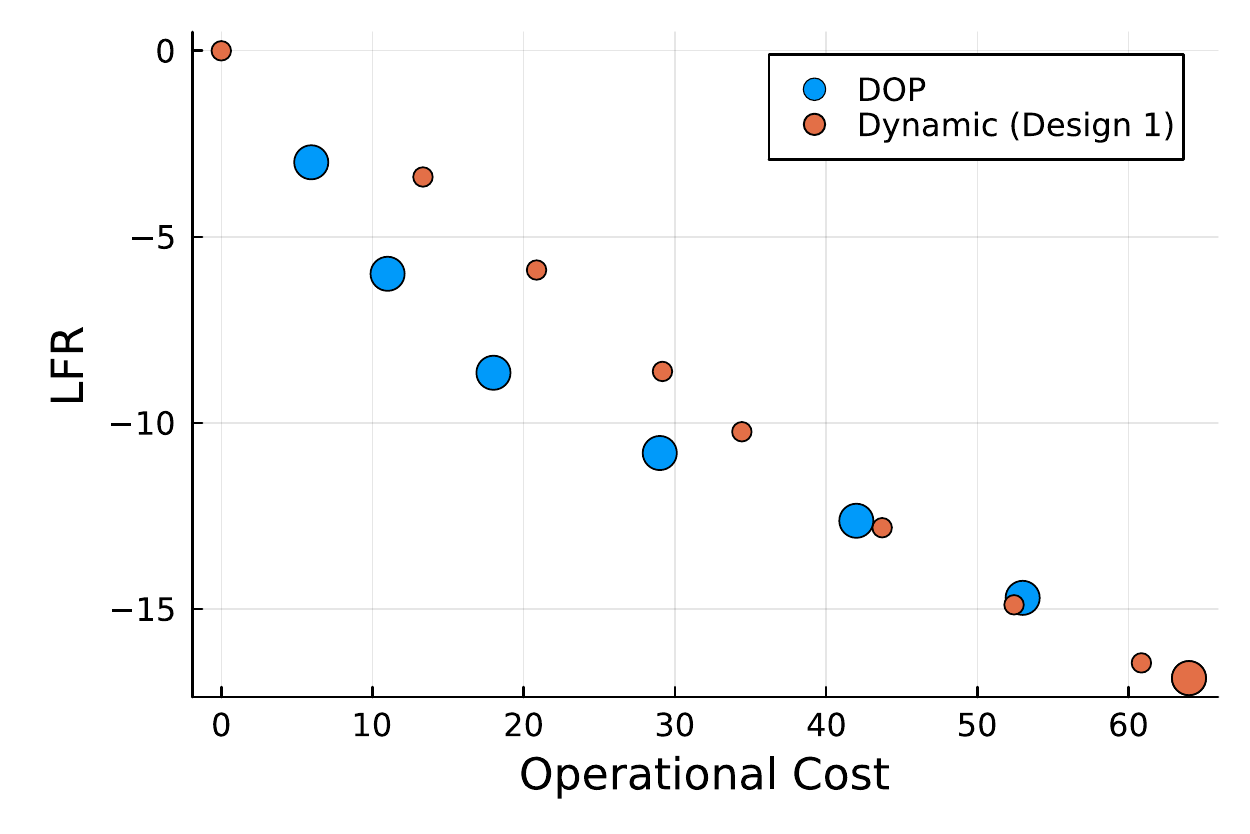}
        \subcaption{}
    \end{subfigure}
    \begin{subfigure}{.5\textwidth}
        \centering
        \includegraphics[width=\textwidth]{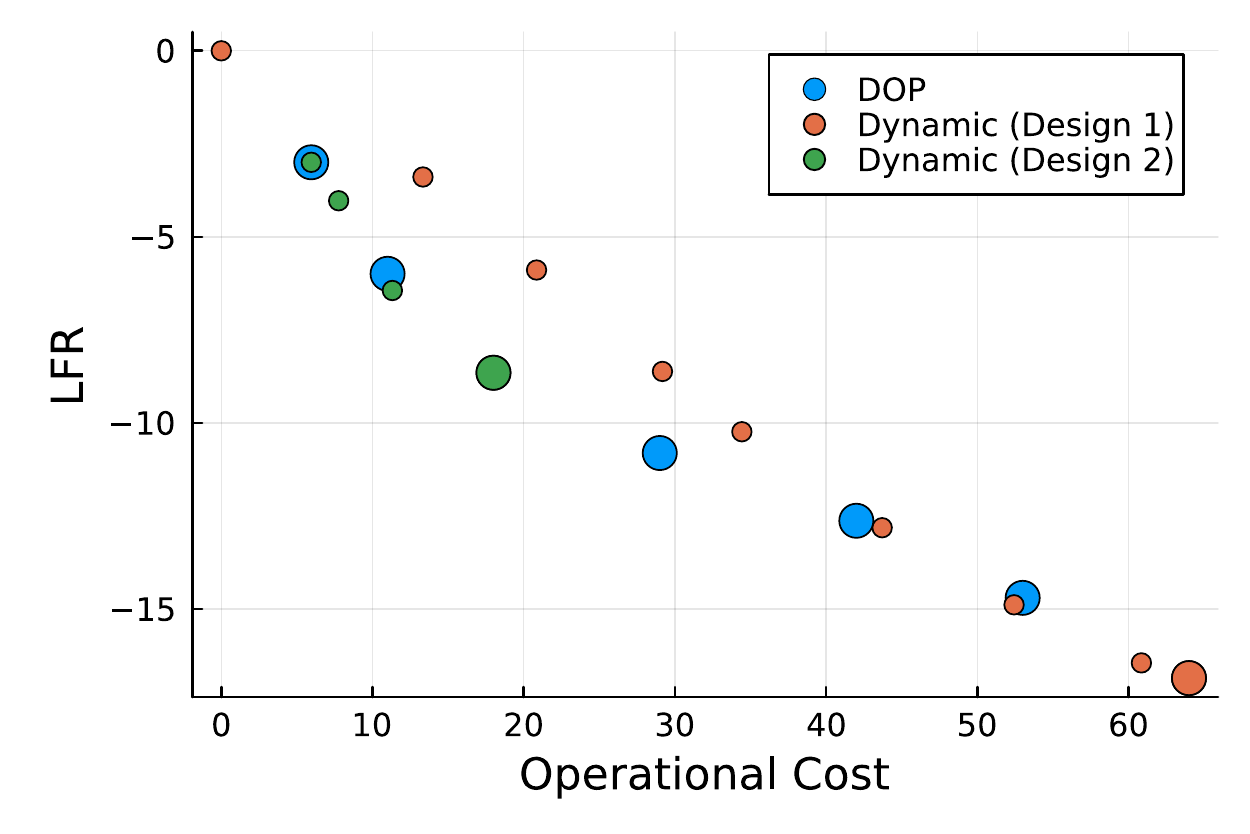}
        \subcaption{}
    \end{subfigure}%
    \begin{subfigure}{.5\textwidth}
        \centering
        \includegraphics[width=\textwidth]{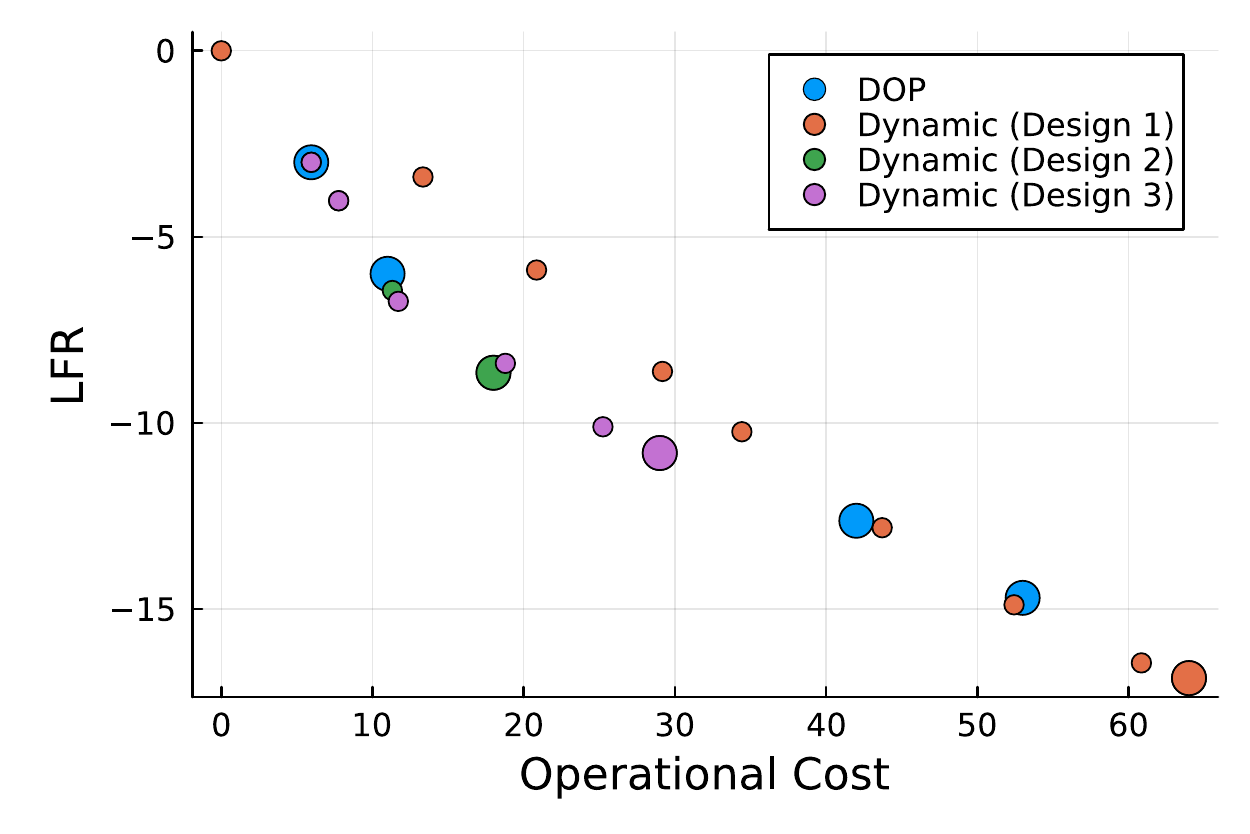}
        \subcaption{}
    \end{subfigure}
    \begin{subfigure}{.5\textwidth}
        \centering
        \includegraphics[width=\textwidth]{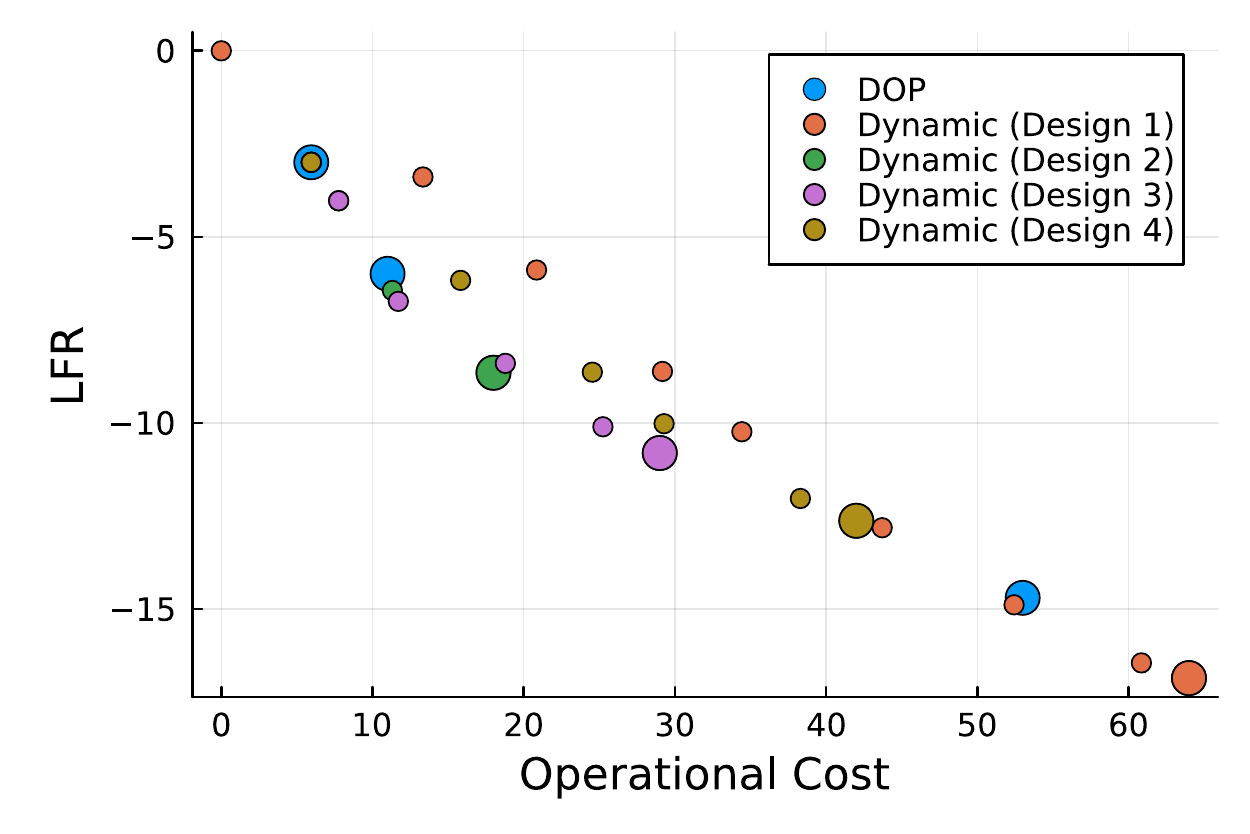}
        \subcaption{}
    \end{subfigure}%
    \begin{subfigure}{.5\textwidth}
        \centering
        \includegraphics[width=\textwidth]{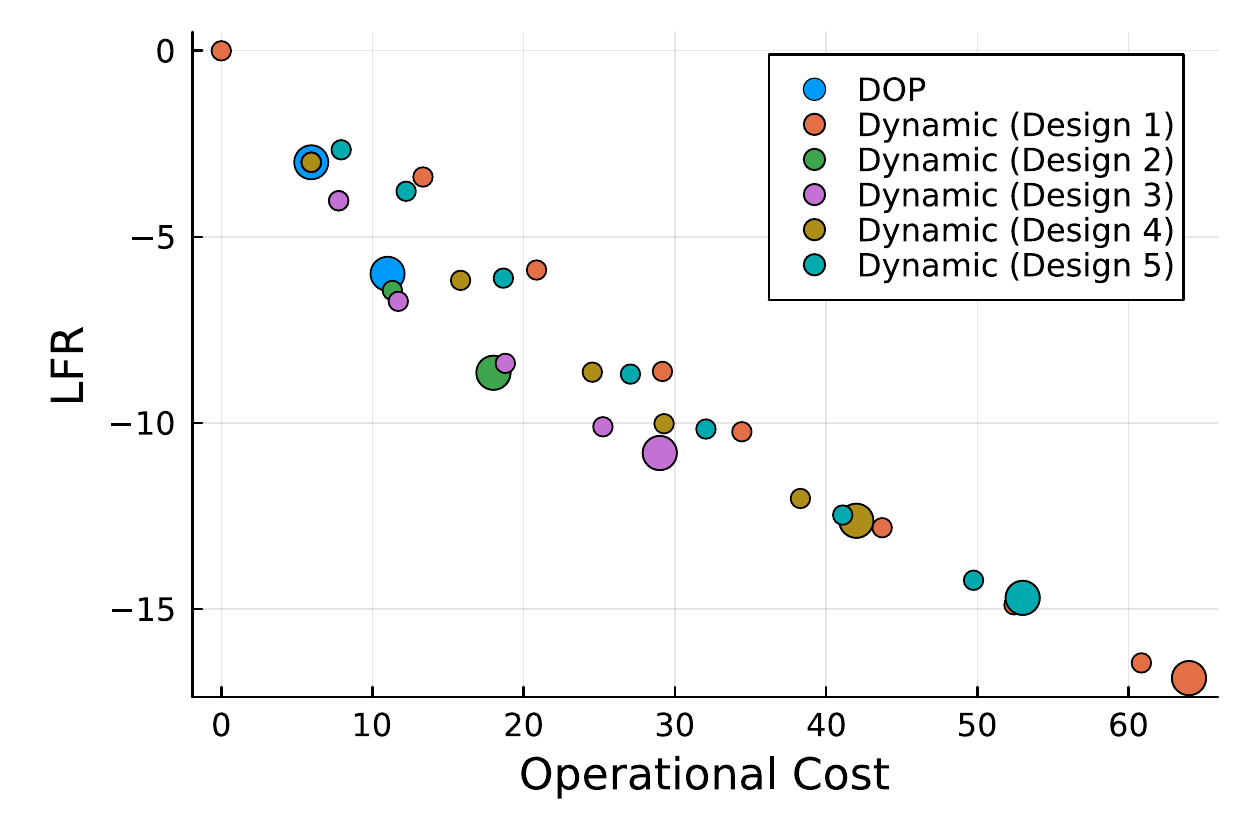}
        \subcaption{}
    \end{subfigure}
    \caption{Illustration of APP building up the population of candidate solutions.}
    \label{fig:APP}
\end{figure}

\begin{figure}[tbp]
    \centering
    \includegraphics[width = 0.75\textwidth]{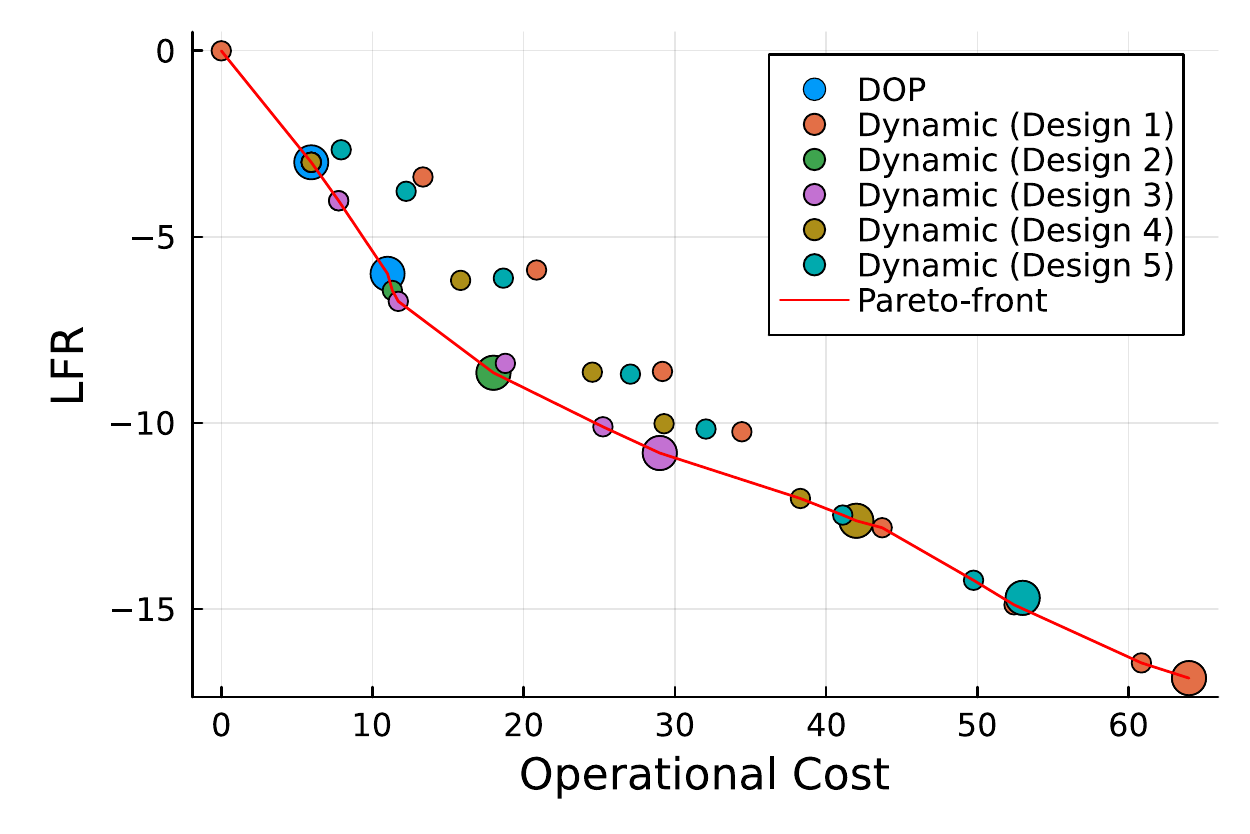}
    \caption{Pareto front of candidate population.}
    \label{fig:APP-final}
\end{figure}

We now give a brief overview and illustration of the algorithmic implementation of APP when applied to BO-IDDMP. See the supplementary material (Section B.2) for the full algorithms and detailed explanations. We start by solving BO-DOP, which is SP1 for BO-IDDMP. We do this by solving F-DOP using a MILP solver to obtain the most reliable solution, and otherwise use the $\varepsilon-$constraint method to obtain the other solutions. For each value of $\varepsilon$ we wish to test, a solution is obtained via a MILP solver on $\varepsilon-\delta-$DOP-PC. Our implementation of the $\varepsilon-$constraint method uses two parameters: the starting value $\varepsilon_{\min}$, and the increment value $\Delta\varepsilon$ (note that both of these are negative, as $\varepsilon$ is a log-probability). After solving the problem using $\varepsilon_{\min}$, the solution obtained will have a value for $\ln g^f$, the log failure rate of the solution. We update the target value $\varepsilon$ for the next solution by setting $\varepsilon = \ln g^f + \Delta\varepsilon$. This is done iteratively for each new solution, until $\varepsilon$ would exceed the minimum value of LFR found by F-DOP.

Next, we solve SP2$(x)$ for each design found by BO-DOP. SP2$(x)$ is simply the dynamic maintenance problem BO-DMP as outlined in Section \ref{paper1:formulation:dynamic}. This is solved using the dichotomic method of \cite{aneja1979dichotomic} to choose the weight of the two objective, which are then normalized to give a weight of $1$ the first objective, and a value of $p$ for the $p$-DMP formulation.

The solutions found by SP2 all form a population of candidate solutions. The final step is to identify all non-dominated solutions in our population, and we return this as our result. To provide further intuition for our methodology, \autoref{fig:APP} illustrates how APP builds up the population of candidate solutions. \autoref{fig:APP}(a) shows the set of static solutions that come from solving BO-DOP. Each of the subsequent subfigures then selects a design in a unique color, and adds in the Pareto front of dynamic policies using that design in the same color. \autoref{fig:APP-final} shows the final population, with the Pareto front of this population shown by the red line. {Albeit subtle (in this example), this front shows how a lazy dynamic maintenance policy from the orange (right-most) design dominates the teal (second-right-most) design under the fully active policy. This demonstrates the benefit of considering dynamic policies on more reliable designs, as opposed to just the design-only problem with fully active policies. Additionally, we see that two policies from the teal design belong to the non-dominated front, despite the base design being dominated. This shows that we should not discard designs from SP1 just because that solution is dominated by a solution from SP2, because that design may still be useful elsewhere in the front.}
With our heuristic fully described, \autoref{paper1:study} now follows with a computational study comparing the exact and heuristic approaches, and investigates the impact of problem parameters.

\section{Numerical Examples and Computational Results}
\label{paper1:study}

In this section we present the results of computational experiments in order to evaluate the performance of the APP heuristic, and also explore the effects of varying the parameters in our model. \autoref{paper1:study:instances} described how we adapt a well known test instance for series-parallel models into a suite of test instances for our parallel model. In \autoref{paper1:study:big} we show that in general our heuristic finds better populated Pareto fronts than the exact method without sacrificing the quality of the solutions found, and does so in far less time than the exact method. We do this by adapting a popular problem set from the RAP literature. In \autoref{paper1:study:usage} and \autoref{paper1:study:repairs} we investigate the effects of varying usage costs and repair costs on the solutions to our problem in order to demonstrate the importance and impact of these parameters in our model. %This means that degradations and repairs both happen faster or slower, but the overall proportion of up-time remains the same.
We demonstrate that such variations have an impact on dynamic policies. The code used can be found on \href{https://github.com/fairleyl/MDP-Design-Parallel}{GitHub}.

{\subsection{Test Instances}
\label{paper1:study:instances}
We use 14 different sets of components based on test instances found in the RAP literature, following the installation cost, weight, and reliability values of the 14 subsystems introduced by \cite{fyffe1968problemset}. To be clear, whereas the problem in \cite{fyffe1968problemset} is a single series-parallel problem with 14 subsystems, we treat each subsystem as a separate one-subsystem parallel problem to obtain a set of parallel-only problems. The original problem set defines, for each component $i$, its reliability $p_i$, its installation cost $C_i$, and its weight $w_i$. For the parameters unique to our model, for each component $i$ we set $\tau_i = 1$ and obtain $\alpha_i$ by solving $\tau_i/(\tau_i + \alpha_i) = p_i$. We set usage cost $c_i = 1$ and repair cost $r_i = 100$ for the initial problem set. The impacts of varying these parameters will be explored in subsequent subsections. For each set we use different budget values $B$ applied to both budget and weight, obtaining knapsack constraints $C^Tx\leq B$ and $w^Tx\leq B$. Clearly, as the budget $B$ increases, the number of possible copies of each component increases, and therefore the size of the state space in the IP increases. For each component set, we choose values of $B$ corresponding to the third up to the eighth multiples of the minimum weight found in the component set. For example, if the minimum weight in a set is 5, we test budgets of 15, 20, 25 and so on up to and including 40, reflecting problem instances that allow from 3 to 8 copies of the lightest component, respectively. We denote problem instances by $a$-$B$, where $a$ is the component set (i.e. subsystem from the original problem), and $B$ is the budget. For example, instance 5-15 includes components from the fifth subsystem and has a budget of 15. }
{\subsection{Runtime and Scalability}
\label{paper1:study:big}

In this section we solve the BO-IDDMP using an exact method, the heuristic method described in Section \ref{paper1:methodology}, and the non-dynamic BO-DOP solutions. For each problem instance, we compare the solutions with respect to runtimes and the numbers of solutions found. All MILPs and LPs are solved using Gurobi \citep{gurobi} and our own fork of MultiObjectiveAlgorithms.jl \citep{dowson2025multiObjectiveAlgorithmsJl}, and all code is written in Julia 1.9.3. The code is run on a computer with an Intel Xeon Gold 6348 with four cores clocked at 2.6 GHz and 16 GB of RAM. 

For each problem instance as defined in \autoref{paper1:study:instances}, we solve the problem in three ways: (i) using an exact method by applying the dichotomic method to the BO-IDDMP formulation and solving the resulting MILPs using Gurobi; (ii) using the APP heuristic, with parameters $\varepsilon_{\min} = 0,\, \Delta\varepsilon = -0.1$, and $\delta = 0.1$; (iii) using the probability chains method to solve BO-DOP, using the same parameters as used in APP. We set a soft time limit of 3600 seconds to both build the model and solve it for the prescribed range of penalty values, but with flexibility to allow Gurobi and the MultiObjectiveAlgorithms package to terminate ``gracefully'' when this limit is exceeded. 

We defer the full two-page table of results to the supplementary material (Section C.1), and discuss the key findings here. In general, we found that our implementation of APP for this problem performed very well in terms of runtimes and completeness of the Pareto front. For some of the smaller instances (i.e. those with smaller budgets and therefore easily solved MILPs) we found that the exact method was faster than the APP heuristic, but for the majority of instances APP was much faster. Exact solution times ranged from 0.31 to 3410.58 seconds for times within the time limit, with four instances exceeding this due to the soft time limit: instance $(6,28)$ with 4624.9 seconds, $(6,32)$ with 7196.05 seconds (failing to find even one solution), $(9,56)$ with 7196.93 seconds (again failing to find a solution), and $(14,48)$ with 5974.65 seconds. An inspection of the Gurobi logs revealed that these instances encountered difficult linear relaxations. On the other hand, APP solution times ranged between 0.32 and 6.30 seconds. Generally, for any fixed set of components, the solution time grows drastically for the exact solver as the budget increases, and grows far more slowly for the APP method. This is because designs identified by the APP method generally only have a small number of component types (usually two), meaning that the MDP associated with any given design to be solved in SP2 is quite small. This contrasts with the exact method, which needs to enumerate all possible state-action pairs across all designs.  

Both solution methods find a large number of solutions for problems with larger budgets, despite a visual inspection of the Pareto front suggesting less diversity in the solution space. This is because two distinct Pareto optimal policies may only differ in the actions taken in states that occur with very low probability, meaning that the overall objective value is very similar yet not exactly the same. Instance (14,42) provides an example of this: the exact solution method finds 34 solutions and APP finds 29 solutions; however, a visual inspection of \autoref{fig:instance14-42} would only suggest a total of about 17 distinct solutions across both methods. Note that all orange-colored APP points overlap a blue-colored exact point, and that the one exact solution not found by APP is still very close to an APP solution. To investigate the ability of APP to find the same or very similar solutions to the exact method, we plotted the Pareto fronts of all instances where at least one of the exact Pareto-optimal solutions did not have an APP solution within 2\% of it in terms of both log failure rate (LFR) and operational costs. The majority of these plots proved uninteresting and can be found in the supplementary material (Section C.2), as all but two instances found that all exact solutions had a nearby APP solution upon visual inspection, or showed that the missed exact solution was itself dominated due to numerical error (this started to occur when failure probabilities dropped below $10^{-9}$). The two exceptions were instances (5,12) and (5,24). Instance (5,12) is the only instance with very obvious solutions missed by APP but found by the exact method, and also showcases that one of the solutions found by APP is dominated by an exact solution, as shown in \autoref{fig:instance5-12}. In this case, one of the designs used by the exact method (including the one that dominates an APP solution) uses 1 copy of each component, whereas this design is not found at all by APP. As such, this is a rare case where there is a design that is not Pareto-optimal for BO-DOP, but which lends itself well to dynamic maintenance. As rare as this instance is, it showcases a flaw of the APP method, in that it does not explicitly search for solutions that work well only under exact dynamic policies. Of course, such a thing is hard to quantify, but could indicate an interesting future research direction. A similar but less dramatic example of domination is found in instance (5,24), where the exact method finds a design with 3 copies of the second and third component type, and this design is not found by APP.

The experiment yielded 18 out of 84 instances where at least one DOP solution was dominated by an IDDMP solution. This means that in these instances, there exist design-policy pairs with dynamic policies that outperform an ``optimal'' design with a fully active maintenance policy across both objectives. This suggests that, in some cases, it is both cheaper and more reliable to use a dynamic maintenance policy on a design with more inherent reliability than it is to design a system to strike a specific balance between the two objectives and maintain all components actively. This generally occurs for instances with larger budgets, and therefore more capacity to install a large number of redundant components, giving more flexibility to the dynamic maintenance policies. \autoref{fig:instance2-64} provides the Pareto front of Instance (2,64), in which 4 of the 9 design-only solutions are dominated by at least one IDDMP solution. Details of all problem instances can be found in the supplementary material, Section C.1.

\begin{figure}[htbp!]
    \centering
    \begin{minipage}{.5\textwidth}
    \includegraphics[width=\linewidth]{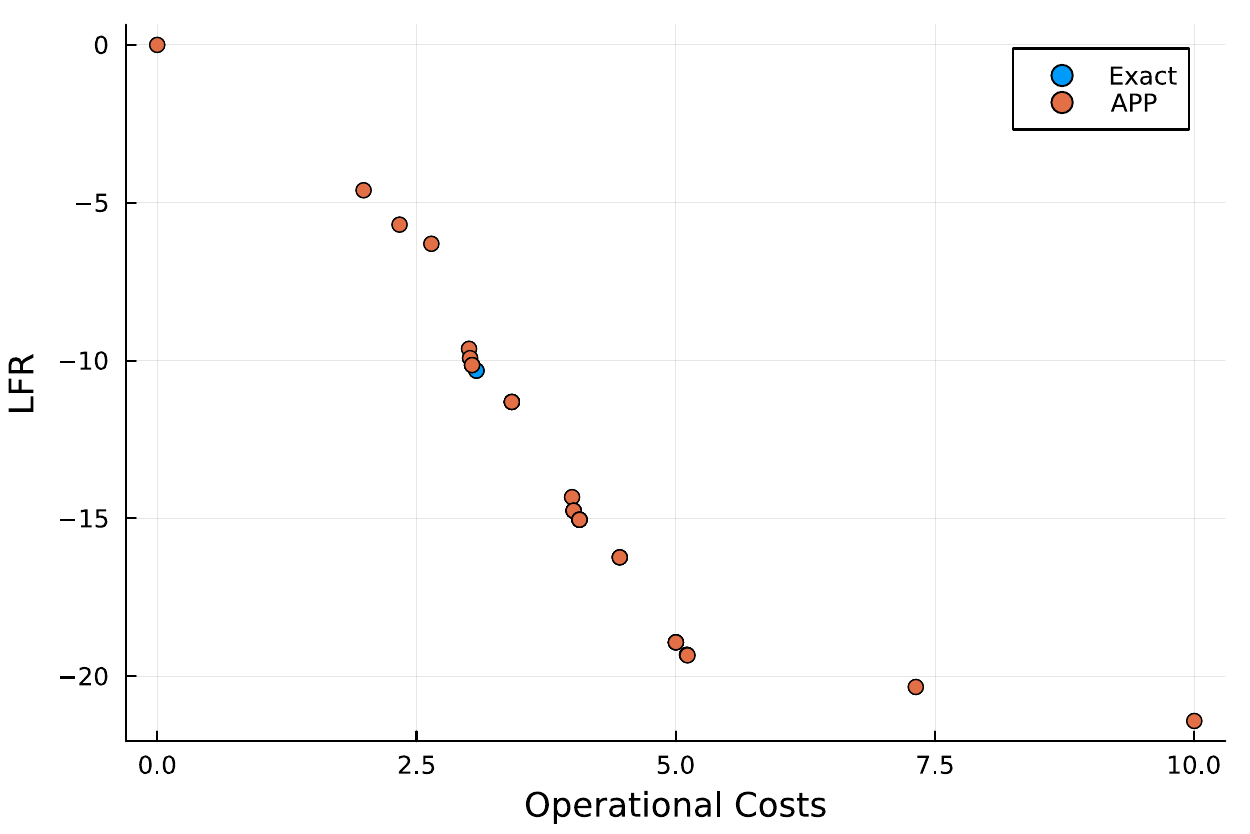}
    \captionof{figure}{Pareto Front for Instance (14,42)}
    \label{fig:instance14-42}   
    \end{minipage}%
    \begin{minipage}{.5\textwidth}
    \includegraphics[width=\linewidth]{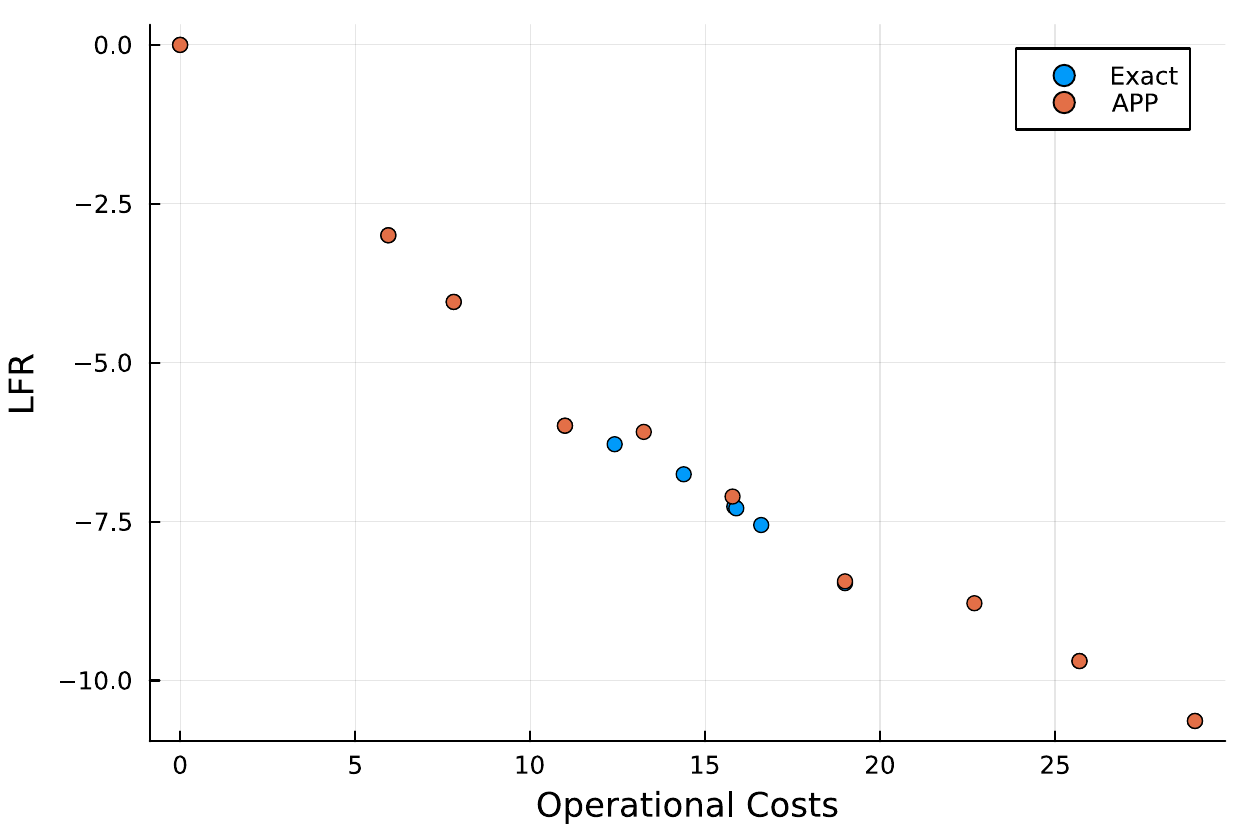}
    \captionof{figure}{Pareto Front for Instance (5,12)}
    \label{fig:instance5-12}
    \end{minipage}
\end{figure}

% \newpage

\begin{figure}[t]
    \centering
    \includegraphics[width=0.7\linewidth]{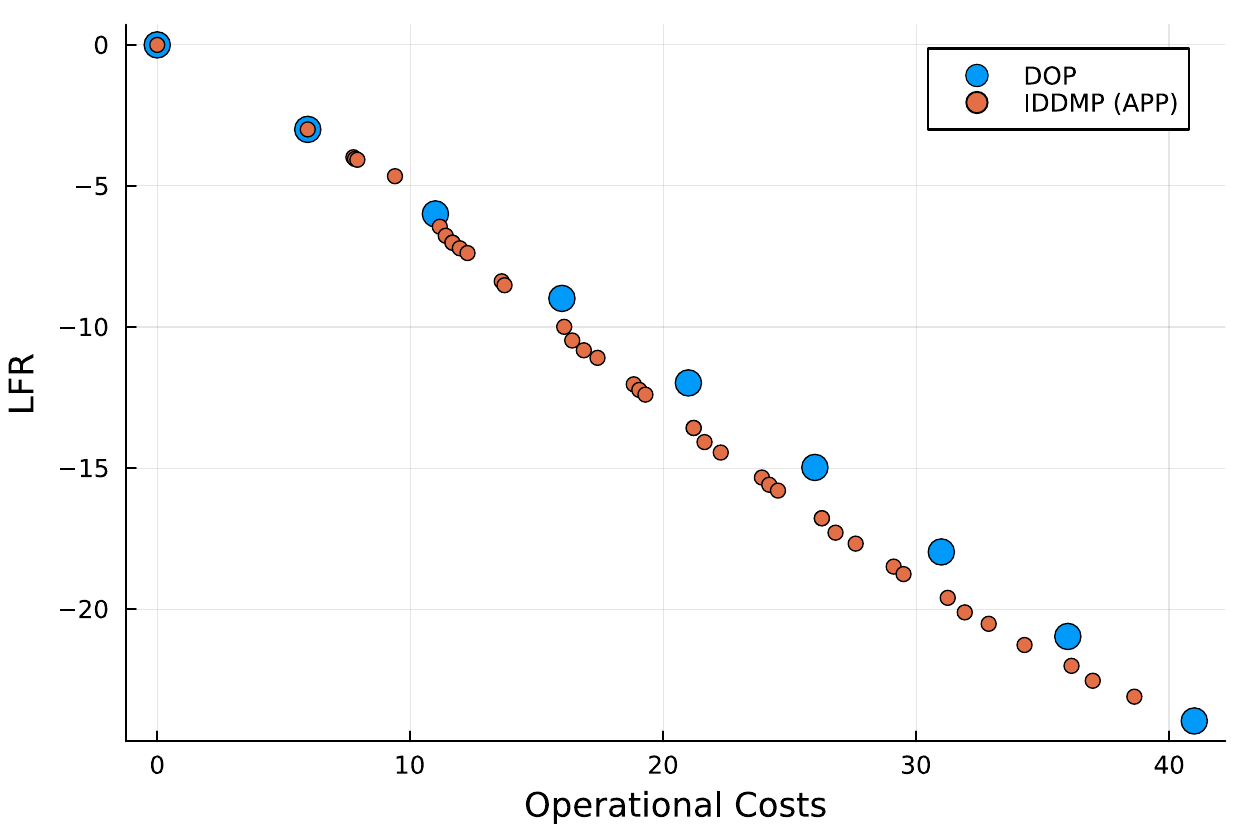}
    \caption{Pareto Front for Instance (2,64)}
    \label{fig:instance2-64}
\end{figure}
}

\subsection{Effect of heterogeneous usage costs}
\label{paper1:study:usage}

\begin{table}[tbp]
    \parbox{.45\linewidth}{
    \centering
    \begin{tabular}{ccccccc}
    \hline
    $i$ & $p$ & $\tau$ & $c$ & $r$ & $C$ & $w$ \\
    1 & 0.99 & 1 & 1 & 100 & 3 & 5\\
    2 & 0.98 & 1 & 1 & 100 & 3 & 4\\
    3 & 0.97 & 1 & 1 & 100 & 2 & 5\\
    4 & 0.96 & 1 & 1 & 100 & 2 & 4\\
    Budget & - & - & - & - & 20 & 20\\
    \hline
    \end{tabular}
    \caption{Base problem parameters.}
    \label{tab:probParams}
    \vspace{1em}

    \begin{tabular}{cccccccc}
    \hline
    $c_1$ & $c_2$ & $x_1$ & $x_2$ & $x_3$ & $x_4$ & $g^o$ & $\ln g^f$\\ 
         \multirow{4}{*}{1} & \multirow{4}{*}{1} & 1 & 0 & 0 & 0 & 1.99 & -4.61\\
         && 2 & 0 & 0 & 0 & 3.00 & -9.21\\
         && 3 & 0 & 0 & 0 & 4.00 & -13.82\\
         && 4 & 0 & 0 & 0 & 5.00 & -18.41\\
         \cmidrule{1-2}
         \multirow{4}{*}{10} & \multirow{4}{*}{1} & 0 & 1 & 0 & 0 & 2.98 & -3.91\\
         &&    1 & 1 & 0 & 0 & 4.18 & -8.52\\
         &&    2 & 1 & 0 & 0 & 5.18 & -13.12\\
         &&    3 & 1 & 0 & 0 & 6.18 & -17.73\\
         \cmidrule{1-2}
         \multirow{5}{*}{10} & \multirow{5}{*}{10} & 0 & 0 & 1 & 0 & 3.97 & -3.51\\
          &&    1 & 0 & 1 & 0 & 5.27 & -8.11\\
          &&    2 & 0 & 1 & 0 & 6.27 & -12.71\\
          &&    3 & 0 & 1 & 0 & 7.27 & -17.32\\
          &&    0 & 4 & 0 & 1 & 13.36 & -18.87\\
        \cmidrule{1-2}
        \multirow{9}{*}{100} & \multirow{9}{*}{100} & 0 & 0 & 1 & 0 & 3.97 & -3.51\\
        &&      0 & 0 & 2 & 0 & 7.00 & -7.01\\
        &&      1 & 0 & 1 & 0 & 7.94 & -8.11\\
        &&      1 & 0 & 2 & 0 & 8.09 & -11.62\\
        &&      2 & 0 & 1 & 0 & 8.97 & -12.72\\
        &&      2 & 0 & 2 & 0 & 9.09 & -16.22\\
        &&      3 & 0 & 1 & 0 & 9.97 & -17.32\\
        &&      0 & 3 & 0 & 2 & 15.16 & -18.17\\
        &&      0 & 4 & 0 & 1 & 16.96 & -18.87\\
        \hline
    \end{tabular}
    \caption{Effect of usage costs on efficient designs.}
    \label{tab:usageCosts}
    }
    \hfill
    \parbox{.45\linewidth}{
\centering
    \begin{tabular}{cccccccc}
    \hline
    $r_1$ & $r_2$ & $x_1$ & $x_2$ & $x_3$ & $x_4$ & $g^o$ & $\ln g^f$\\ 
         \multirow{4}{*}{100} & \multirow{4}{*}{100} %
         & 1 & 0 & 0 & 0 & 1.99 & -4.61\\
         && 2 & 0 & 0 & 0 & 3.00 & -9.21\\
         && 3 & 0 & 0 & 0 & 4.00 & -13.82\\
         && 4 & 0 & 0 & 0 & 5.00 & -18.42\\
         \cmidrule{1-2}
        \multirow{8}{*}{300} & \multirow{8}{*}{100} %
         &  0 & 1 & 0 & 0 & 2.98 & -3.91\\
         && 1 & 0 & 0 & 0 & 3.99 & -4.61\\
         && 0 & 2 & 0 & 0 & 5.00 & -7.82\\
         && 1 & 1 & 0 & 0 & 6.00 & -8.52\\
         && 0 & 3 & 0 & 0 & 7.00 & -11.74\\
         && 1 & 2 & 0 & 0 & 8.00 & -12.43\\
         && 0 & 4 & 0 & 0 & 9.00 & -15.65\\
         && 1 & 3 & 0 & 0 & 10.00 & -16.34\\
         \cmidrule{1-2}
        \multirow{6}{*}{300} & \multirow{6}{*}{300}%
        &  1 & 0 & 0 & 0 & 3.99   & -4.61\\
        && 1 & 0 & 1 & 0 & 6.9997 & -8.11\\
        && 2 & 0 & 0 & 0 & 7.00   & -9.21\\
        && 3 & 0 & 0 & 0 & 10.00  & -13.82\\
        && 4 & 0 & 0 & 0 & 13.00  & -18.42\\
        && 0 & 4 & 0 & 1 & 29.00  & -18.87\\
        \cmidrule{1-2}
        \multirow{12}{*}{500} & \multirow{12}{*}{500}%
         & 0 & 0 & 1 & 0 & 3.97  & -3.51\\
        && 1 & 0 & 0 & 0 & 5.99  & -4.61\\
        && 0 & 0 & 2 & 0 & 7.00  & -7.01\\
        && 1 & 0 & 1 & 0 & 9.00  & -8.11\\
        && 0 & 0 & 3 & 0 & 10.00 & -10.52\\
        && 1 & 0 & 2 & 0 & 12.00 & -11.62\\
        && 0 & 0 & 4 & 0 & 13.00 & -14.03\\
        && 1 & 0 & 3 & 0 & 15.00 & -15.12\\
        && 2 & 0 & 2 & 0 & 17.00 & -16.22\\
        && 3 & 0 & 1 & 0 & 19.00 & -17.32\\
        && 4 & 0 & 0 & 0 & 21.00 & -18.42\\
        && 0 & 4 & 0 & 1 & 45.00 & -18.87\\
        \hline
    \end{tabular}
    \caption{Effect of repair costs on efficient designs.}
    \label{tab:repCosts}
    }
\end{table}

In our model we consider the inclusion of usage costs, corresponding to the costs of using components when there are no cheaper components in healthy condition. Usage costs have not typically been included in previous RAP formulations, and therefore it is useful to investigate their effects on the resulting optimal solutions. In this subsection we consider only the static BO-DOP as opposed to the full BO-IDDMP. This is for two main reasons: firstly, the static BO-DOP is sufficiently novel to demonstrate the effects of usage costs without introducing the complications of dynamic maintenance policies, and secondly, the results in \autoref{paper1:study:big} have already shown that designs found in exact solutions to BO-IDDMP but not found by BO-DOP are somewhat rare. 

Throughout this and subsequent subsections, we focus on problem 6-20 from the problem set given in \autoref{paper1:study:big}. The base parameters for this problem are given in \autoref{tab:probParams}. In this base problem, we see that components 1 and 2 dominate components 3 and 4 respectively in everything except installation costs $C$, which are a weaker constraining factor than weight. As such, we focus on varying $c_1$ and $c_2$, the usage costs of components 1 and 2. We consider different combinations of the cost values 1, 10, and 100. The pair of costs $c_1$ and $c_2$, the design solutions $x$ in general integer form, and the respective two objectives of operational costs $g^o$ and log failure rate $\ln g^f$ are reported in \autoref{tab:usageCosts}. Omitted from the table for each problem is the solution $x = (0,5,0,0)$, which is the solution to F-DOP, and is therefore always included in the front. 

% \begin{table}
%     \centering
%     \begin{tabular}{ccccccc}
%     \hline
%     $i$ & $p$ & $\tau$ & $c$ & $r$ & $C$ & $w$ \\
%     1 & 0.99 & 1 & 1 & 100 & 3 & 5\\
%     2 & 0.98 & 1 & 1 & 100 & 3 & 4\\
%     3 & 0.97 & 1 & 1 & 100 & 2 & 5\\
%     4 & 0.96 & 1 & 1 & 100 & 2 & 4\\
%     Budget & - & - & - & - & 20 & 20\\
%     \hline
%     \end{tabular}
%     \caption{Base problem parameters.}
%     \label{tab:probParams}
% \end{table}

Cost combination $(c_1,c_2) = (1,1)$ has Pareto-optimal solutions which are all multiples of component 1 (excluding the F-DOP solution), which makes sense intuitively as this is the most reliable component. This not only helps in the reliability objective, but also decreases long-run repair costs. For costs $(10,1)$, we see that all solutions excluding F-DOP instead have at least one copy of component 2, and some number of copies (including none) of component 1. This also aligns with our intuition, since component 1 is still the most reliable and the least expensive to maintain. Additionally, just one copy of component 2 means that we do not have to worry about the higher cost-rate of component 1 most of the time, so the effect of this higher value becomes negligible.  

For cost combination $(10,10)$, we see that components 3 and 4 start to become more desirable. The first four solutions are of a similar type to those found for $(10,1)$, as they all have one copy of a component with a low usage cost - which offsets the long-run impact of all other usage costs in the solution - and then some number of copies of component 1. We also have the solution $x = (0,4,0,1)$. This solution sees an interplay between the two objectives and the weight budget. We would like to add an extra component to $(3,0,1,0)$ for further reliability, but this would exceed our weight budget, so we instead have to use the lighter components to achieve a similar solution. This solution still follows the same pattern of having one component with a low usage cost and a number of copies of a higher reliability component.

The final combination $(100,100)$ is a very extreme example, because in this case the usage costs and repair costs for components 1 and 2 are equal. Despite this, they do not go unused, because they can be used alongside the cheaper-to-use components 3 and 4, so the extreme usage cost has a low impact. We now see solutions such as $(1,0,2,0)$ and $(2,0,2,0)$, which include two copies of the cheaper component as opposed to the one copy we saw for previous problems. The interpretation of this is not quite as clear, but it could be due to the fact that multiple copies of a cheap-to-use component further offset the contribution to the long-run average cost of the expensive-to-use yet reliable component 1. An explainable managerial insight here is that, aside from the choice of some ``primary'' cheap-to-use component type, the usage costs of backup components may be seen as a negligible consideration outside of extreme cases.

\subsection{Effect of heterogeneous repair costs}
\label{paper1:study:repairs}

\begin{figure}[tbp]
    \centering
    \begin{minipage}{.5\textwidth}
    \centering
    \includegraphics[width=\linewidth]{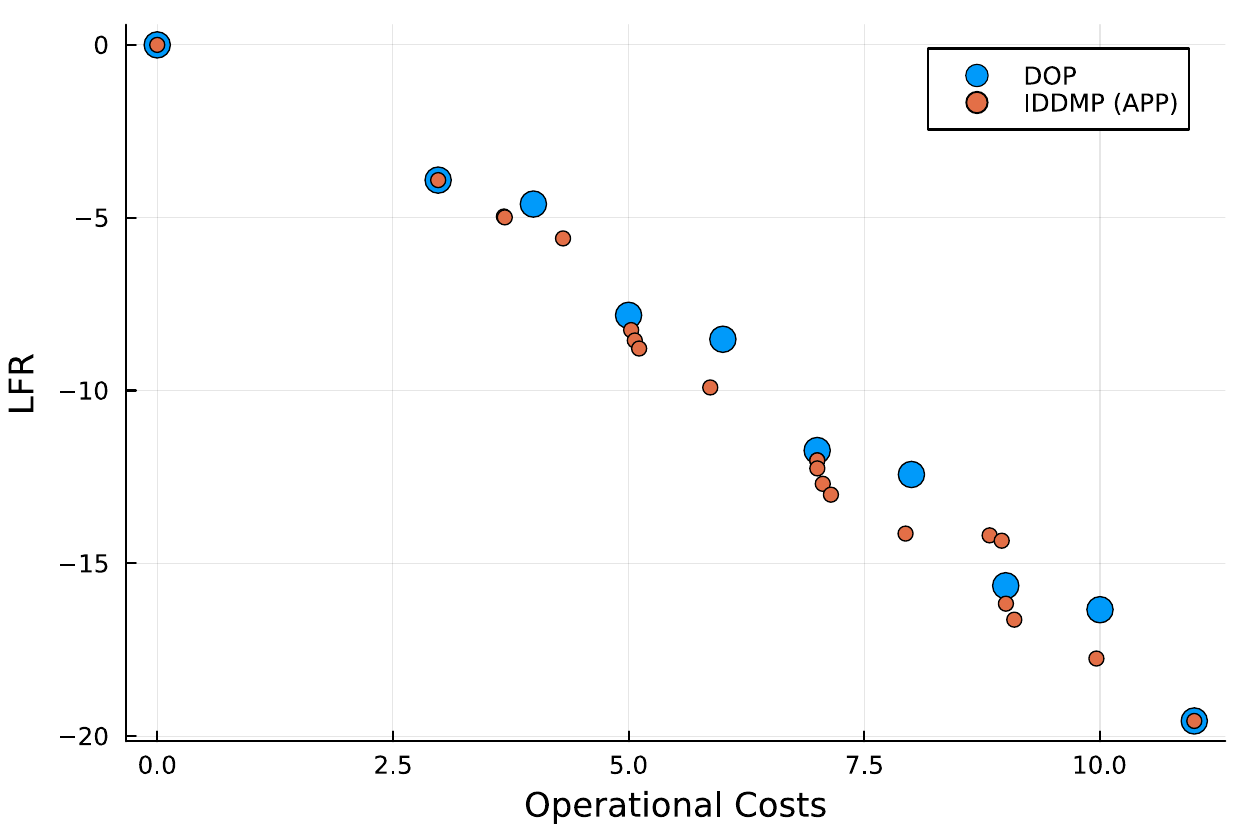}
    \captionof{figure}{\centering Pareto fronts of DOP and \\IDDMP solutions for $(r_1,r_2) = (300,100)$.}
    \label{fig:r=300-100}   
    \end{minipage}%
    \begin{minipage}{0.5\textwidth}
    \centering
    \includegraphics[width=\linewidth]{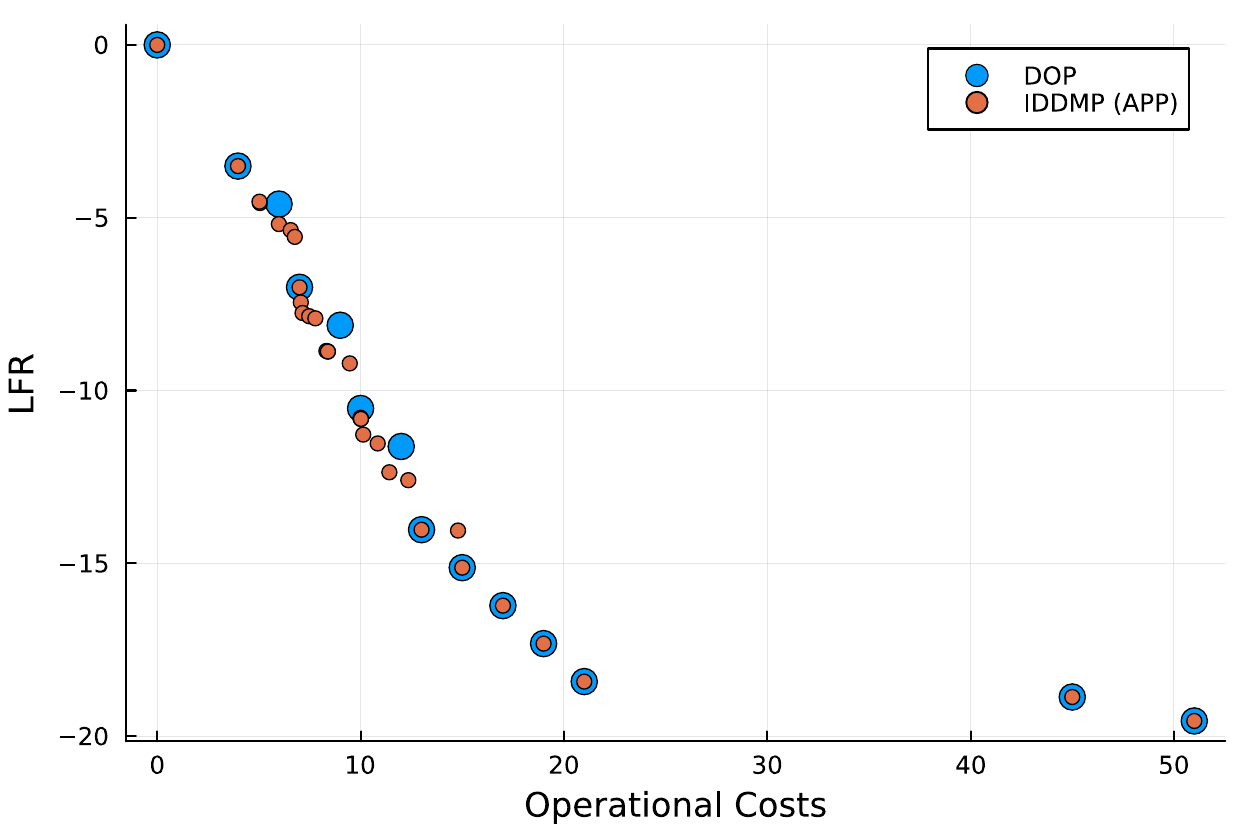}
    \captionof{figure}{\centering Pareto fronts of DOP and \\IDDMP solutions for $(r_1,r_2) = (500,500)$.}
    \label{fig:r=500-500}   
    \end{minipage}
\end{figure}

In addition to usage costs, another aspect of our model that differentiates it from much of the existing RAP literature is the inclusion of heterogeneous repair costs, whereby different components have different repair costs per unit time. We are aware of only one previous study that makes similar assumptions about repair costs \citep{LINS2011manyCost}, although several other works consider repair costs attributable to repair workers being assigned to subsystems. In this subsection we explore the impact of these costs on solutions to the problem, again using instance 6-20 from \autoref{paper1:study:big} as a basis. We again vary the repair cost-rate parameters of components 1 and 2, $r_1$ and $r_2$. \autoref{tab:repCosts} reports the repair costs being investigated, their Pareto-optimal BO-DOP solutions, and the objective values of these solutions under the fully active repair policy. 

The case $(r_1,r_2) = (100,100)$ is simply a re-run of the base problem. Case $(300,100)$ demonstrates the effect of tripling the repair cost-rate of the first component. Immediately of note is the fact that we obtain more design solutions. These can be split into two groups: those with one copy of component 1, and those with no copies of component 1. In the DOP case this is not particularly notable, but in the IDDMP case we see that all designs that use a copy of component 1 are dominated by a design without this component, when used alongside a dynamic maintenance policy. This demonstrates another situation where dynamic maintenance policies not only fill in the Pareto front, but actually dominate some of the static solutions. This is evident from the plot of the Pareto front in \autoref{fig:r=300-100}. Notably, this situation is not unrealistic, as it could be quite reasonable that a more reliable component is more expensive to perfectly repair or replace when it does fail. 

Case $(300,300)$ still primarily shows preference towards the first component, similarly to the base problem. Particularly of note within this problem is that it highlights an idiosyncrasy of our model. Solutions $(1,0,1,0)$ and $(2,0,0,0)$ are very similar. Their operational costs due to repairs are identical, as $r_1q_1 = 300(1-0.99) = 3$ and $r_3q_3 = 100(1-0.97) = 3$, and the usage costs of all components are identical. However, design $(1,0,1,0)$ is slightly less reliable than $(2,0,0,0)$, and therefore it incurs usage costs slightly less often and is not dominated by the other solution. In reality, it is clear that design $(1,0,1,0)$ should not be chosen as its usage costs are almost identical to those of $(2,0,0,0)$, and the latter solution is more reliable.

Case $(500,500)$ has a very populated Pareto front for the BO-DOP, with none of these solutions being idiosyncratic as described in the previous case, and with only two of these designs, namely $(1,0,1,0)$ and $(1,0,2,0)$, being dominated by a dynamic solution (see \autoref{fig:r=500-500}). The rest of the solutions are reasonably well spaced in terms of both objectives, with large jumps in cost only occurring between $(4,0,0,0)$ and $(0,4,0,1)$, and from that solution to $(0,5,0,0)$, which has a cost of 51. We observe that these solutions are quite diverse in their choices of component, with component 3 now becoming more preferable across the solution set.

{\subsection{Managerial Insights}
We can briefly summarize some of our findings at the qualitative level, in a way that could provide applicable advice for redundant systems more complex than those explored here. The first thing to note, which is apparent from \autoref{tab:usageCosts} and \autoref{tab:repCosts} and which was common across all experiments, is that the optimal system designs generally only included two types of component. We can interpret this as including a ``primary'' component that is better in terms of usage costs, and a ``secondary'' component which is better in terms of reliability and maintenance costs. At an intuitive level, it makes sense that this would hold for many systems with redundancy, where one type of component is included because it has a good balance between all factors including usage costs (where this cheaper usage cost is the one that will be incurred most of the time), and another type will be included for reasons other than its usage cost (i.e. where we don't mind occasionally paying this higher usage cost, because this won't happen very often). Of course, it may sometimes be the case that the optimal design of a system includes more types of component (in the case of awkward knapsack constraints, for example), but this two-component-type design principle seems reasonable as a baseline. 

The second aspect to summarize is the instances where dynamic policies seem to work the best. Across all instances tested in \autoref{paper1:study:big}, we found that dynamic policies always provided a richer and more complete set of Pareto-optimal solutions than the design-only problem, which is restricted to the ``always repair'' maintenance policy. This shows that dynamic policies can ``interpolate'' between designs in terms of their reliability and operational costs, allowing decision-makers to strike a more specific balance between the two. This is particularly useful when there are large gaps in the Pareto front of the design-only problem. We also found a number of instances where a dynamic policy dominated a design-only solution, and noted that this only occurred in problem instances with a larger budget, and therefore more scope for extra redundancy. {\color{black} A key takeaway here is that dynamic maintenance policies are at their best (i.e. they can dominate static policies) for systems with lots of redundant components, and this is likely due to the inherent leeway granted by such systems. This is best demonstrated by \autoref{fig:instance2-64}, which shows many static solutions dominated by dynamic solutions. Another insight is that the dynamic solutions show a more sub-linear relationship between costs and LFR compared to the mostly linear relationship between the two for static solutions. This indicates that dynamic solutions can obtain cost savings with a more favorable, less pronounced decrease in reliability. Again, this seems more likely for systems with more redundant components.} In real terms, this suggests that it is potentially worthwhile to install additional redundancy into a system and maintain it using a dynamic maintenance policy, especially in situations with more relaxed physical constraints, and where long-run costs and reliability are of greater concern than upfront costs. {\color{black} This additional redundancy could also be useful in situations where repairs for some reason cannot be performed straight away, for example due to repair worker availability or budgetary issues.} %, however this is a more general insight and not one derived from our model or experiments.} 

% note is that, as discovered in \autoref{paper1:study:big}, large variations between components in the rates at which events occur can noticeably affect the performance of dynamic maintenance policies. From an intuitive perspective, if we have a component that can be repaired quickly compared to the failure or repair times for another component (perhaps parts or services are more readily available from the manufacture, or on-site engineers are more familiar with the component), then it may seem more reasonable to wait until the overall system gets into a worse state of repair before repairing that component. This is because the probability of the system failing in the time it takes to repair such a component is low, or at least comparatively low compared to other components without this feature. This becomes even more clear if such a feature also fails relatively often, because a ``repair-straight-away'' approach to such a component would be quite wasteful over a long period of time, when it might be more reasonable to just repair such a component when it seems like it might be needed. One can imagine such a component being an older model or a cheaper yet common component which can be quickly be restored by familiar engineers with common parts in a timely fashion when it is absolutely needed (i.e. when the overall system seems close to failure), but which under normal operating condition may not be a maintenance priority.

While these managerial insights seem to make sense at the high level, it is important to note that we have only discussed parallel systems under simplifying assumptions such as hot standby, perfect switching, and exponential failure and repair times. If this model is used to obtain system designs and dynamic maintenance policies, additional work must be done to ``sanity check'' the suggested solutions with respect to the real-world system, and to verify the performance of a design and/or maintenance strategy on a model with more of the important real-world details, e.g. by using simulation.
}

In this section we have demonstrated that our heuristic is capable of providing a multitude of high-quality solutions to the BO-IDDMP in far less time than the exact method, and have showcased and interpreted the effects of usage and repair costs on particular solutions. {We have also found that dynamic solutions greatly outnumber static solutions, and in some cases dominate static solutions. This demonstrates that the integration of strategic design decisions with operational maintenance decisions in a unified model is a worthwhile endeavor.}

\section{Conclusion}\label{paper1:conclusion}
In this paper we have introduced a novel BO-MDP Design extension of the RAP, with parameters such as usage and repair costs which are seldom seen in the RAP literature. Due to the complexity of the resulting problem, we constructed a bespoke heuristic to obtain an approximate Pareto front. We compared the performance of this heuristic to that of an exact method for small problem instances inspired by earlier research, and found that our heuristic was much faster than the exact method and also found more solutions due to its reduced sensitivity to parameter values, and that the solutions found were optimal in the majority of cases. We also demonstrated the advantages of considering dynamic maintenance policies, showing that they lead to better populated Pareto fronts, and sometimes find design-policy pairs that dominate static solutions. We also investigated the effects of varying the usage costs and repair costs, and in all cases were able to provide intuitive interpretations of the results.

{While our focus on parallel systems is arguably quite restrictive, we believe that this work offers an important first step in integrating system design and dynamic maintenance decisions. Extending to the series-parallel case is non-trivial, as the number of states in the MDP grows exponentially with the number of subsystems, even before adding in design decisions. As such, extending the work presented in this paper to the series-parallel case offers an interesting direction for future research. We believe that such systems could be decomposed into parallel system subproblems via methods such as Dantzig-Wolfe decomposition, allowing us to use the methods in this paper to solve those subproblems.} There are also many other variations of the RAP as identified in \autoref{paper1:intro:rap}, such as $k$-out-of-$n$ systems or complex systems, cold or mixed standby with imperfect switching, the allocation of repair teams, and multi-state components. A limitation of the model employed in this paper is the need to assume exponential failure and repair times in order to use an MDP formulation. Further work could be done to evaluate the performances of solutions to our model on more realistic systems with non-exponential event times, or one could go further and integrate general time distributions into the optimization model directly. {Another limitation of our model is the hot-standby assumption, where components always have the same constant rate of failure, regardless of whether or not they are being used. The MDP model could be extended to the cold- or warm-standby cases by including switching decisions whenever the in-use component fails, or whenever a component comes back online. Switching decisions would aim to find a balance between usage costs and the reliability and future maintenance costs of available components.} On a final note, further work could be done to investigate how to improve upon the heuristic presented in this paper, or to construct an efficient algorithm for finding exact solutions.

From the methodological point of view, we have contributed to the emerging MDP Design problem. This is a very young and sparse area of literature, and developing the theory of this field should be important and fruitful, due to the potential of applying the model to other real-world problems such as inventory and queuing problems \citep{BROWN2024mdpDesignLetters}. %and coming up with better ways to solve these sorts of problem.
We have also introduced the APP method, where for bi-objective MDP Design problems we first obtain a set of first-stage decisions corresponding to a Pareto front based on simplifying assumptions on the second-stage decisions, and then obtain a new Pareto front for each of these first-stage decisions by optimizing the second-stage decisions. It remains to see how well this methodology works for other MDP Design problems.  

{As we have presented APP as a general heuristic framework for bi-objective MDP Design problems, an interesting research direction would explore how APP could work in other settings. For example, one could consider a queuing design and control problem. The design aspects could include the hiring of servers with different service rates or the allocation of limited waiting space to different queues, and the second-stage decisions could include the routing of customers to queues. Objectives could be based on server hiring costs, overall customer satisfaction (e.g. due to waiting times), satisfaction of different customer types, etc. A parameterized family of heuristics could be based on static stochastic routing policies, i.e. state-agnostic random routing of customers to queues. These are well-known to provide good baseline policies for queuing systems, which can be improved upon by dynamic polices \citep{shone2020queueing}. The design problem would then find different designs that balance between the different objectives, and the second stage would find deterministic dynamic policies over each of the designs. Queuing control itself is a hard problem; however, dynamic heuristics can be used that are known to improve upon static stochastic policies \citep{krishnan1990,argon2009,shone2020queueing}.} 

{While the APP method has been successful in our study, we acknowledge that in general APP has no iterative element, and cannot explore solutions beyond its two stages. As such, the development of metaheuristic methods for MDP Design problems remains an interesting research direction.}

\subsection{Acknowledgements}
{\color{black} We are grateful for the constructive feedback and suggestions provided by the anonymous reviewers, which have helped to strengthen this paper.} This paper is based on work completed at the EPSRC funded STOR-i Centre for Doctoral Training (EP/S022252/1).
% In this paper, we also introduce the APP method, and we believe that this could be generalised to other bi-objective two-stage problems. In general, this method obtains a set of first-stage decisions corresponding to a Pareto front based on simplifying assumptions on the second-stage decisions, then obtains a new Pareto front for each of these first-stage decisions by varying the second-stage decisions. While this is by no means the first population-based approach to two-stage bi-objective optimisation, as metaheuristics are very often population-based, we believe it is the first population-based matheuristic for such problems.

\bibliographystyle{apalike}
%\small
\bibliography{bib}

@article{krishnan1990,
title = {Joining the right queue: a state-dependent decision rule},
journal = {IEE Transactions on Automatic Control},
volume = {35},
pages = {104-108},
year = {1990},
author = {K R Krishnan},
}

@article{argon2009,
title = {Dynamic routing of customers with general delay costs
in a multiserver queuing system},
journal = {Probability in the Engineering and Informational Sciences},
volume = {23},
number = {2},
pages = {175-203},
year = {2009},
author = {N T Argon and L Ding and K D Glazebrook and S Ziya},
}

@article{CHERN1992309,
title = {On the computational complexity of reliability redundancy allocation in a series system},
journal = {Operations Research Letters},
volume = {11},
number = {5},
pages = {309-315},
year = {1992},
issn = {0167-6377},
doi = {https://doi.org/10.1016/0167-6377(92)90008-Q},
url = {https://www.sciencedirect.com/science/article/pii/016763779290008Q},
author = {Maw-Sheng Chern}
}

@article{REIHANEH20221112,
title = {An exact algorithm for the redundancy allocation problem with heterogeneous components under the mixed redundancy strategy},
journal = {European Journal of Operational Research},
volume = {297},
number = {3},
pages = {1112-1125},
year = {2022},
issn = {0377-2217},
doi = {https://doi.org/10.1016/j.ejor.2021.06.033},
url = {https://www.sciencedirect.com/science/article/pii/S037722172100549X},
author = {Mohammad Reihaneh and Mostafa {Abouei Ardakan} and Majid Eskandarpour},
}

@ARTICLE{ZIA2010,
  author={Zia, Leila and Coit, David W.},
  journal={IEEE Transactions on Reliability}, 
  title={Redundancy Allocation for Series-Parallel Systems Using a Column Generation Approach}, 
  year={2010},
  volume={59},
  number={4},
  pages={706-717},
  doi={10.1109/TR.2010.2085530}}

@article{KIM201864,
title = {Maximization of system reliability with the consideration of component sequencing},
journal = {Reliability Engineering \& System Safety},
volume = {170},
pages = {64-72},
year = {2018},
issn = {0951-8320},
doi = {https://doi.org/10.1016/j.ress.2017.10.020},
url = {https://www.sciencedirect.com/science/article/pii/S0951832017301813},
author = {Heungseob Kim}
}

@article{OHANLEY201363,
title = {Probability chains: A general linearization technique for modeling reliability in facility location and related problems},
journal = {European Journal of Operational Research},
volume = {230},
number = {1},
pages = {63-75},
year = {2013},
issn = {0377-2217},
doi = {https://doi.org/10.1016/j.ejor.2013.03.021},
url = {https://www.sciencedirect.com/science/article/pii/S037722171300249X},
author = {Jesse R. O’Hanley and M. Paola Scaparra and Sergio García}
}

@book{Puterman1994MarkovDP,
  title={Markov Decision Processes: Discrete Stochastic Dynamic Programming},
  author={Martin L. Puterman},
  publisher={Wiley \& Sons, New Jersey},
  year={1994}
}

@book{powell2011adp,
    author = {Warren B. Powell},
    title = {Approximate Dynamic Programming: Solving the Curses of Dimensionality, 2nd Edition},
    year = {2011},
    publisher = {Wiley \& Sons, New Jersey}
}

@article{roijers2013momdp_survey,
  title={A Survey of Multi-Objective Sequential Decision-Making},
  author={Roijers, Diederik M and Vamplew, Peter and Whiteson, Shimon and Dazeley, Richard},
  journal={Journal of Artificial Intelligence Research},
  volume={48},
  pages={67--113},
  year={2013}
}

@ARTICLE{fyffe1968problemset,
  author={Fyffe, David E. and Hines, William W. and Lee, Nam Kee},
  journal={IEEE Transactions on Reliability}, 
  title={System Reliability Allocation and a Computational Algorithm}, 
  year={1968},
  volume={17},
  number={2},
  pages={64-69},
  doi={10.1109/TR.1968.5217517}}

@book{ehrgott2005multicriteria,
  title={Multicriteria Optimization},
  author={Ehrgott, Matthias},
  year={2005},
  publisher={Springer Science \& Business Media}
}

@book{xianping2009continuous,
  title={Continuous-time markov decision processes: theory and applications},
  author={Xianping. Guo and Hern{\'a}ndez-Lerma, O},
  publisher={Springer-Verlag},
    year = {2009}
}

@article{bellman1957markovian,
  title={A {M}arkovian decision process},
  author={Bellman, Richard},
  journal={Journal of Mathematics and Mechanics},
  volume = {6},
  number = {5},
  pages={679--684},
  year={1957},
}

@book{howard1960dynamic,
  title={Dynamic programming and {M}arkov processes.},
  author={Howard, Ronald A},
  year={1960},
  publisher={MIT Press, Cambridge, MA}
}

@article{epenoux1963mdpLP,
 ISSN = {00251909, 15265501},
 URL = {http://www.jstor.org/stable/2627210},
 author = {F. d'Epenoux},
 journal = {Management Science},
 number = {1},
 pages = {98--108},
 publisher = {INFORMS},
 title = {A Probabilistic Production and Inventory Problem},
 urldate = {2024-06-14},
 volume = {10},
 year = {1963}
}

@book{sutton2018reinforcement,
  title={Reinforcement learning: An introduction},
  author={Sutton, Richard S and Barto, Andrew G},
  year={2018},
  publisher={MIT Press, Cambridge, MA}
}

@article{nunes2009markovpatient,
  title={Markov decision process applied to the control of hospital elective admissions},
  author={Nunes, Luiz Guilherme Nadal and de Carvalho, Solon Venancio and Rodrigues, Rita de C{\'a}ssia Meneses},
  journal={Artificial Intelligence in Medicine},
  volume={47},
  number={2},
  pages={159--171},
  year={2009},
  publisher={Elsevier}
}

@article{chen2005conditionbased,
  title={Optimization for condition-based maintenance with semi-{M}arkov decision process},
  author={Chen, Dongyan and Trivedi, Kishor S},
  journal={{Reliability Engineering \& System Safety}},
  volume={90},
  number={1},
  pages={25--29},
  year={2005},
  publisher={Elsevier}
}

@inproceedings{amari2006conditionbased,
  title={Cost-effective condition-based maintenance using {M}arkov decision processes},
  author={Amari, Suprasad V and McLaughlin, Leland and Pham, Hoang},
  booktitle={RAMS'06. Annual Reliability and Maintainability Symposium, 2006.},
  pages={464--469},
  year={2006},
  organization={IEEE}
}

@article{GUO2022pomdp,
title = {A predictive {M}arkov decision process for optimizing inspection and maintenance strategies of partially observable multi-state systems},
journal = {Reliability Engineering \& System Safety},
volume = {226},
pages = {108683},
year = {2022},
issn = {0951-8320},
doi = {https://doi.org/10.1016/j.ress.2022.108683},
url = {https://www.sciencedirect.com/science/article/pii/S0951832022003167},
author = {Chunhui Guo and Zhenglin Liang}
}

@article{DEEP2023pomdp,
title = {Partially observable {M}arkov decision process-based optimal maintenance planning with time-dependent observations},
journal = {European Journal of Operational Research},
volume = {311},
number = {2},
pages = {533-544},
year = {2023},
issn = {0377-2217},
doi = {https://doi.org/10.1016/j.ejor.2023.05.022},
url = {https://www.sciencedirect.com/science/article/pii/S0377221723003867},
author = {Akash Deep and Shiyu Zhou and Dharmaraj Veeramani and Yong Chen}
}

@article{flynn2004heuristic,
  title={A heuristic algorithm for determining replacement policies in consecutive k-out-of-n systems},
  author={Flynn, James and Chung, Chia-Shin},
  journal={Computers \& Operations Research},
  volume={31},
  number={8},
  pages={1335--1348},
  year={2004},
  publisher={Elsevier}
}

@book{bertsekas2012dynamic,
  title={Dynamic programming and optimal control: Volume I},
  author={Bertsekas, Dimitri},
  year={2012},
  publisher={Athena Scientific}
}

@article{ZOULFAGHARI2014biObjRepairable,
title = {Bi-objective redundancy allocation problem for a system with mixed repairable and non-repairable components},
journal = {ISA Transactions},
volume = {53},
number = {1},
pages = {17-24},
year = {2014},
issn = {0019-0578},
doi = {https://doi.org/10.1016/j.isatra.2013.08.002},
url = {https://www.sciencedirect.com/science/article/pii/S0019057813001328},
author = {Hossein Zoulfaghari and Ali {Zeinal Hamadani} and Mostafa {Abouei Ardakan}}
}

@article{KAYEDPOUR2017repairableMarkovFiniteHorizon,
title = {Multi-objective redundancy allocation problem for a system with repairable components considering instantaneous availability and strategy selection},
journal = {Reliability Engineering \& System Safety},
volume = {160},
pages = {11-20},
year = {2017},
issn = {0951-8320},
doi = {https://doi.org/10.1016/j.ress.2016.10.009},
url = {https://www.sciencedirect.com/science/article/pii/S0951832016306573},
author = {Farjam Kayedpour and Maghsoud Amiri and Mahmoud Rafizadeh and Arash {Shahryari Nia}}
}

@article{tavana2018multistate,
  title={An evolutionary computation approach to solving repairable multi-state multi-objective redundancy allocation problems},
  author={Tavana, Madjid and Khalili-Damghani, Kaveh and Di Caprio, Debora and Oveisi, Zeynab},
  journal={Neural Computing and Applications},
  volume={30},
  pages={127--139},
  year={2018},
  publisher={Springer}
}

@article{LINS2011manyCost,
title = {Redundancy allocation problems considering systems with imperfect repairs using multi-objective genetic algorithms and discrete event simulation},
journal = {Simulation Modelling Practice and Theory},
volume = {19},
number = {1},
pages = {362-381},
year = {2011},
issn = {1569-190X},
doi = {https://doi.org/10.1016/j.simpat.2010.07.010},
url = {https://www.sciencedirect.com/science/article/pii/S1569190X10001723},
author = {Isis Didier Lins and Enrique López Droguett}
}

@inproceedings{dimitrov2009mdpDesign,
  title={Combinatorial design of a stochastic {M}arkov decision process},
  author={Dimitrov, Nedialko B and Morton, David P},
  booktitle={Operations Research and Cyber-Infrastructure},
  pages={167--193},
  year={2009},
  organization={Springer}
}

@article{DIMITROV2013mdpDesignMalaria,
title = {Selecting malaria interventions: A top-down approach},
journal = {Computers \& Operations Research},
volume = {40},
number = {9},
pages = {2229-2240},
year = {2013},
issn = {0305-0548},
doi = {https://doi.org/10.1016/j.cor.2011.07.023},
url = {https://www.sciencedirect.com/science/article/pii/S0305054811002164},
author = {Nedialko B. Dimitrov and Alexander Moffett and David P. Morton and Sahotra Sarkar}
}

@article{BROWN2024mdpDesignLetters,
title = {Markov decision process design: A framework for integrating strategic and operational decisions},
journal = {Operations Research Letters},
volume = {54},
pages = {107090},
year = {2024},
issn = {0167-6377},
doi = {https://doi.org/10.1016/j.orl.2024.107090},
url = {https://www.sciencedirect.com/science/article/pii/S0167637724000269},
author = {Seth Brown and Saumya Sinha and Andrew J. Schaefer}
}

@article{KESHAVARZGHORABAEE2015kOutOfN,
title = {Genetic algorithm for solving bi-objective redundancy allocation problem with k-out-of-n subsystems},
journal = {Applied Mathematical Modelling},
volume = {39},
number = {20},
pages = {6396-6409},
year = {2015},
issn = {0307-904X},
doi = {https://doi.org/10.1016/j.apm.2015.01.070},
url = {https://www.sciencedirect.com/science/article/pii/S0307904X15000736},
author = {Mehdi {Keshavarz Ghorabaee} and Maghsoud Amiri and Parham Azimi},
}

@article{kulturelkonak2003nonRep,
author = {Sadan Kulturel-Konak and Alice E. Smith and David W. Coit},
title = {Efficiently Solving the Redundancy Allocation Problem Using Tabu Search},
journal = {IIE Transactions},
volume = {35},
number = {6},
pages = {515--526},
year = {2003},
publisher = {Taylor \& Francis},
doi = {10.1080/07408170304422},
URL = {https://doi.org/10.1080/07408170304422},
eprint = {https://doi.org/10.1080/07408170304422}
}

@misc{gurobi,
  author = {{Gurobi Optimization, LLC}},
  title = {{Gurobi Optimizer Reference Manual}},
  year = 2024,
  url = "https://www.gurobi.com"
}

@article{dai2019inpatient,
    author = {J. G. Dai and Pengyi Shi},
    title = {Inpatient Overflow: An Approximate Dynamic Programming Approach},
    journal = {Manufacturing \& Service Operations Management},
    volume = {21},
    number = {4},
    pages = {894--911},
    year =  {2019}
}

@article{lim2012facility,
    author = {Michael K. Lim and Achal Bassamboo and Sunil Chopra and Mark S. Daskin},
    title = {Facility Location Decisions with Random Disruptions and Imperfect Estimation},
    journal = {Manufacturing \& Service Operations Management},
    year = {2012},
    volume = {15},
    number = {2},
    pages = {239--249}
}

@article{lim2010facility,
author = {Lim, Michael and Daskin, Mark S. and Bassamboo, Achal and Chopra, Sunil},
title = {A facility reliability problem: Formulation, properties, and algorithm},
journal = {Naval Research Logistics},
volume = {57},
number = {1},
pages = {58-70},
doi = {https://doi.org/10.1002/nav.20385},
url = {https://onlinelibrary.wiley.com/doi/abs/10.1002/nav.20385},
eprint = {https://onlinelibrary.wiley.com/doi/pdf/10.1002/nav.20385},
year = {2010}
}

@article{adusumilli2010queue,
    author = {Adusumilli, K.M. and Hasenbein, J.J.},
    title = {Dynamic admission and service rate control of a queue},
    journal = {Queueing Systems},
    volume = {66},
    pages = {131-154},
    year = {2010}
}

@article{park2020complexInforms,
author = {Park, Young Woong},
title = {{MILP} Models for Complex System Reliability Redundancy Allocation with Mixed Components},
journal = {INFORMS Journal on Computing},
volume = {32},
number = {3},
pages = {600-619},
year = {2020},
doi = {10.1287/ijoc.2019.0895},
URL = {https://doi.org/10.1287/ijoc.2019.0895},
eprint = {https://doi.org/10.1287/ijoc.2019.0895}
}

@article{ozekici1988replacement,
ISSN = {0030364X, 15265463},
URL = {http://www.jstor.org/stable/171132},
author = {Suleyman Ozekici},
journal = {Operations Research},
number = {4},
pages = {542--552},
publisher = {INFORMS},
title = {Optimal Periodic Replacement of Multicomponent Reliability Systems},
urldate = {2024-07-19},
volume = {36},
year = {1988}
}

@article{majety1999reliability,
author = {Majety, Subba Rao V. and Dawande, Milind and Rajgopal, Jayant},
title = {Optimal Reliability Allocation with Discrete Cost-Reliability Data for Components},
journal = {Operations Research},
volume = {47},
number = {6},
pages = {899-906},
year = {1999},
doi = {10.1287/opre.47.6.899},
URL = {https://doi.org/10.1287/opre.47.6.899},
eprint = {https://doi.org/10.1287/opre.47.6.899}
}

@article{WHITE1982momdpNonStationary,
title = {Multi-objective infinite-horizon discounted Markov decision processes},
journal = {Journal of Mathematical Analysis and Applications},
volume = {89},
number = {2},
pages = {639-647},
year = {1982},
issn = {0022-247X},
doi = {https://doi.org/10.1016/0022-247X(82)90122-6},
url = {https://www.sciencedirect.com/science/article/pii/0022247X82901226},
author = {D.J White}
}

@article{shone2020queueing,
  title={A conservative index heuristic for routing problems with multiple heterogeneous service facilities},
  author={Shone, R. and Knight, V.A. and Harper, P.R.},
  journal={Mathematical Methods of Operations Research},
  year={2020},
  volume = {92},
  pages = {511-543},
  publisher={Springer}
}

@article{drexl2015locationRoutingsurvey,
  title={A survey of variants and extensions of the location-routing problem},
  author={Drexl, Michael and Schneider, Michael},
  journal={European Journal of Operational Research},
  volume={241},
  number={2},
  pages={283--308},
  year={2015},
  publisher={Elsevier}
}

@article{pillac2013dvrpSurvey,
  title={A review of dynamic vehicle routing problems},
  author={Pillac, Victor and Gendreau, Michel and Gu{\'e}ret, Christelle and Medaglia, Andr{\'e}s L},
  journal={European Journal of Operational Research},
  volume={225},
  number={1},
  pages={1--11},
  year={2013},
  publisher={Elsevier}
}

@article{BEI2019designAndMaintenanceRiskAverse,
title = {A risk-averse stochastic program for integrated system design and preventive maintenance planning},
journal = {European Journal of Operational Research},
volume = {276},
number = {2},
pages = {536-548},
year = {2019},
issn = {0377-2217},
doi = {https://doi.org/10.1016/j.ejor.2019.01.038},
url = {https://www.sciencedirect.com/science/article/pii/S0377221719300736},
author = {Xiaoqiang Bei and Xiaoyan Zhu and David W. Coit}
}

@ARTICLE{bei2017designAndMaintenanceRAP,
  author={Bei, Xiaoqiang and Chatwattanasiri, Nida and Coit, David W. and Zhu, Xiaoyan},
  journal={IEEE Transactions on Reliability}, 
  title={Combined Redundancy Allocation and Maintenance Planning Using a Two-Stage Stochastic Programming Model for Multiple Component Systems}, 
  year={2017},
  volume={66},
  number={3},
  pages={950-962},
  doi={10.1109/TR.2017.2715172}
}

@ARTICLE{zhu2018designAndMaintenanceSequential,
  author={Zhu, Xiaoyan and Bei, Xiaoqiang and Chatwattanasiri, Nida and Coit, David W.},
  journal={IEEE Transactions on Reliability}, 
  title={Optimal System Design and Sequential Preventive Maintenance Under Uncertain Aperiodic-Changing Stresses}, 
  year={2018},
  volume={67},
  number={3},
  pages={907-919},
  doi={10.1109/TR.2018.2798298}
}

@book{kuo2001optimalReliabilityDesign,
  title={Optimal reliability design: fundamentals and applications},
  author={Kuo, Way},
  year={2001},
  publisher={Cambridge university press}
}

@article{devi2023RAPlitReview,
    author = {Devi, S. and Garg, H. and Garg, D.},
    title = {A review of redundancy allocation problem for two decades: bibliometrics and future directions},
    journal = {Artif Intell Rev},
    volume = {56},
    pages = {7457–7548},
    year = {2023},
    doi = {https://doi.org/10.1007/s10462-022-10363-6}
}

@article{guanCoit2025RAPlitReview,
author = {Bowen Guan and Zhanhang Li and David W. Coit and Yan-Fu Li},
title = {Review of the redundancy allocation problem to optimize system reliability},
journal = {Engineering Optimization},
volume = {57},
number = {1},
pages = {44--68},
year = {2025},
publisher = {Taylor \& Francis},
doi = {10.1080/0305215X.2024.2447078},
URL = {https://doi.org/10.1080/0305215X.2024.2447078},
eprint = {https://doi.org/10.1080/0305215X.2024.2447078}
}

@article{ANAND1994pwrCooling,
title = {Optimal redundancy allocation of a {PWR} cooling loop using a multi-stage {M}onte {C}arlo method},
journal = {Microelectronics Reliability},
volume = {34},
number = {4},
pages = {741-745},
year = {1994},
issn = {0026-2714},
doi = {https://doi.org/10.1016/0026-2714(94)90039-6},
url = {https://www.sciencedirect.com/science/article/pii/0026271494900396},
author = {S. Anand and M. Chidambaram}
}

@article{SULE2019SafetyCriticalEnergy,
title = {P-graph-based multi-objective risk analysis and redundancy allocation in safety-critical energy systems},
journal = {Energy},
volume = {179},
pages = {989-1003},
year = {2019},
issn = {0360-5442},
doi = {https://doi.org/10.1016/j.energy.2019.05.043},
url = {https://www.sciencedirect.com/science/article/pii/S0360544219309065},
author = {Zoltán Süle and János Baumgartner and Gyula Dörgő and János Abonyi}
}

@article{ling2021energySystems,
author = {Ling, Wen Choong and Andiappan, Viknesh and Chew, Irene M. L.},
title = {Design of energy systems with redundancy allocation for unit operations based on supply reliability},
journal = {International Journal of Energy Research},
volume = {45},
number = {15},
pages = {21114-21139},
doi = {https://doi.org/10.1002/er.7167},
url = {https://onlinelibrary.wiley.com/doi/abs/10.1002/er.7167},
eprint = {https://onlinelibrary.wiley.com/doi/pdf/10.1002/er.7167},
year = {2021}
}

@article{aneja1979dichotomic,
  title={Bicriteria transportation problem},
  author={Aneja, Yash P and Nair, Kunhiraman PK},
  journal={Management Science},
  volume={25},
  number={1},
  pages={73--78},
  year={1979},
  publisher={INFORMS}
}

@misc{dowson2025multiObjectiveAlgorithmsJl,
      title={MultiObjectiveAlgorithms.jl: a Julia package for solving multi-objective optimization problems},
      author={Oscar Dowson and Xavier Gandibleux and G{\"o}khan Kof},
      year={2025},
      eprint={2507.05501},
      archivePrefix={arXiv},
      primaryClass={math.OC},
      url={https://arxiv.org/abs/2507.05501}
}

\appendix
\section{Section 2 Supplement}

\subsection{Weakly Communicating and not Unichain}
Here, we provide the proof regarding the weakly communicating but not unichain properties of DMP. 

\begin{theorem}
    DMP is weakly communicating for $N>0$. However, there exist $N$ and $M_i$ for $i=1,...,N$ such that DMP is not unichain.
\end{theorem}
\begin{proof}
    We first show that DMP is weakly communicating. Let $\mathcal{S} = \bigtimes_{i=1}^N \{(s_1,s_2):s_1+s_2\leq M_i\}$, and choose two states $s,s'\in\mathcal{S}$. Assume for now that state $s'$ contains at least one healthy component. We choose an action $a\in\mathcal{A}(s')$ as follows:
    $$a_i = M_i - s'_{i2}.$$
    We then define the policy $\mu$ as follows:
    $$\mu(s) = \left\{\begin{array}{cc}
         a, & s=s_3,\\
         0, & \text{otherwise.}
    \end{array}\right.,$$
    where $s_3 = ((0,M_1),...,(0,M_N))$. Let $p_\mu^m(s,s')$ be the probability of reaching state $s'$ from state $s$ after $m$ transitions under policy $\mu$. Clearly, $p_\mu^m(s,s_3) > 0$, for some $m>0$ due to the inactivity of policy $\mu$, and $p_\mu^n(s_3, s') > 0$ for some $n>0$, as with non-zero probability after taking action $a$ in $s_3$, the suitable subset of components which are healthy in $s'$ all finish repairing before the components which are still repairing in $s'$. Therefore $p_\mu^{m+n}(s,s') > 0$.

    Assume now that state $s'$ has no healthy components, so all components are either damaged or currently being repaired. Taking action $a$ in state $s_3$ would ``leapfrog'' state $s'$ due to impulsive actions, so we amend action $a$ by starting repairs on an additional component (assuming that this exists). Then, with non-zero probability, this additional component both finishes repairing first and immediately fails again, leaving the desired components still repairing.

    All that is left now is the consideration of state $s_2 = ((M_1,0),...,(M_N,0))$. In this state, there is no ``additional'' component to repair, as all components are currently being repaired. In fact, this state is transient under all policies, as transitions from any state under any action represent either: a repair completion, so the next state must have at least one healthy component; or a repair failure or new degradation, meaning the next state must have at least one damaged component. Therefore, $s_2$ is transient under all policies, meaning DMP is weakly communicating as $p_\mu^m(s,s') > 0$ for all pairs of states $s,s'$ for some policy $\mu$, except for $s'=s_2$, which is uniquely transient under all policies. 

    It remains to show that DMP is not unichain for some $N$ and $M_i$. Consider a DMP with $N=2, M_1 = 1, M_2 = 2.$, and a policy satisfying the following:
    \begin{align*}
        \mu((0,1),(0,2)) &= (1,0)\\
        \mu((0,0),(0,2)) &= (0,0)\\
        &\\
        \mu((0,1),(1,0)) &= (0,0)\\
        \mu((0,1),(1,1)) &= \mu((0,1),(0,1)) = (0,1)\\
        \mu((0,1),(0,0)) &= \mu((0,1),(2,0))) = (0,0).
    \end{align*}

    Suppose we start in state $((0,1),(0,2))$, then it is clear to see we have recurrent class\\ $\{((0,1),(0,2)),((0,0),(0,2))\}$, as all we ever do is repair component type 1, and ``leapfrog'' state $((1,0),(0,2))$. We also have recurrent class\\ $\{((0,1),(1,0)),((0,1),(1,1)),((0,1),(0,1)),((0,1),(0,0)),((0,1),(2,0)))\}$. This is because when starting from any of these states, we follow a ``fully active'' maintenance policy on the two components of type 2, whilst ignoring component 1. From this class, we never degrade into state $((0,1),(0,2))$, as degradations happen one at a time, and are always immediately placed under repair. Therefore, the chain induced by this policy has two recurrent classes, therefore DMP is not unichain for this configuration of $N$ and $M_i$.
\end{proof}

\subsection{Efficient Construction of IDDMP State Space}

In Section 2.2 of the main paper, we note that many of the states in $\mathcal{S}^{\max}$ and many of the state-action pairs are infeasible, that is to say that $\pi(\bm{s},\bm{a})$ must be zero because no feasible design allows for them to take on any other value. As such, it is ideal to avoid generating such states and state-action pairs at all, allowing for larger problem instances to be solved. Note that the $\mathcal{S}_i$ (the single-type state spaces) are already constructed so as to follow the knapsack constraints.

We construct the reduced state space using a dynamic programming approach, starting by producing the state space the only uses two component type, then using this to construct the state space that only uses three component types, and so on. Let $x(s)$ represent the minimal design that allows for state $s$, i.e. the design that has $M_i-s_{ij2}$ copies of each component type $i$. Let $\mathcal{S}^{\max}_{1:n}$ be the space of partial states that only allows for the first $n\leq N$ component types. If $s_{1:n}\in \mathcal{S}^{\max}_{1:n}$, $[x(s_{1:n})]_i$ returns 0 for $i>n$. We start by defining $\mathcal{S}^{\max}_{1:2}=\left\{s_{1:2}\in \mathcal{S}_1\times \mathcal{S}_2: Ax(s_{1:2})\leq b\right\}$. This set is constructed computationally by enumerating all pairs $(s_1,s_2)$ and checking the knapsack conditions. We then use the recurrence $\mathcal{S}^{\max}_{1:n}=\left\{s_{1:n}\in \mathcal{S}^{\max}_{1:(n-1)}\times \mathcal{S}_n:Ax(s_{1:n})\leq b\right\}$. To construct $\mathcal{S}^{\max}_{1:3}$ up to $\mathcal{S}^{\max}_{1:N}$, which we call $\mathcal{S}^{\max}_{knap}$, which can be used in place of $\mathcal{S}^{\max}$ as it includes all states for which $\pi$ can feasibly take a non-zero action. This method allows for $\mathcal{S}^{\max}_{knap}$ to be enumerated without explicitly constructing $\mathcal{S}^{\max}$ and evaluating the knapsack constraint for each state.

The action spaces $\mathcal{A}(s)$ can be constructed similarly. For any state $s$, we start with $\mathcal{A}_i(s_i) = \left\{0,...,s_{ij2}\right\}$. Then $\mathcal{A}_{1:2}((s_1,s_2))=\left\{a_{1:2}\in A_1(s_1)\times A_2(s_2):Ax((s_1,s_2)\oplus a_{1:2})\leq b\right\}$, via abuse of notation on the impulsive operator (recall this is used to give every state-action pair a ``target'' state that it is effectively the same as). Recursively, $\mathcal{A}_{1:n}((s_1,...,s_n))=\left\{a_{1:n}\in \mathcal{A}_{1:(n-1)}\times \mathcal{A}_n: Ax((s_1,...,s_n)\oplus a_{1:n})\leq b\right\}$. $\mathcal{A}_{1:N}(s)$, aka $\mathcal{A}_{knap}(s)$ can then be used in place of $\mathcal{A}(s)$ for all $s\in\mathcal{S}^{\max}_{knap}$. When constructing the MILP model for IDDMP, when then need only construct the variables $\pi(s,a)$ such that $s\in\mathcal{S}^{\max}_{knap}$ and $a\in\mathcal{A}_{knap}$, and only need to construct the MDP balance equation constraints for states in $\mathcal{S}^{\max}_{knap}$. This allows the exact MILP formulation to scale better.

\newpage
\section{Section 3 Supplement}
\label{sec:supp3}
\autoref{sec:supp3:bound} provides a tighter bound of $M_i$, the maximum number of copies of component type $i$, for $\varepsilon-\delta-$DOP. This can reduce the number of binary decision variables in the model, and also allows for solutions to be obtained when no knapsack constraints are provided. To prove this bound, we also prove that the objective function of $\varepsilon-\delta-$DOP is supermodular. \autoref{sec:supp3:alg} provides the algorithms used for the application of APP to BO-IDDMP.

\subsection{Tighter Bound for $M_i$}
\label{sec:supp3:bound}

We state the bound in \autoref{proposition}.
\begin{proposition} \label{proposition}
    The optimal number of copies of component $i$ for $\varepsilon-\delta-$DOP is no more than 
    $$M_i = \min\left\{\max\left\{\llceil\left(\frac{1}{\ln q_i}\right)\ln\left\{-\frac{q_ir_i}{[(1+\delta)c_N - c_i]\ln q_i}\right\}\rrceil, \left\lceil\frac{\varepsilon}{\log q_i}\right\rceil\right\}, \min_j\{b_j/a_{ji}\}\right\}.$$
\end{proposition}

In order to prove this, we make use of an alternative formulation of $\varepsilon-\delta-$DOP which uses general integer variables instead of binary variables, which we call $\varepsilon-\delta-$DOP-G, and a few additional definitions and lemmas. We start with the model:

\begin{align}
    &\text{($\varepsilon-\delta-$DOP-G)}& \min_x\,&\sum_{i=1}^N r_iq_ix_i + \sum_{i=1}^Nc_i(1 - q_i^{x_i})\prod_{k<i}q_k^{x_k} + (1+\delta)c_N\prod_{i=1}^N q_i^{x_i}\label{e-DOP-G-obj}\\
    &&\text{s.t. } &Ax\leq b \label{e-DOP-G-knapsack}\\
    &&&\sum_{i=1}^N \log(q_i)x_i \leq \varepsilon \label{e-DOP-G-epsilon}\\
    &&&x_i \in \{0,1,2,3,...\} \text{ for all $i$ in $1,...,N$} \label{e-DOP-G-int}.
\end{align}

We now analyse single-component-type and multi-component-type variations of the problem to work towards a proof of \autoref{proposition}.

\begin{lemma}
\label{lemma:convex}
    The objective function of $\varepsilon-\delta-$DOP-G when allowing for only one component type $i$ in the solution is a convex function, with minimum: 
    $$x_i^* \in \left\{ \llfloor\left(\frac{1}{\ln q_i}\right)\ln\left\{-\frac{q_ir_i}{[(1+\delta)c_N - c_i]\ln q_i}\right\}\rrfloor, \llceil\left(\frac{1}{\ln q_i}\right)\ln\left\{-\frac{q_ir_i}{[(1+\delta)c_N - c_i]\ln q_i}\right\}\rrceil\right\}.$$
\end{lemma}

\begin{proof}
    Consider the continuous relaxation of the objective function of ($\varepsilon-\delta-$DOP-G) with only one $x_i$ non-zero for some component $i$:
    $$g_i:\RR^+\to\RR,\,g_i(x) = q_ir_ix + c_i(1 - q_i^x) + (1+\delta)c_Nq_i^x.$$
    We compute first and second derivatives:
    \begin{align*}
        g_i'(x) &= q_ir_i - (\ln q_i)c_iq_i^x + (1+\delta)c_N(\ln q_i)q_i^x,\\
        &= q_ir_i + \left[(1+\delta)c_N - c_i\right](\ln q_i)q_i^x\\
        g_i''(x) &= \left[(1+\delta)c_N - c_i\right](\ln q_i)^2q_i^x \geq 0
    \end{align*}
    As $g_i''(x) \geq 0$, $g$ is convex and therefore $g$ restricted to integers is also convex. By solving $g_i'(x^*) = 0$ we obtain a global minimum at:
    $$x^* = \left(\frac{1}{\ln q_i}\right)\ln\left\{-\frac{q_ir_i}{[(1+\delta)c_N - c_i]\ln q_i}\right\},$$
    hence the integer minimum will be this value rounded up or down.
\end{proof}

\begin{corollary}
    $g_i$ is increasing for $x > x^*$ as defined in \autoref{lemma:convex}.
\end{corollary}

We now wish to extend this notation of convexity to the case where we allow for multiple copies of each component. As such, we must introduce the notations of set functions and supermodularity.

\begin{definition}[Set function]
    Let $\mathcal{L}$ be some set. A function $f:2^\mathcal{L} \to \RR$ whose domain is the power set of $\mathcal{L}$ is called a (real) set function, and $\mathcal{L}$ is called the base set.
\end{definition}

We can alter our representation of $g$ using set functions. Let $\mathcal{L}' = \{(i,n) | i\in C, n\in\NN\}$. This represents an expanded version of our component set $\mathcal{L}$ with an infinite number of unique copies of each component. We then enforce the following partial ordering:
\begin{align*}
    (i,m) &\preceq (j,n) \text{ for all }i < j, i,j\in\mathcal{L}\\
    (i,m) & \preceq (i,n) \text{ for all }i\in\mathcal{L}, m < n.
\end{align*}

We then define the set-function equivalent of $g$, $g^{SET}$, as:
$$g^{SET}:2^{\mathcal{L}'}\to\RR, g^{SET}(K) = \sum_{i\in K}q_ir_i + \sum_{i\in K}c_ip_i\prod_{j\in K,j\prec i}q_j + (1+\delta)c_N\prod_{i\in K}q_i,$$
where we write $i$ and $j$ instead of $(i,n)$ and $(j,m)$ for brevity. We can map decision variables $x$ to subsets of $\mathcal{L}'$ as follows:
$$K:X\to 2^{\mathcal{L}'}, K(x) = \{(i,n)|n\leq x_i\}.$$
We then obtain the equality $g(x) = g^{SET}(K(x))$.

\begin{definition}[Marginal Set function]
    Let $\mathcal{L}$ be a base set and $f:2^\mathcal{L} \to \RR$ be a set function. Let $X\subset \mathcal{L}$ be some strict subset of $\mathcal{L}$ and $c\in \mathcal{L}\backslash X$ be some element not contained in $X$. Then the function $f_X:\mathcal{L}\backslash X\to \RR$ satisfying $f_X(c) = f(X\cup\{c\}) - f(X)$ is called the marginal set function of $f$ at $X$, and is denoted $f(c|X)$.
\end{definition}

\begin{definition}[Set Function Supermodularity]
    Let $\mathcal{L}$ be the base set of some set function $f:2^\mathcal{L} \to \RR$. $f$ is said to be a supermodular set function if and only if, for any subsets $X \subset Y \subset \mathcal{L}$, and any element $c\in N\backslash Y$, the following holds:
    $$f(c|X) \leq f(c|Y).$$
\end{definition}

With our definitions laid out, we may state and prove an important property of our objective function.

\begin{lemma}
\label{lemma:supermodular}
    $g^{SET}$ is a supermodular set function.
\end{lemma}
\begin{proof}
We shall suppress the $SET$ superscript for brevity. Recall :
$$g(K) = \sum_{i\in\mathcal{L}'}q_ir_i + \sum_{i\in K}c_ip_i\prod_{j\in K,j\preceq i)}q_j + (1+\delta)c_N\prod_{i\in K}q_i.$$
Clearly, the first sum is a modular set function, so we need only show the supermodularity of the following:
$$f(K) := \sum_{i\in K}c_ip_i\prod_{j\in K,j\preceq i)}q_j + (1+\delta)c_N\prod_{i\in K}q_i,$$
as the positive linear combination of supermodular functions is supermodular.

To show supermodularity, for any subsets $X\subset Y \subset C'$, and any component $i\in \mathcal{L}'\backslash Y$, we must show that:
$$f(X\cup\{i\}) - f(X) \leq f(Y\cup\{i\}) - f(Y).$$

To simplify notation, we may use $K^{>i}$ and $K^{<i}$ to denote the partitions of a set $K$ with only the components greater than $l$ under the partial ordering, and the components considered lesser than $i$, respectively. We may also write:
$$f(i|K) := f(K\cup\{i\}) - f(K),$$
and $f(K;c)$ to be $f$ with $p$ replaced by $c$.

We first calculate $f(K\cup \{i\})$ as follows:
\begin{align*}
    f(K\cup\{i\}) &= \sum_{j\in K\cup\{i\}}\left( c_j p_j \prod_{k\in K\cup\{i\}, k<j}q_k\right) + (1+\delta)c_N\prod_{j\in K\cup\{i\}}q_j\\
    &= \sum_{j\in K, j < i}\left( c_j p_j \prod_{k\in K, k<j}q_k\right)\\
    &\quad+ c_ip_i\prod_{j < i}q_j
    + \sum_{j\in K, j > i}\left( c_j p_j \prod_{k\in K \cup \{i\}, k<j}q_k\right)\\
    &\quad+ (1+\delta)c_N \prod_{j\in K \cup \{i\}}q_j\\
    &= \sum_{j\in K, j < i}\left( c_j q_j \prod_{k\in K, k<j}q_k\right)\\
    &\quad+ c_ip_i\prod_{j < i}q_j
    + q_i\sum_{j\in K, j > i}\left( c_j p_j \prod_{k\in K, k<j}q_k\right)\\
    &\quad+ q_i(1+\delta)c_N \prod_{j\in K}q_j\\\vspace{1em}
    &= c_lx_1^l\prod_{k < l}(1 - x_1^k)\\
    &\quad+x_1^l\sum_{j\in K, j < l}\left( c_j x_1^j \prod_{k\in K, k<j}(1 - x_1^k)\right)\\
    &\quad+(1 - x_1^l)\left\{\sum_{j\in K}\left( c_j x_1^j \prod_{k\in K, k<j}(1 - x_1^i)\right) + p \prod_{i\in K}(1 - x_1^i)\right\}\\\vspace{1em}
    &= x_1^l\left\{c_l\prod_{k < l}(1 - x_1^k) + \sum_{i\in K, i < l}\left( c_i x_1^i \prod_{k\in K, k<j}(1 - p_k)\right)\right\}\\
    &\quad+ q_if(K)\\\vspace{1em}
    &= p_i f(K^{<i};c_i) + q_if(K).
\end{align*}
This makes intuitive sense. By including an extra component $i$, with probability $p_i$ you are at worst using component $i$ and ignoring all links worse than $i$, and with probability $q_i$ the link is unavailable and the system falls back onto the original links in $K$. Notably, this shows that $f$ is monotone. We then see that:
\begin{align*}
    f(i|K) &= f(K\cup\{i\}) - f(K)\\
    &= p_i\left(f(K^{<l};c_I) - f(K)\right).
\end{align*}
As the constant $p_i$ does not depend on $K$, we need only focus on $f(K^{<l};c_l) - f(K)$. We see:
\begin{align*}
    f(K^{<i};c_i) - f(K) &= c_i\prod_{k\in K, k < i}q_k + \sum_{j\in K, j < i}\left( c_j p_j \prod_{k\in K, k<j}q_k\right)\\
    &\quad-\left\{\sum_{j\in K}\left( c_j p_j \prod_{k\in K, k<j}q_j\right) + p \prod_{j\in K}q_j\right\}\\
    &= c_i\prod_{k\in K, k < i}q_k - \sum_{j\in K, j > i}\left( c_j p_j \prod_{k\in K, k<j}q_k\right) - p\prod_{j\in K}q_j\\
    &= \left\{\prod_{k\in K, k < i}q_k\right\}\left\{ c_i - \sum_{j\in K, j > i}c_j p_j \prod_{k\in K, i < k < j}q_k - p\prod_{j\in K, j > i}q_k\right\}\\
    &= \left\{\prod_{k\in K, k < i}q_k\right\}\left\{c_i - f(K^{>i})\right\}\\
    &=-\left\{\prod_{k\in K, k < i}q_k\right\}\left\{f(K^{>i}) - c_i\right\}.
\end{align*}

The magnitude of this expression decreases as we add more links to $K$ that are cheaper than $i$ by the first term, and decreases as we add more links more expensive than $i$ by the second term, as $f(K^{>i}) \geq c_i$. Therefore the function is supermodular. 
\end{proof}

Using these properties, we may now prove \autoref{proposition}.

\begin{proof}
    Firstly, assume we have no knapsack constraints. We assume that our bound must allow us to represent the following feasible solution:
    $$K = \left\{(i,n)\,|\,n\leq\max\left\{\llceil\left(\frac{1}{\ln q_i}\right)\ln\left\{-\frac{q_ir_i}{[(1+\delta)c_N - c_i]\ln q_i}\right\}\rrceil,\llceil \frac{\varepsilon}{\ln q_i}\rrceil\right\}\right\}.$$
    This solution is feasible according to the $\varepsilon$-constraint, and cannot be improved by adding more copies of component $i$ by \autoref{lemma:convex}. Extending this, for any sets $L \supset K$ and $L'$ obtained by adding one extra copy of component $i$, we have $g^{SET}(L') \geq g^{SET}(L)$ by supermodularity. As such, we need not consider solutions with more than $|K|$ copies of component $i$, so we obtain bound $M^{NO-KNAP}_i = \max\left\{\llceil -\frac{q_ir_i}{[(1+\delta)c_N - c_i]\ln q_i} \rrceil,\llceil \frac{\varepsilon}{\ln q_i}\rrceil\right\}$. Now adding back in knapsack constraints, we need not consider solutions that violate any knapsack constraint, so we strengthen our bound to $M_i = \min\{M_i^{NO-KNAP}, \min_j\{b_j/a_{ji}\}\}$, as desired.
\end{proof}
\newpage
\subsection{Algorithms}
\label{sec:supp3:alg}
Here, we provide the algorithms used to implement the APP methodology for BO-IDDMP.

\autoref{alg:sp1} takes as input a starting value $\varepsilon_{\min}$, an increment value $\Delta\varepsilon$, and the problem parameters, and returns a set of Pareto-optimal solutions to this problem. It begins by obtaining the two objective values and design solution for F-DOP, giving a lower-bound on $\varepsilon$ and providing the most extreme solution in terms of reliability. Then starting with $\varepsilon = \varepsilon_{\min}$, we iteratively solve $\varepsilon-\delta-$DOP-PC, store the resulting solution $(g^o, \ln g^f, x)$, and update $\varepsilon$ to $\ln g^f + \Delta\varepsilon$ , ensuring that the next solution is more reliable. We do this until the lower bound on $\epsilon$ is reached, at which point we return all objective values and solutions found.

\begin{algorithm}[h]
\caption{SP1} \label{alg:sp1}
\begin{algorithmic}
\Require $\varepsilon_{\min}, \Delta\varepsilon,\mathcal{L} =  (N, \alpha, \tau, c, r, A, b)$
\State StaticSolutions $\gets [\,]$
\State $(g^o_f, \ln g^f_f, x_f) \gets$ F-DOP$(\mathcal{L})$
\State $\text{StaticSolutions}\gets [(g^o_f, \ln g^f_f, x_f)]$
\State $\varepsilon \gets \varepsilon_{\min}$
\While{$\varepsilon \leq \ln g^f_f$}
    \State $(g^o, \ln g^f, x) \gets \varepsilon-\delta-\text{DOP-PC}(\mathcal{L}, \varepsilon)$
    \State push(StaticSolutions, $(g^o, \ln g^f, x)$)
    \State $\varepsilon \gets \ln g^f + \Delta\varepsilon$
\EndWhile

\Return $\text{StaticSolutions}$
\end{algorithmic}
\end{algorithm}

\autoref{alg:sp2} takes as input a design $x$, a label for that design $i$, a minimum LFR $\ln g^f_{static}$, a minimum penalty $p_{\min}$, a multiplicative increment $\Delta p$, and the problem parameters. It returns a set of Pareto-optimal solutions. Starting with $p = p_{\min}$, it solves $p-$DMP for the design $x$ and stores the tuple of the objective values, design label, and penalty used, $(g^o, \ln g^f, i, p)$. It then performs the update $p\gets p\times\Delta p$, and loops back around. It does this until $\ln g^f \neq \ln g^f_{static}$, at which point the penalty $p$ has been increased far enough that the fully-active maintenance policy has been recovered. It then returns all solutions found.

\begin{algorithm}[h]
\caption{SP2} \label{alg:sp2}
\begin{algorithmic}
\Require $x, i, \ln g^f_{static}, p_{\min}, \Delta p, \mathcal{L} =  (N, \alpha, \tau, c, r, A, b)$
\State DynamicSolutions $\gets [\,]$
\State $p\gets p_{\min}$
\State $(g^o, g^f) \gets (0,0)$
\While{$\ln g^f \neq \ln g^f_{static}$} 
    \State $(g^o, g^f) \gets p$-DMP$(\mathcal{L}, x, p)$
    \State push(DynamicSolutions, $(g^o, \ln g^f, i, p)$)
    \State $p \gets p\times\Delta p$
\EndWhile

\Return DynamicSolutions
\end{algorithmic}
\end{algorithm}

\autoref{alg:app} is the full algorithm for applying APP to BO-IDDMP. It takes as input $\varepsilon_{\min}, \Delta\varepsilon, p_{\min}, \Delta p, \delta$, and the problem parameters. It constructs an approximate set of Pareto-optimal solutions to the problem. Alongside the separate functions for SP1 and SP2$(x)$, it makes use of two helper functions: NonNestedDesigns and NonDomSolutions. NonNestedDesigns takes as input a list of three-tuples where the first two entries contain objective values, and the third entry contains decision variables representing a design. A design $x_1$ is considered ``nested'' in another design $x_2$ if the $x_2$ contains at least as many copies of each type of component as $x_1$, and at least one extra component. The purpose of eliminating such designs is that any dynamic policy on design $x_1$ can be emulated on $x_2$ by simply ignoring the extra components. As such, obtaining a Pareto-front of policies for each design would result in repeated solutions and wasted computational time. NonNestedSolutions therefore checks all designs against all other designs to check for nesting, and returns the sublist that only contains solutions which are not nested in other solutions. NonDomSolutions takes a list of tuples of varying sizes where the first two components of the tuple describes the two objectives. It iterates over all solutions and checks them against all other solutions to check for domination of one solution over another. It then returns the sublist of non-dominated tuples. 

\begin{algorithm}[h]
\caption{APP} \label{alg:app}
\begin{algorithmic}
\Require $\varepsilon_{\min}, \Delta\varepsilon, p_{\min}, \Delta p, \delta, \mathcal{L} =  (N, \alpha, \tau, c, r, A, b)$
\State StaticSolutions$\gets \text{SP1}(\varepsilon_{\min}, \Delta\varepsilon,\mathcal{L})$
\State $\text{NonNestedDesigns}\gets$NonNestedDesigns(StaticSolutions)
\State DynamicSolutions $\gets [\,]$
\For{$i = 1,...,$length(NonNestedDesigns)}
    \State $(g^o_{static}, \ln g^f_{static}, x) \gets$ NonNestedDesigns$[i]$
    \State NewDynamicSolutions$\gets \text{SP2}(x, i, \ln g^f_{static}, p_{\min}, \Delta p, \delta, \mathcal{L})$
    \State append(DynamicSolutions, NewDynamicSolutions)
\EndFor

\State Population $\gets (\text{StaticSolutions} \backslash \text{NonNestedDesigns}) \cup \text{DynamicSolutions}$
\State NonDomSolutions $\gets$ NonDomSolutions(Population)

\end{algorithmic}
\end{algorithm}

Both NonNestedDesigns and NonDomSolutions use very similar logic to check for domination or nesting, so for the sake of brevity we only show the algorithm for NonDomSolutions in \autoref{alg:nondomsol}.

\begin{algorithm}[h]
\caption{NonDomSolutions}  \label{alg:nondomsol}
\begin{algorithmic}
\Require Population  
\For{sol $\in$ Population}
    \State dominated $\gets$ False
    \For{$\text{sol}'\in\text{Population}\backslash\{\text{sol}\}$}
        \If{$\text{sol}'[1] \leq \text{sol}[1] \text{ and } \text{sol}'[2] \leq \text{sol}[2] \text{ and } \left( \text{sol}'[1] < \text{sol}[1] \text{ or } \text{sol}'[2] < \text{sol}[2]\right)$}
            \State dominated $\gets$ True
            \State Break
        \EndIf
    \EndFor
    \If{not dominated}
        \State push(NonDomSolutions, sol)
    \EndIf 
\EndFor
\end{algorithmic}    
\end{algorithm}

\section{Section 4 Supplement}

\subsection{Results Table for Section 4.2}
\begin{table}[tbp]
    \centering
    \begin{tabular}{cccccccc}\toprule
    & & \multicolumn{2}{c}{LP}& \multicolumn{2}{c}{Heuristic} & \multicolumn{2}{c}{DOP}\\ \cmidrule(lr){3-4}\cmidrule(lr){5-6}\cmidrule(lr){7-8}
    Comp. Set & Budget & Time    & Solutions & Time & Solutions & Time  & Solutions \\
1         & 6      & 0.38    & 6         & 0.95 & 8         & 0.011 & 4         \\
1         & 8      & 0.86    & 9         & 0.99 & 15        & 0.032 & 5         \\
1         & 10     & 3.09    & 14        & 1.3  & 24        & 0.024 & 6         \\
1         & 12     & 14.46   & 22        & 1.37 & 34        & 0.038 & 7         \\
1         & 14     & 63.84   & 31        & 1.81 & 54        & 0.049 & 8(-1)         \\
1         & 16     & 294.75  & 38        & 2.48 & 86        & 0.075 & 9(-1)         \\
2         & 24     & 0.38    & 8         & 0.32 & 7         & 0.02  & 4         \\
2         & 32     & 0.51    & 12        & 0.33 & 11        & 0.028 & 5         \\
2         & 40     & 1.72    & 19        & 0.34 & 17        & 0.034 & 6         \\
2         & 48     & 6.47    & 27        & 0.35 & 26        & 0.023 & 7         \\
2         & 56     & 26.82   & 36        & 0.36 & 32        & 0.026 & 8(-3)         \\
2         & 64     & 134.35  & 43        & 0.41 & 41        & 0.042 & 9(-4)         \\
3         & 12     & 0.34    & 8         & 0.33 & 8         & 0.008 & 4         \\
3         & 16     & 0.53    & 13        & 0.34 & 11        & 0.011 & 5         \\
3         & 20     & 1.63    & 17        & 0.36 & 19        & 0.02  & 6         \\
3         & 24     & 6.23    & 26        & 0.37 & 24        & 0.027 & 7(-2)         \\
3         & 28     & 34.84   & 34        & 0.41 & 34        & 0.053 & 8(-3)         \\
3         & 32     & 165.31  & 44        & 0.42 & 41        & 0.047 & 9(-4)         \\
4         & 12     & 0.31    & 4         & 0.32 & 4         & 0.011 & 3         \\
4         & 16     & 0.41    & 8         & 0.33 & 7         & 0.014 & 4         \\
4         & 20     & 0.87    & 12        & 0.32 & 11        & 0.012 & 5         \\
4         & 24     & 4.76    & 30        & 0.38 & 31        & 0.022 & 6         \\
4         & 28     & 57.69   & 45        & 0.9  & 68        & 0.038 & 7(-1)         \\
4         & 32     & 207.76  & 42        & 1.57 & 93        & 0.054 & 8(-1)         \\
5         & 9      & 0.35    & 7         & 0.64 & 7         & 0.01  & 4         \\
5         & 12     & 1.05    & 12        & 0.95 & 10(-1)        & 0.013 & 5         \\
5         & 15     & 2.34    & 14        & 1.01 & 36        & 0.018 & 6         \\
5         & 18     & 14.56   & 27        & 1.27 & 52        & 0.044 & 7         \\
5         & 21     & 97.07   & 39        & 2.17 & 65        & 0.035 & 8         \\
5         & 24     & 382.46  & 54        & 5.44 & 109(-6)       & 0.226 & 9         \\
6         & 12     & 0.37    & 5         & 0.64 & 6         & 0.009 & 4         \\
6         & 16     & 0.81    & 9         & 0.64 & 9         & 0.017 & 5         \\
6         & 20     & 4.08    & 14        & 0.68 & 12        & 0.033 & 6         \\
6         & 24     & 66.83   & 39        & 0.73 & 32        & 0.025 & 7         \\
6         & 28     & 4624.9  & 31        & 1.26 & 60        & 0.044 & 8         \\
6         & 32     & 7196.05 & 0         & 1.87 & 62        & 0.054 & 9(-3)         \\
7         & 21     & 0.33    & 5         & 0.63 & 6         & 0.007 & 4         \\
7         & 28     & 0.45    & 9         & 0.64 & 10        & 0.009 & 5         \\
7         & 35     & 1.88    & 20        & 0.67 & 22        & 0.014 & 6         \\
7         & 42     & 32.8    & 39        & 1.23 & 63        & 0.026 & 7         \\
7         & 49     & 145.97  & 39        & 2.33 & 83        & 0.043 & 8         \\
7         & 56     & 385.03  & 31        & 6.3  & 83        & 0.078 & 9         \\

    \end{tabular}
    \caption{Comparison of runtimes and numbers of solutions found for each method.}
    \label{tab:runtimes1}
\end{table}

\begin{table}[tbp]
    \centering
    \begin{tabular}{cccccccc}\toprule
    & & \multicolumn{2}{c}{LP}& \multicolumn{2}{c}{Heuristic} & \multicolumn{2}{c}{DOP}\\ \cmidrule(lr){3-4}\cmidrule(lr){5-6}\cmidrule(lr){7-8}
    Comp. Set & Budget & Time    & Solutions & Time & Solutions & Time  & Solutions \\
8         & 12     & 0.32    & 5         & 0.66 & 5         & 0.008 & 4         \\
8         & 16     & 0.42    & 10        & 0.67 & 10        & 0.009 & 5         \\
8         & 20     & 1.15    & 11        & 1.08 & 19        & 0.019 & 6         \\
8         & 24     & 2.41    & 15        & 1.09 & 24        & 0.02  & 7         \\
8         & 28     & 12.97   & 31        & 1.33 & 42        & 0.027 & 8         \\
8         & 32     & 77.89   & 36        & 2.35 & 78        & 0.06  & 9(-1)\\
9         & 21     & 0.36    & 5         & 0.64 & 5         & 0.011 & 4         \\
9         & 28     & 0.57    & 8         & 0.33 & 8         & 0.01  & 4         \\
9         & 35     & 4.98    & 20        & 0.37 & 19        & 0.016 & 5         \\
9         & 42     & 254.48  & 40        & 0.89 & 51        & 0.04  & 6         \\
9         & 49     & 2938.78 & 36        & 1.92 & 60        & 0.035 & 7         \\
9         & 56     & 7196.93 & 0         & 0.39 & 22        & 0.055 & 7(-1)\\
10        & 15     & 0.34    & 5         & 0.64 & 6         & 0.007 & 4         \\
10        & 20     & 0.46    & 9         & 0.65 & 9         & 0.01  & 5         \\
10        & 25     & 1.23    & 14        & 0.65 & 13        & 0.032 & 6         \\
10        & 30     & 4.79    & 19        & 0.35 & 17        & 0.02  & 6         \\
10        & 35     & 39.9    & 46        & 0.47 & 49        & 0.03  & 7(-2)\\
10        & 40     & 916.98  & 88        & 1.28 & 113       & 0.034 & 8(-3)\\
11        & 15     & 0.34    & 5         & 0.64 & 6         & 0.008 & 4         \\
11        & 20     & 0.45    & 9         & 0.66 & 9         & 0.01  & 5         \\
11        & 25     & 1.29    & 14        & 0.68 & 15        & 0.019 & 6         \\
11        & 30     & 5.44    & 20        & 0.69 & 18        & 0.022 & 7         \\
11        & 35     & 56.36   & 46        & 0.81 & 49        & 0.032 & 8         \\
11        & 40     & 1160.02 & 79        & 1.74 & 106       & 0.054 & 9(-1)\\
12        & 12     & 0.33    & 5         & 0.69 & 6         & 0.01  & 4         \\
12        & 16     & 0.77    & 7         & 1.16 & 10        & 0.027 & 5         \\
12        & 20     & 2.8     & 12        & 1.25 & 19        & 0.029 & 6         \\
12        & 24     & 23.32   & 15        & 1.71 & 38        & 0.096 & 7         \\
12        & 28     & 94.76   & 17        & 2.97 & 37        & 0.39  & 8         \\
12        & 32     & 524.15  & 33        & 5.53 & 61        & 0.674 & 9         \\
13        & 15     & 0.38    & 8         & 0.4  & 8         & 0.009 & 4         \\
13        & 20     & 0.76    & 14        & 0.35 & 11        & 0.012 & 5         \\
13        & 25     & 2.27    & 20        & 0.36 & 17        & 0.03  & 6         \\
13        & 30     & 7.99    & 21        & 0.41 & 22        & 0.03  & 7(-1)\\
13        & 35     & 79.29   & 35        & 0.46 & 30        & 0.037 & 8(-1)\\
13        & 40     & 615.59  & 40        & 0.44 & 36        & 0.043 & 9(-1)\\
14        & 18     & 0.35    & 4         & 0.32 & 4         & 0.007 & 3         \\
14        & 24     & 0.94    & 10        & 0.32 & 11        & 0.011 & 4         \\
14        & 30     & 6.0     & 13        & 0.64 & 20        & 0.018 & 5         \\
14        & 36     & 117.9   & 18        & 0.32 & 11        & 0.026 & 5         \\
14        & 42     & 3410.58 & 34        & 0.37 & 29        & 0.028 & 6         \\
14        & 48     & 5974.65 & 6         & 0.86 & 57        & 0.05  & 7
    \end{tabular}
\caption{Comparison of runtimes and numbers of solutions found for each method.}
\end{table}
In the Heuristic solutions column, $(-x)$ indicates that $x$ of the solutions found were dominated by an exact solution. In the DOP solutions column, $(-x)$ indicates that $x$ solutions were dominated by an APP solution, demonstrating where a dynamic policies not only provided a more complete Pareto front, but also dominated an always-maintained design solution across both objectives.

\newpage
\subsection{Additional Plots for Section 4.2}

\begin{figure}[h]
    \centering
    \includegraphics[width=0.9\linewidth]{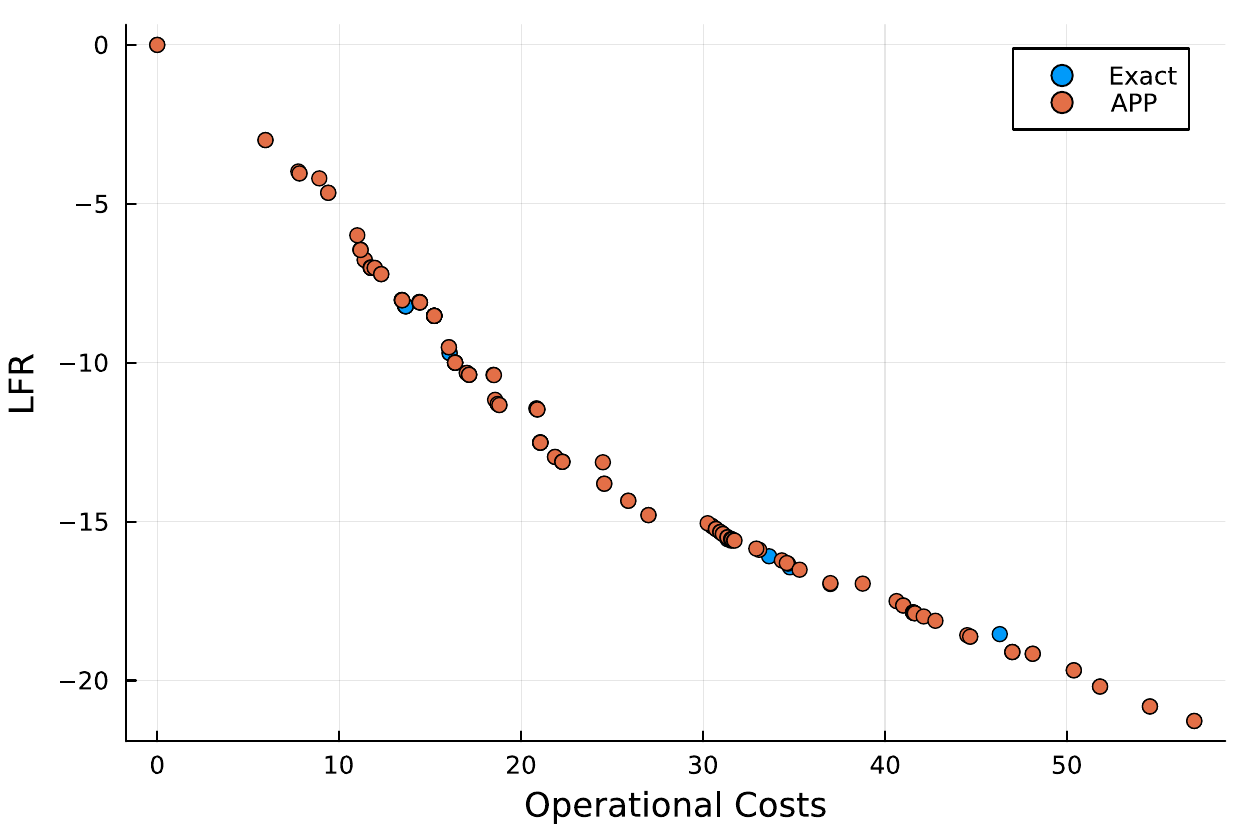}
    \caption{Pareto Front for Instance (5,24)}
    
\end{figure}

\begin{figure}
    \centering
    \includegraphics[width=0.9\linewidth]{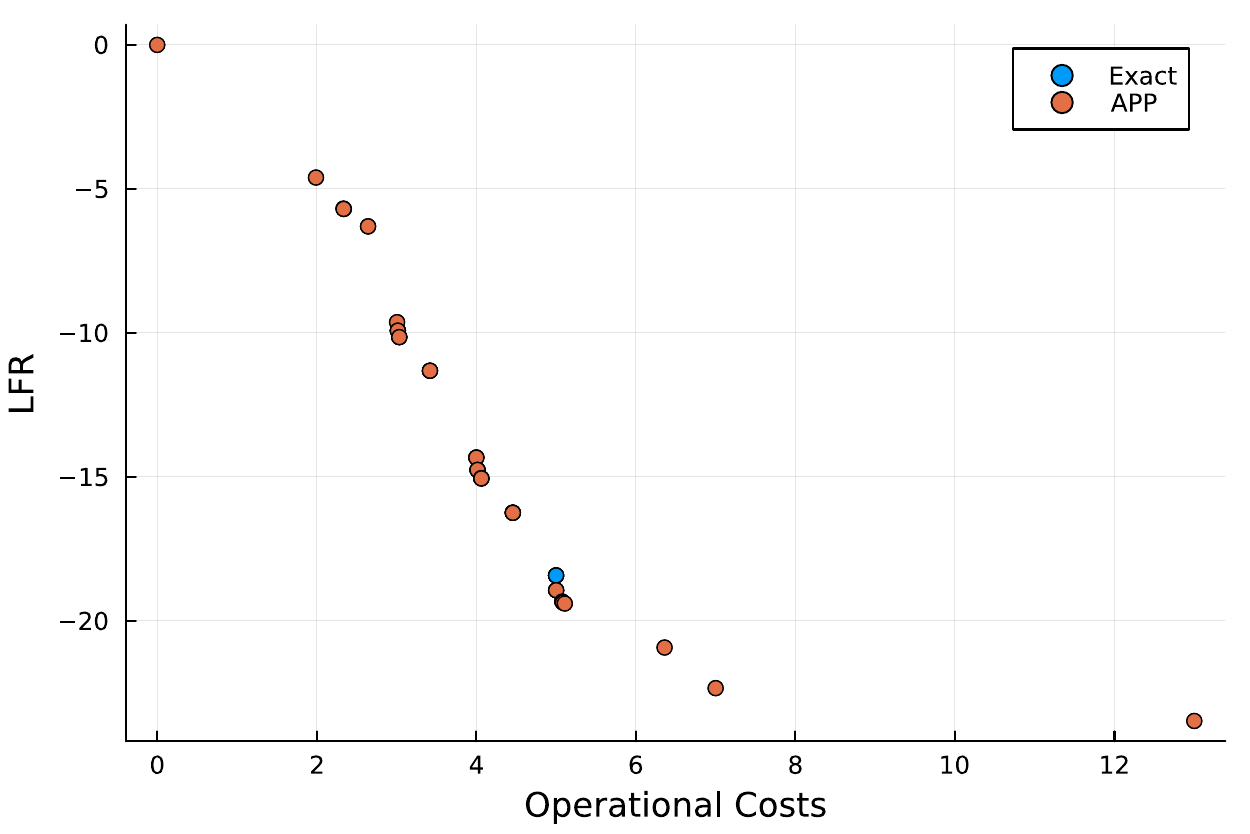}
    \caption{Pareto Front for Instance (6,24)}
    
\end{figure}
\begin{figure}
    \centering
    \includegraphics[width=0.9\linewidth]{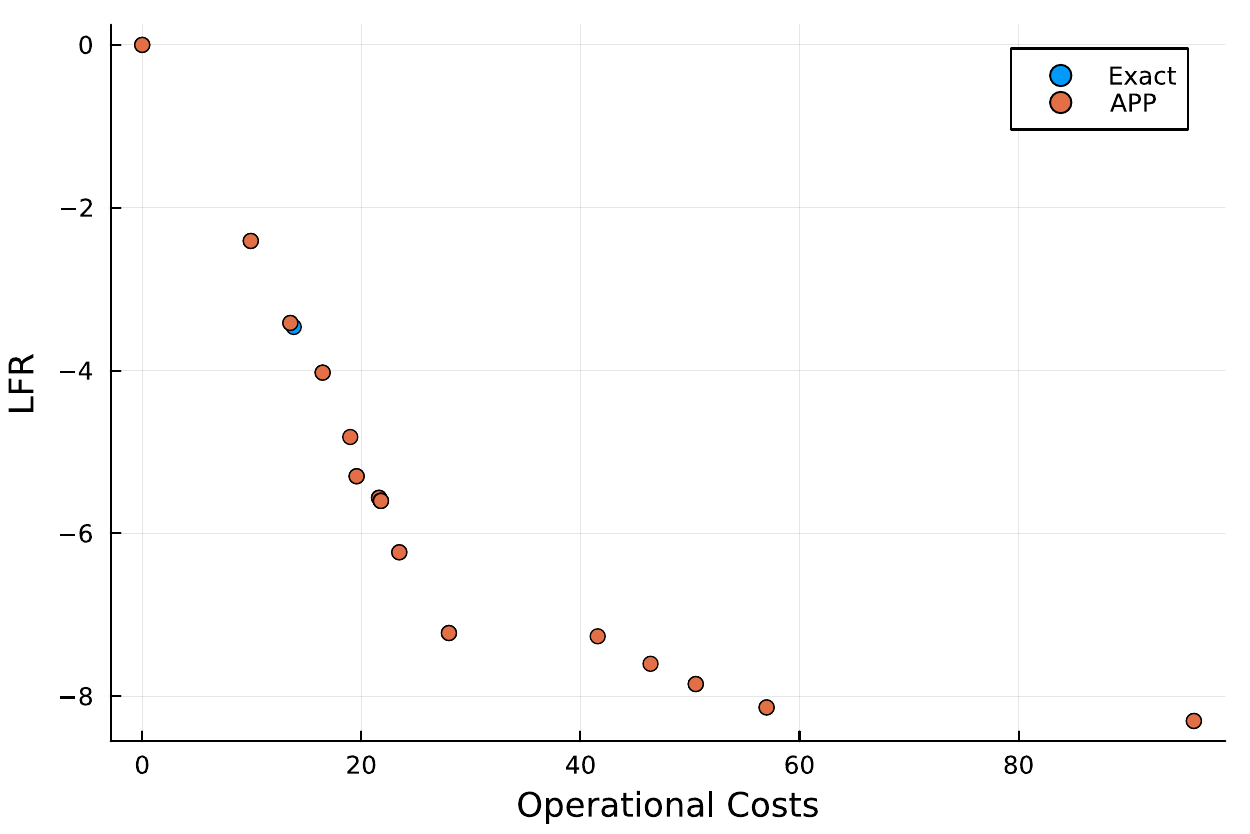}
    \caption{Pareto Front for Instance (8,20)}
    
\end{figure}
\begin{figure}
    \centering
    \includegraphics[width=0.9\linewidth]{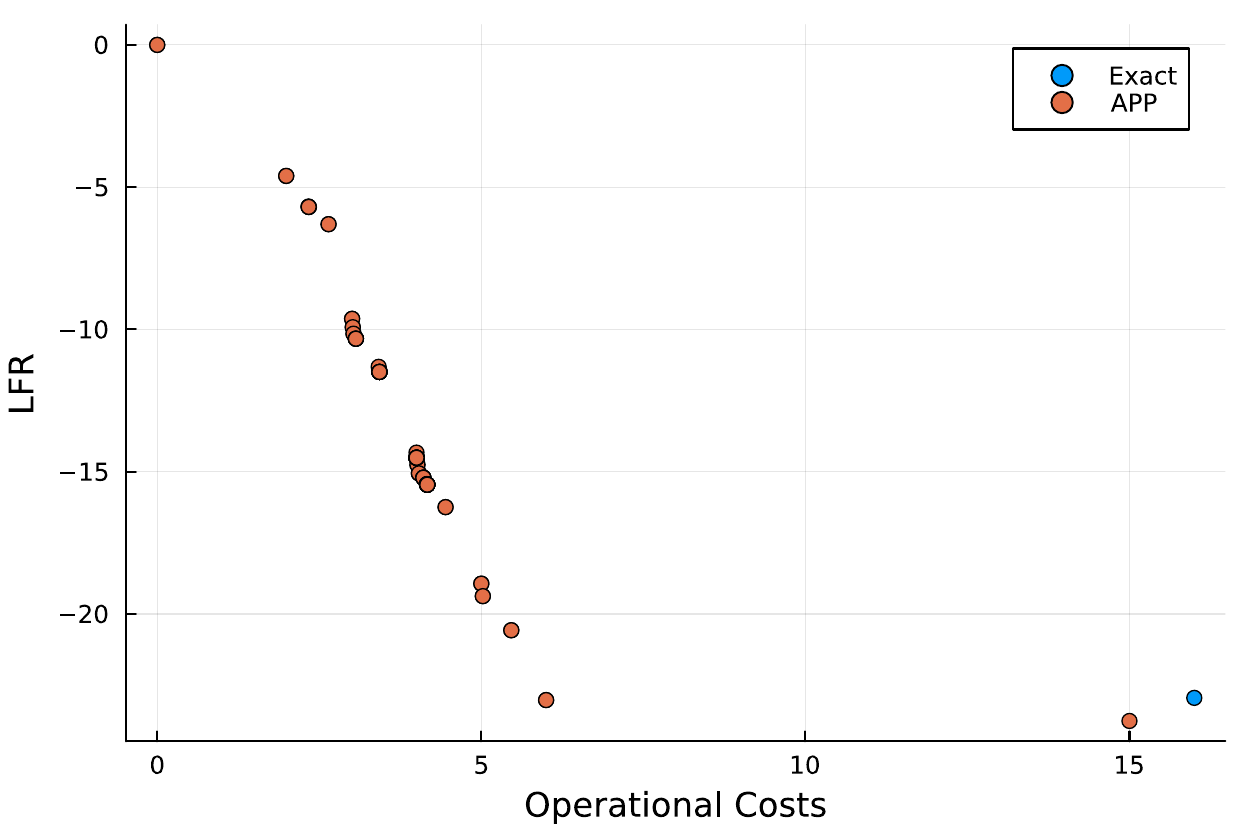}
    \caption{Pareto Front for Instance (9,49)}
    
\end{figure}\begin{figure}
    \centering
    \includegraphics[width=0.9\linewidth]{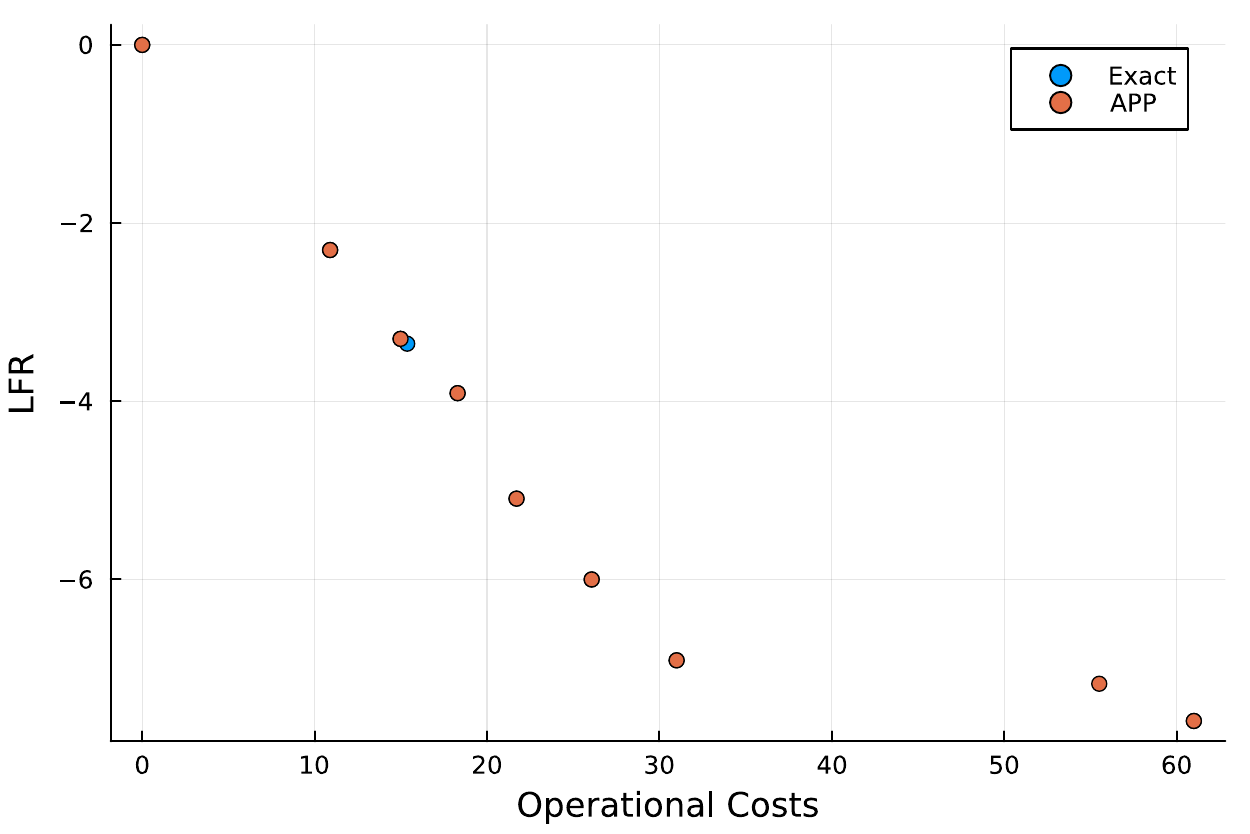}
    \caption{Pareto Front for Instance (10,20)}
    
\end{figure}
\begin{figure}
    \centering
    \includegraphics[width=0.9\linewidth]{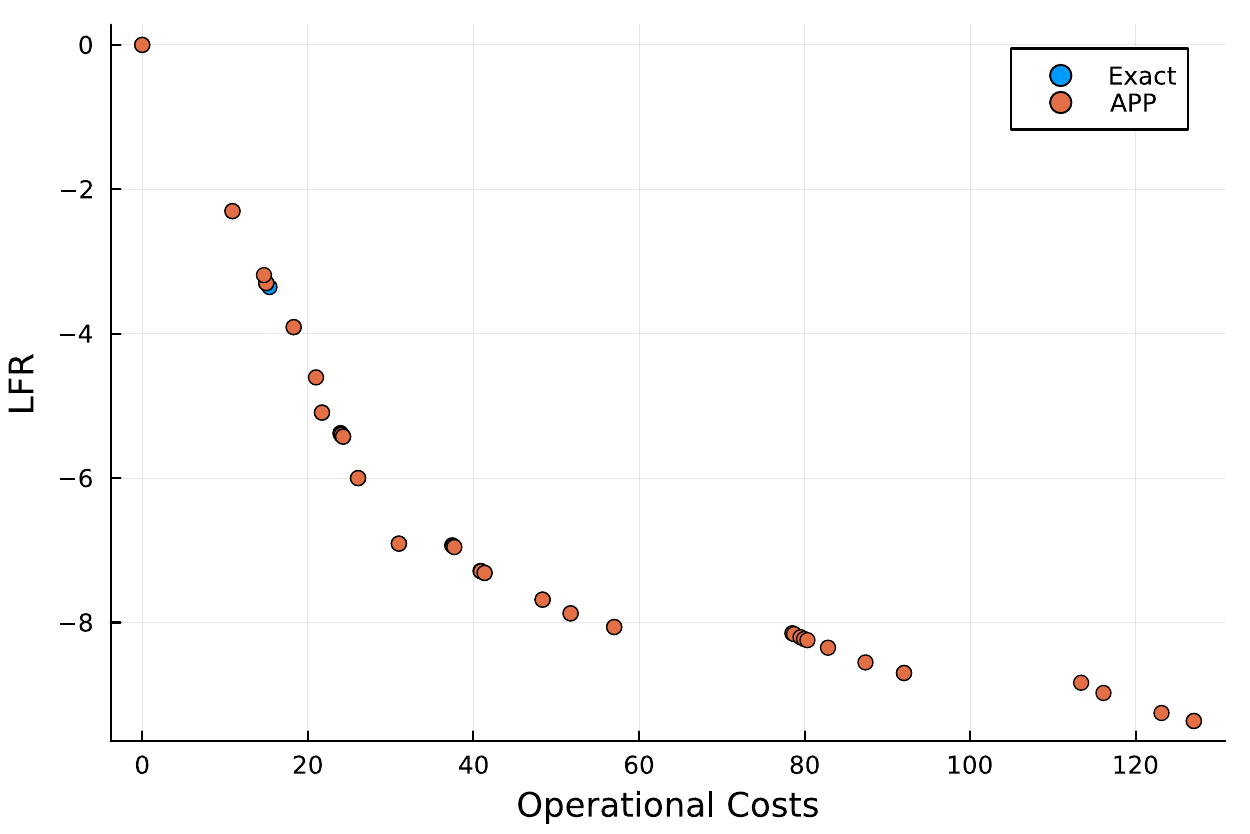}
    \caption{Pareto Front for Instance (12,24)}
    
\end{figure}
\begin{figure}
    \centering
    \includegraphics[width=0.9\linewidth]{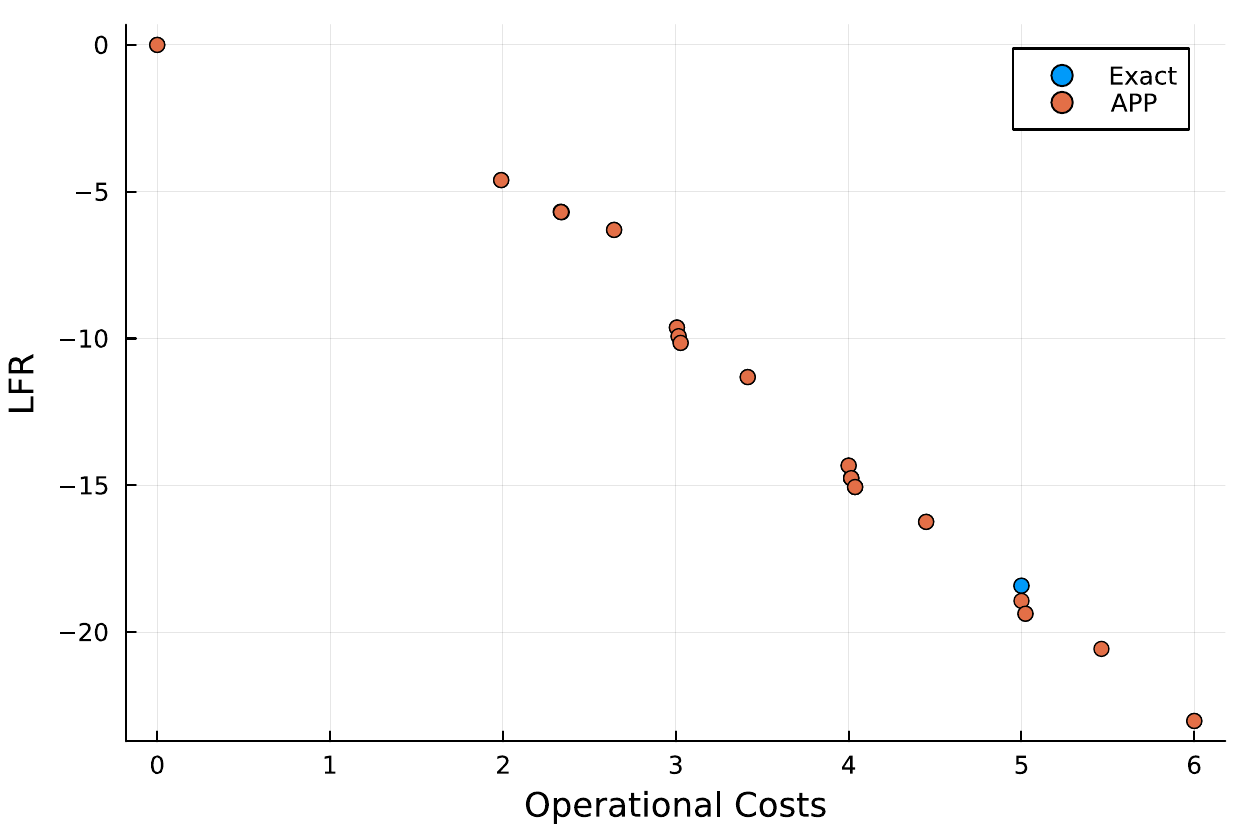}
    \caption{Pareto Front for Instance (13,25)}
    
\end{figure}
\begin{figure}
    \centering
    \includegraphics[width=0.9\linewidth]{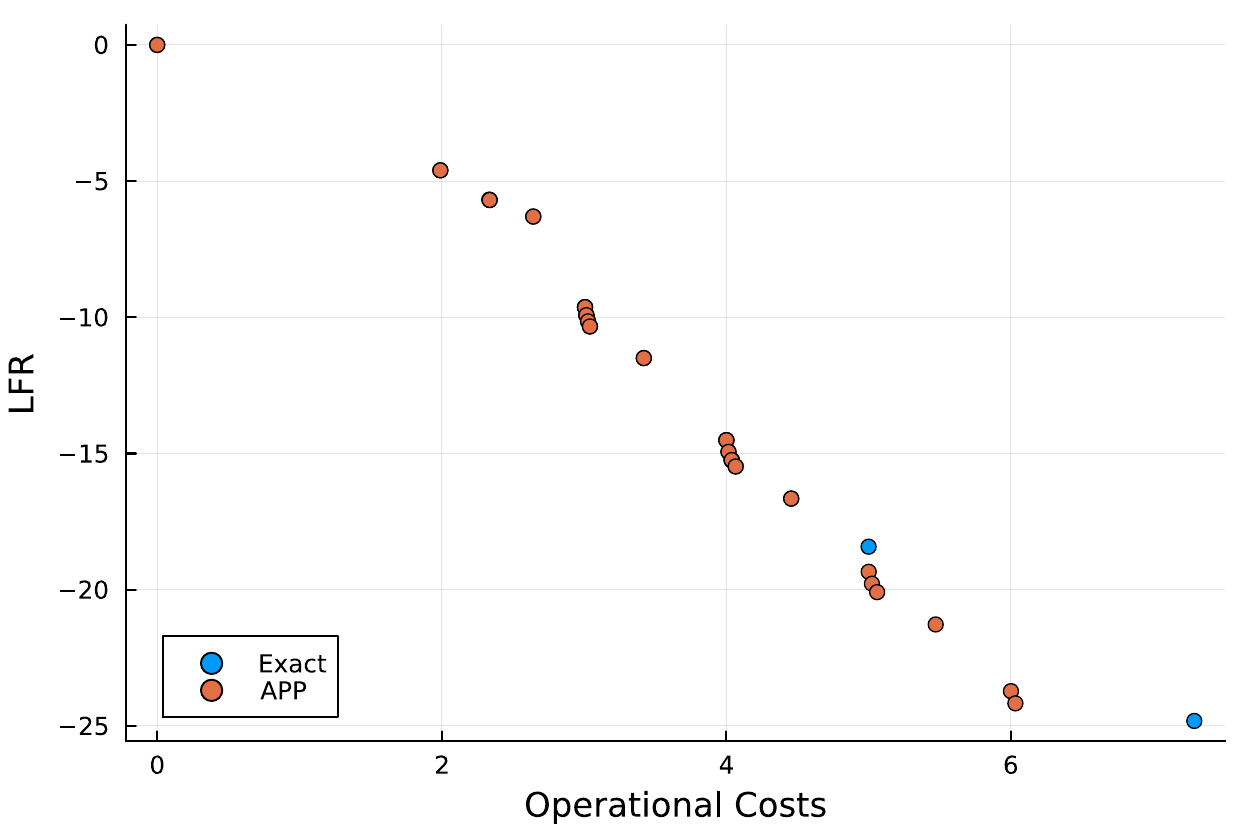}
    \caption{Pareto Front for Instance (13,30)}
    
\end{figure}
\begin{figure}
    \centering
    \includegraphics[width=0.9\linewidth]{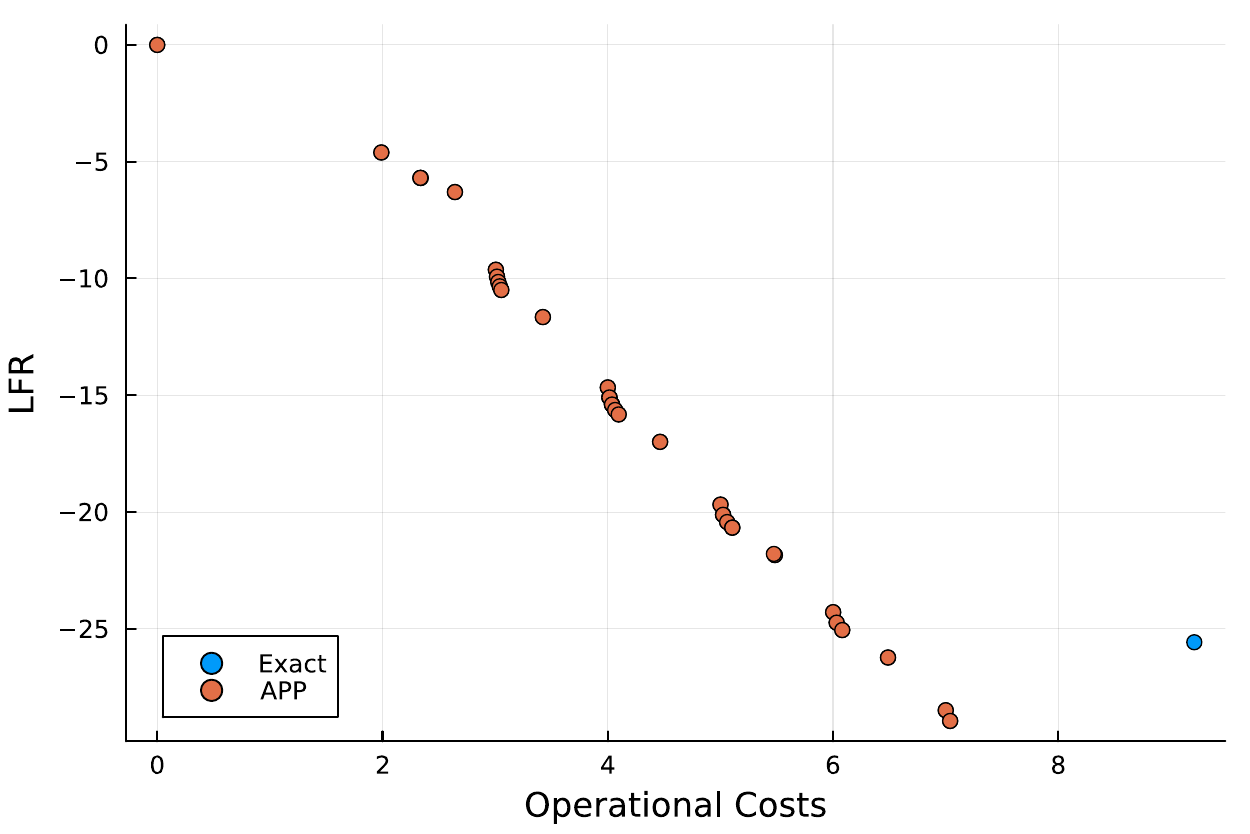}
    \caption{Pareto Front for Instance (13,35)}
    
\end{figure}
\begin{figure}
    \centering
    \includegraphics[width=0.9\linewidth]{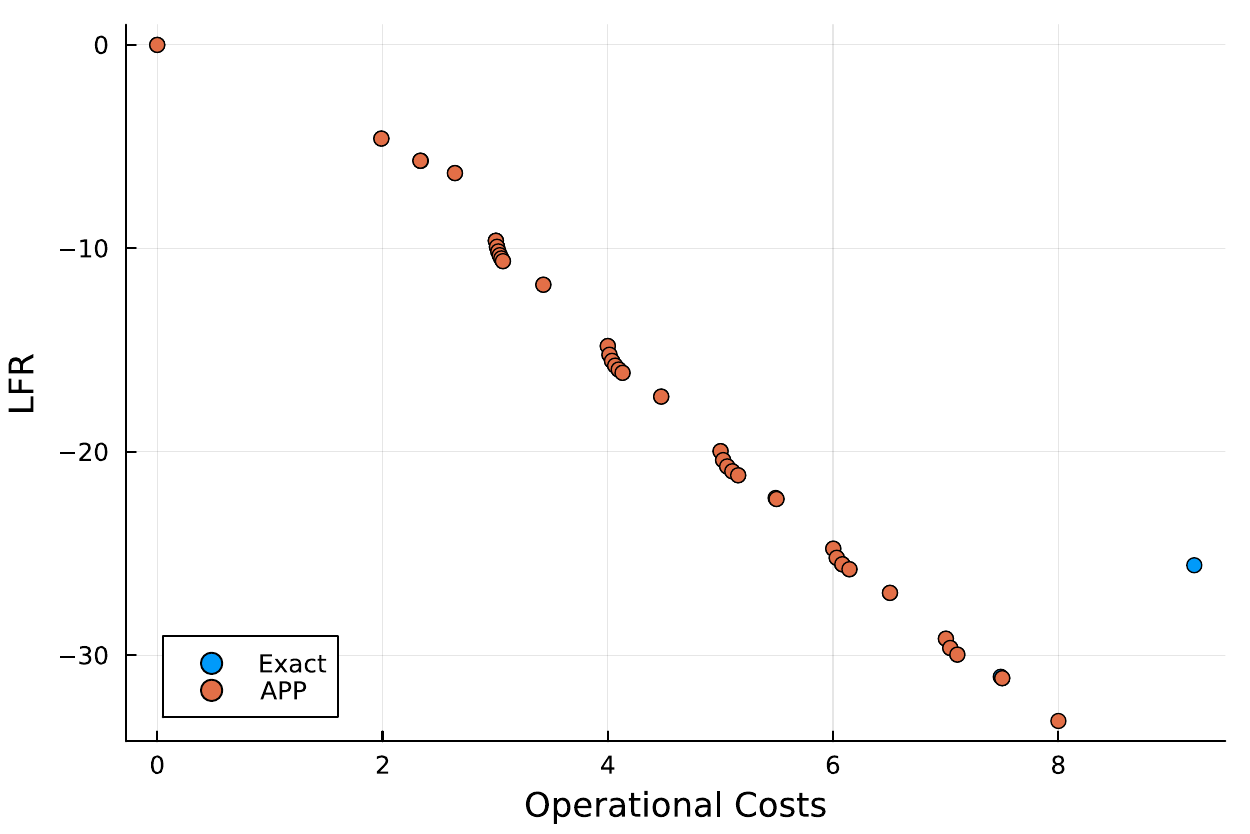}
    \caption{Pareto Front for Instance (13,40)}
    
\end{figure}
\begin{figure}
    \centering
    \includegraphics[width=0.9\linewidth]{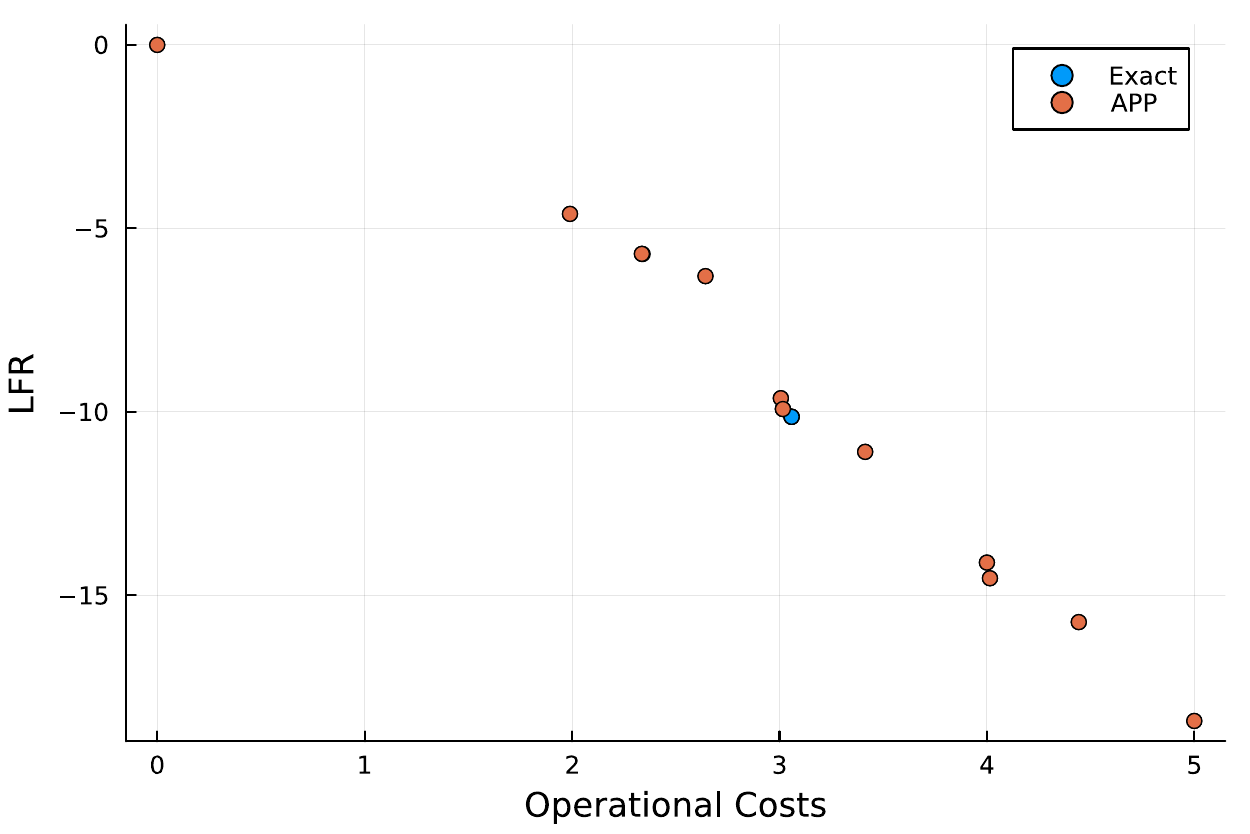}
    \caption{Pareto Front for Instance (14,36)}
    
\end{figure}
\begin{figure}
    \centering
    \includegraphics[width=0.9\linewidth]{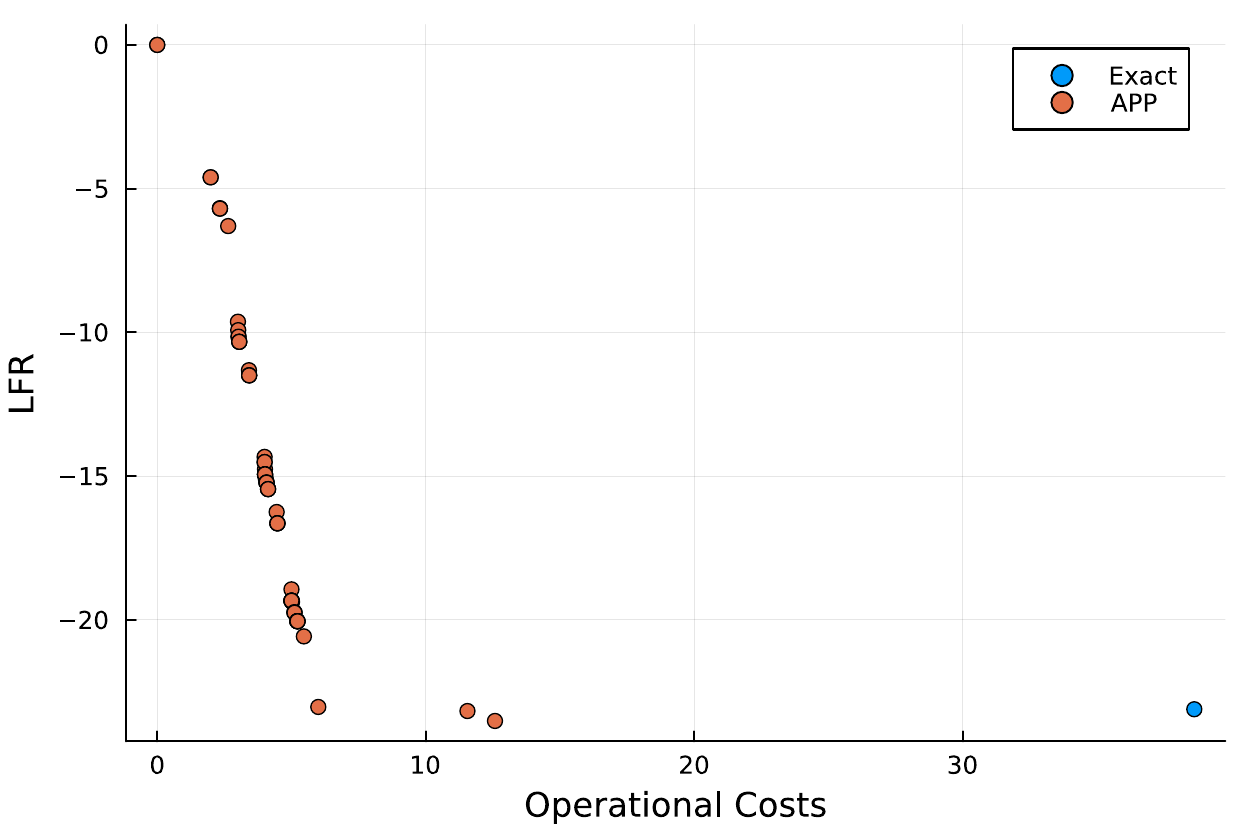}
    \caption{Pareto Front for Instance (14,48)}
\end{figure}

\newpage 
\subsection{Effect of event rate scaling}
\label{sec:study:rates}

\begin{sidewaysfigure}[tbp]
\centering
\caption[]{Comparison of Pareto fronts when varying event rates.}
\begin{subfigure}{.33\textwidth}
    \centering
    \includegraphics[width=\textwidth]{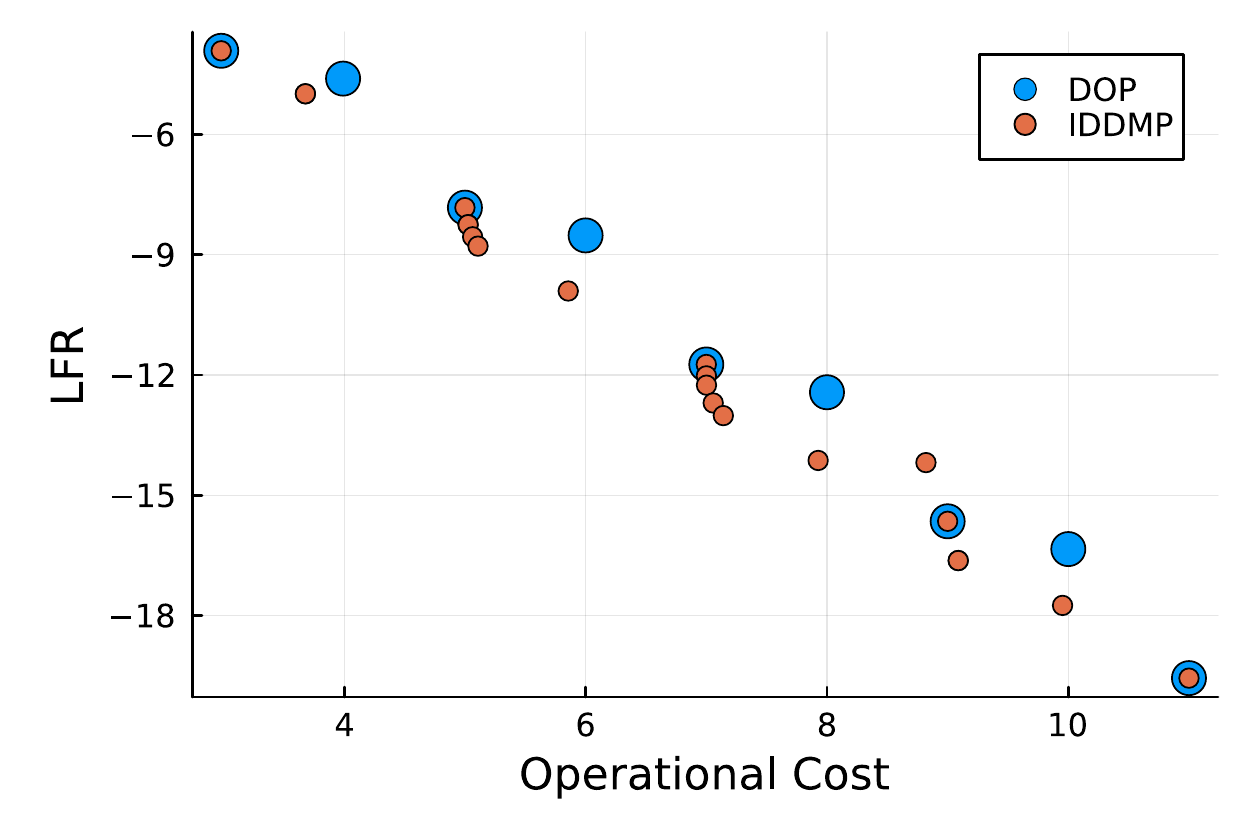}
    \subcaption{$1.0,1.0$}
\end{subfigure}%
\begin{subfigure}{.33\textwidth}
    \centering
    \includegraphics[width=\textwidth]{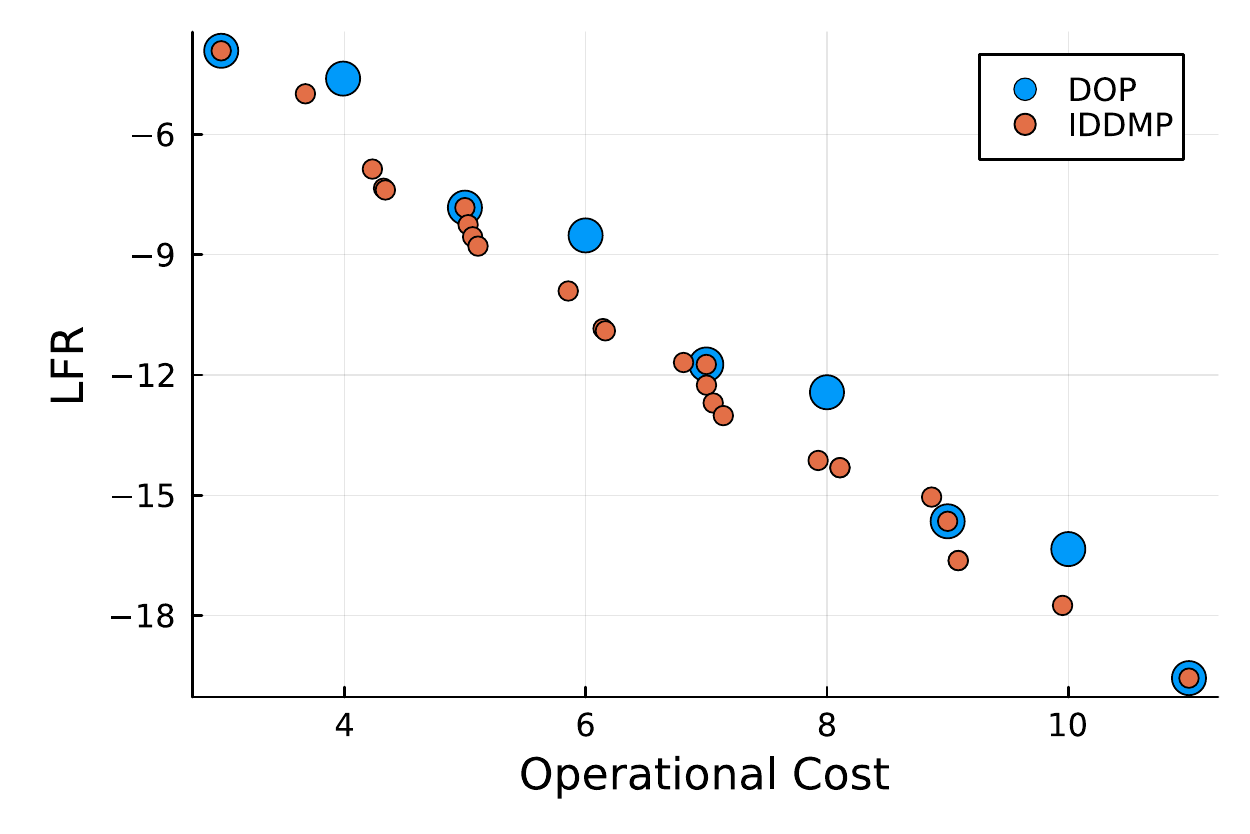}
    \subcaption{$1.0,5.0$}
\end{subfigure}%
\begin{subfigure}{.33\textwidth}
    \centering
    \includegraphics[width=\textwidth]{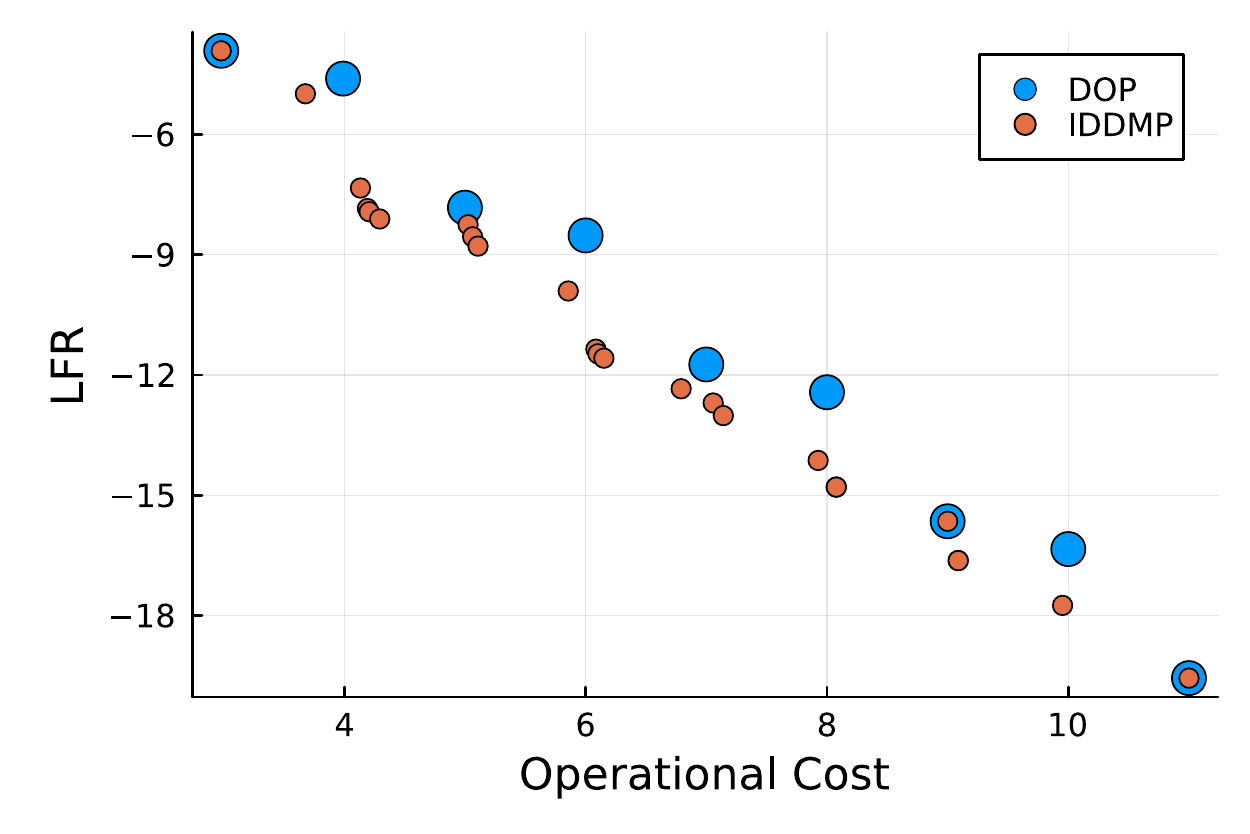}
    \subcaption{$1.0,10.0$}
\end{subfigure}
\begin{subfigure}{.33\textwidth}
    \centering
    \includegraphics[width=\textwidth]{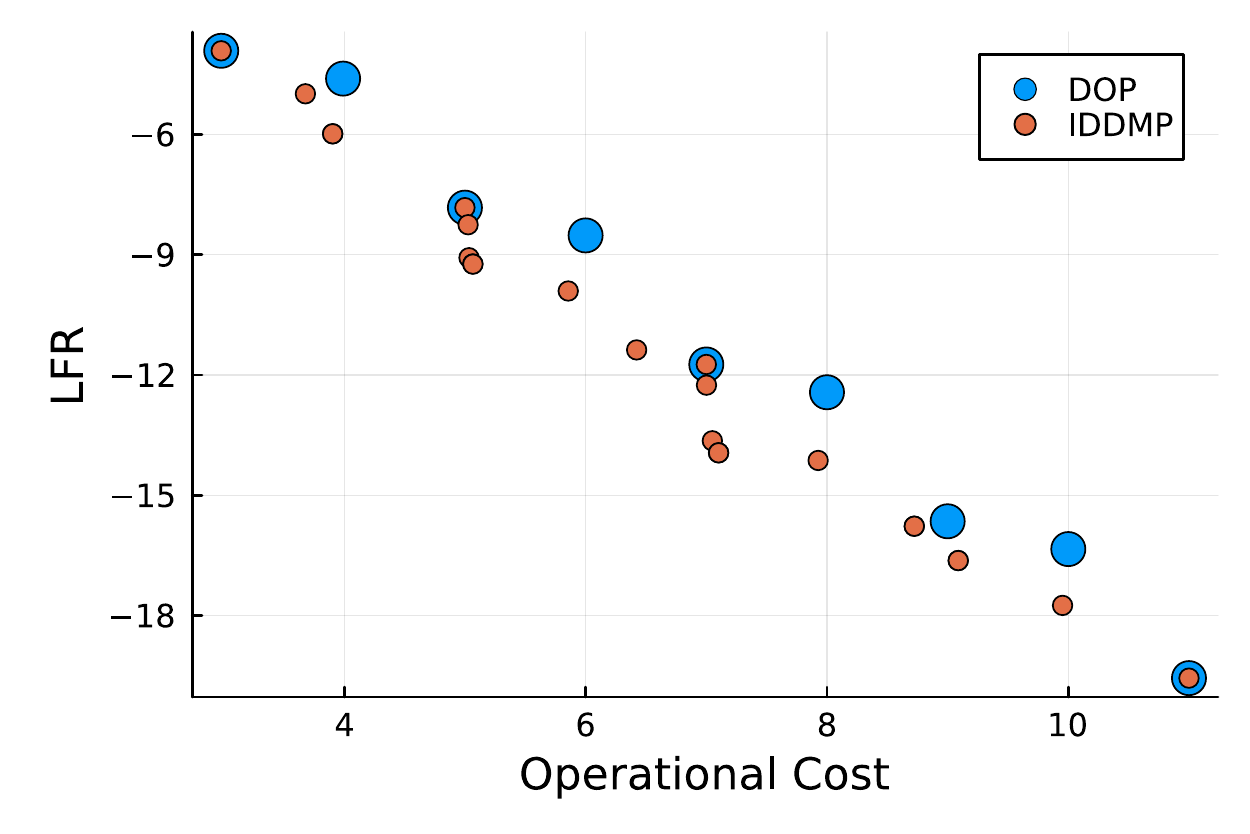}
    \subcaption{$5.0,1.0$}
\end{subfigure}%
\begin{subfigure}{.33\textwidth}
    \raggedleft
    \includegraphics[width=\textwidth]{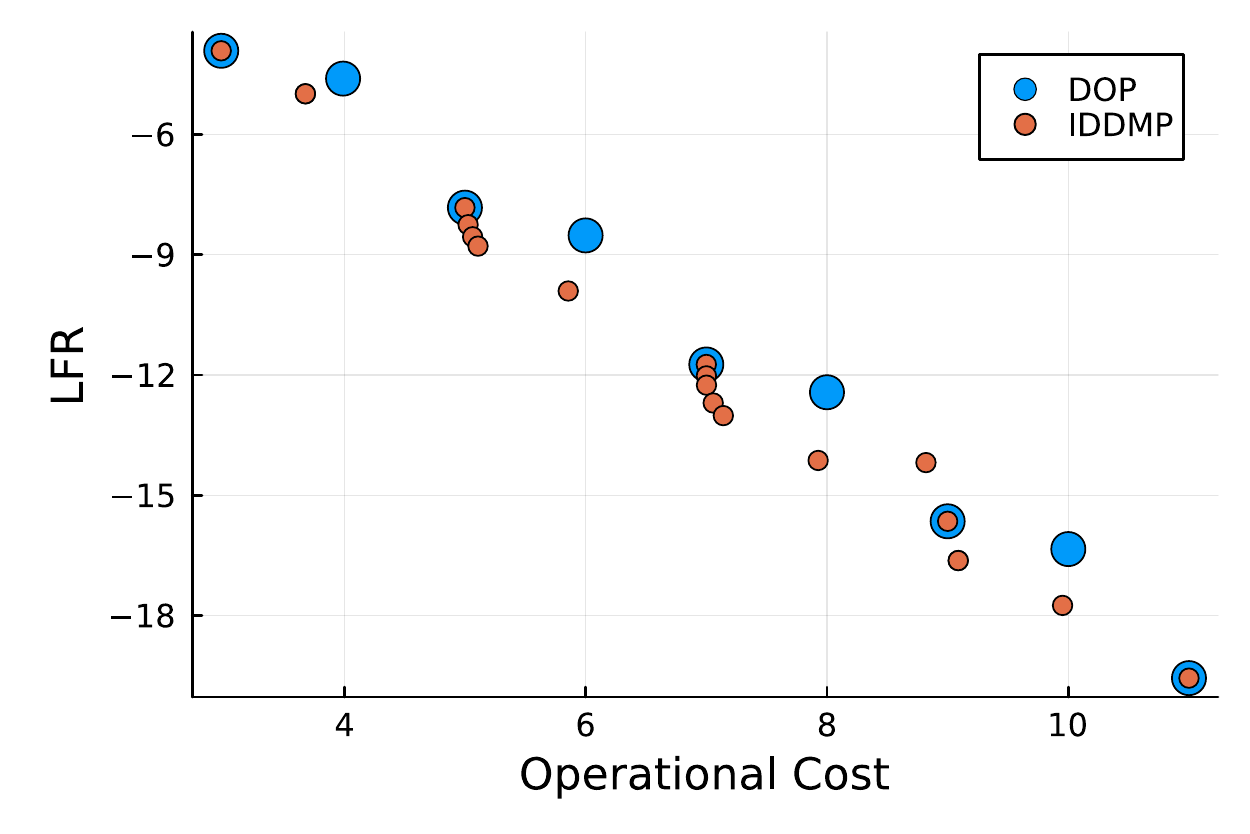}
    \subcaption{$5.0,5.0$}
\end{subfigure}%
\begin{subfigure}{.33\textwidth}
    \centering
    \includegraphics[width=\textwidth]{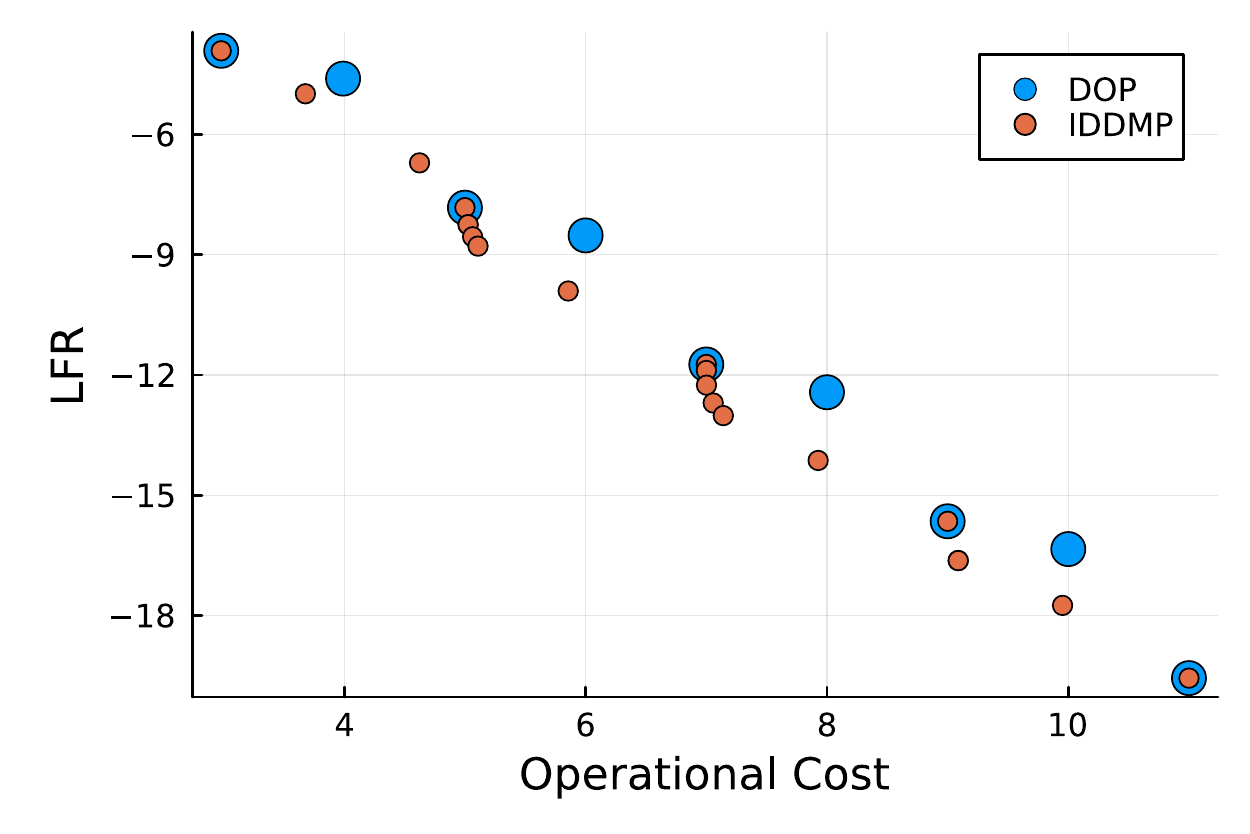}
    \subcaption{$5.0,10.0$}
\end{subfigure}
\begin{subfigure}{.33\textwidth}
    \raggedleft
    \includegraphics[width=\textwidth]{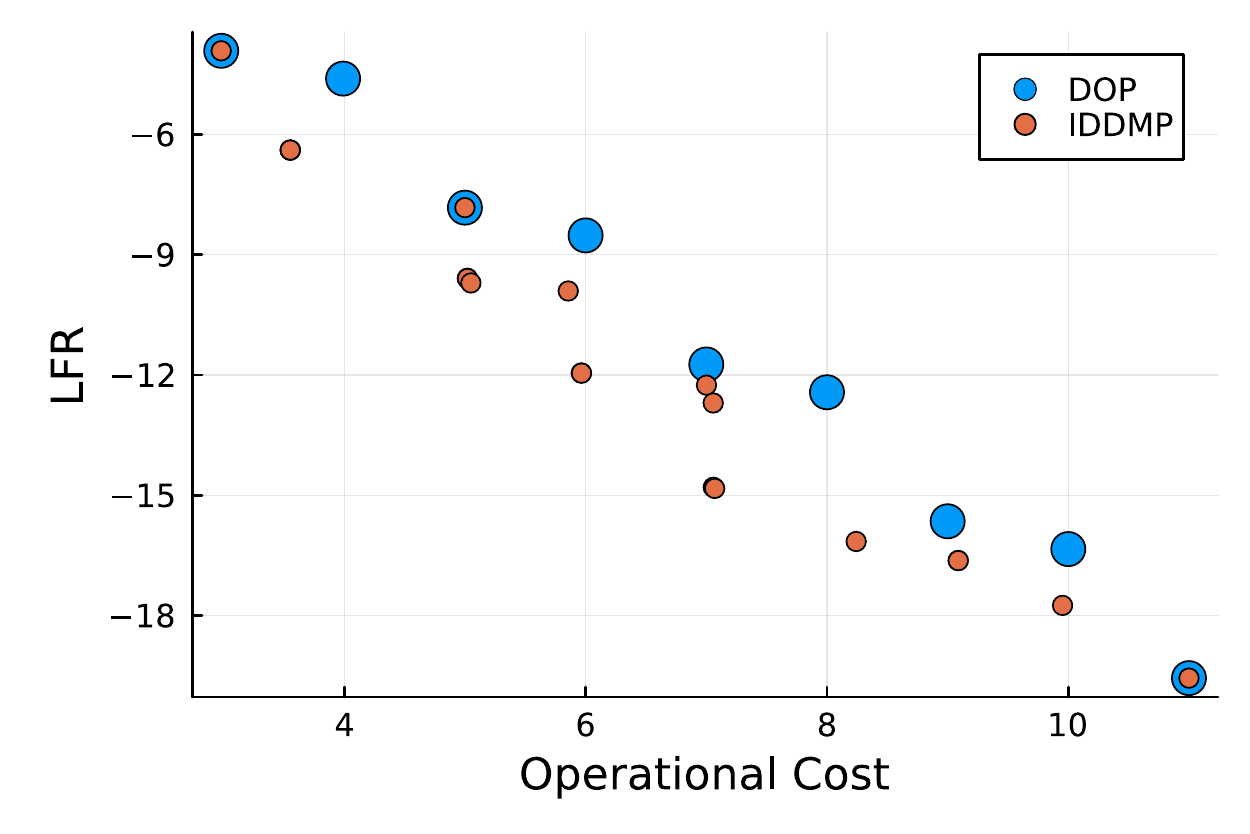}
    \subcaption{$10.0,1.0$}
\end{subfigure}%
\begin{subfigure}{.33\textwidth}
    \raggedleft
    \includegraphics[width=\textwidth]{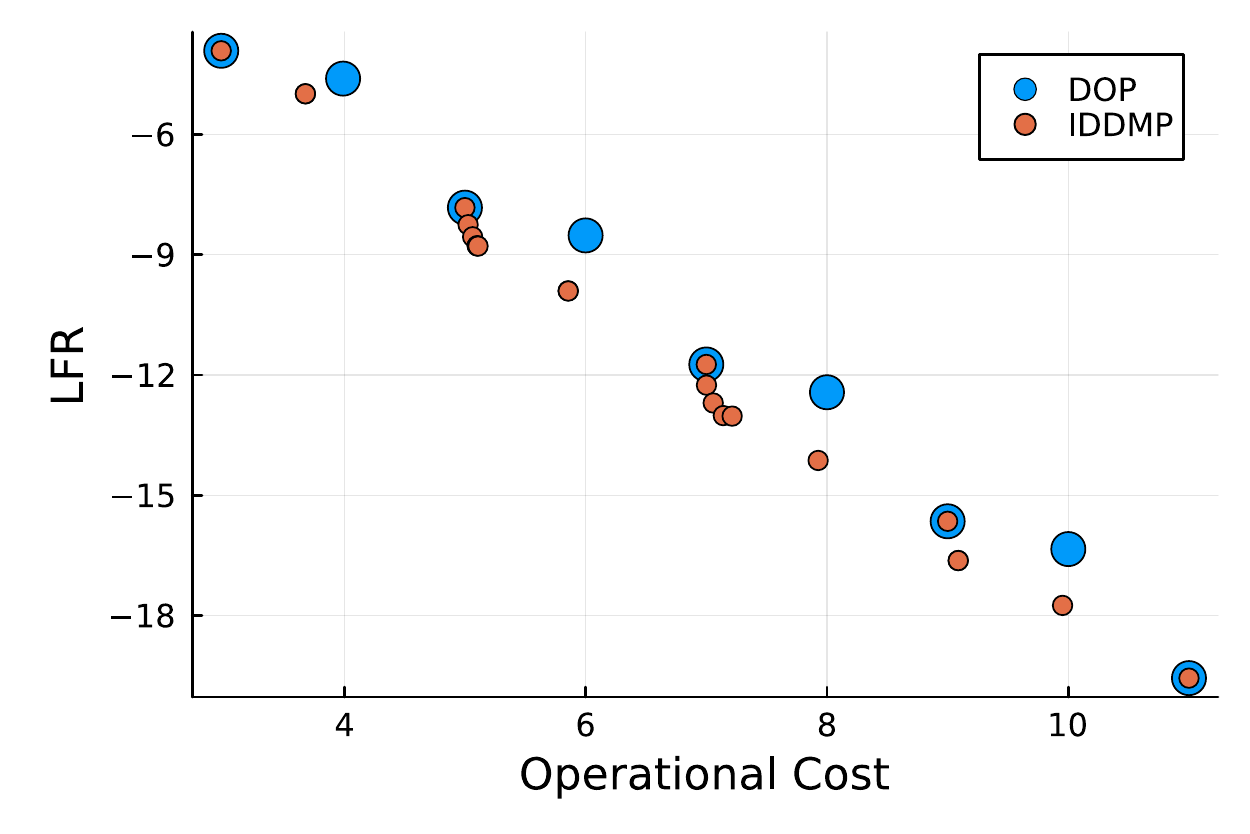}
    \subcaption{$10.0,5.0$}
\end{subfigure}%
\begin{subfigure}{.33\textwidth}
    \raggedleft
    \includegraphics[width=\textwidth]{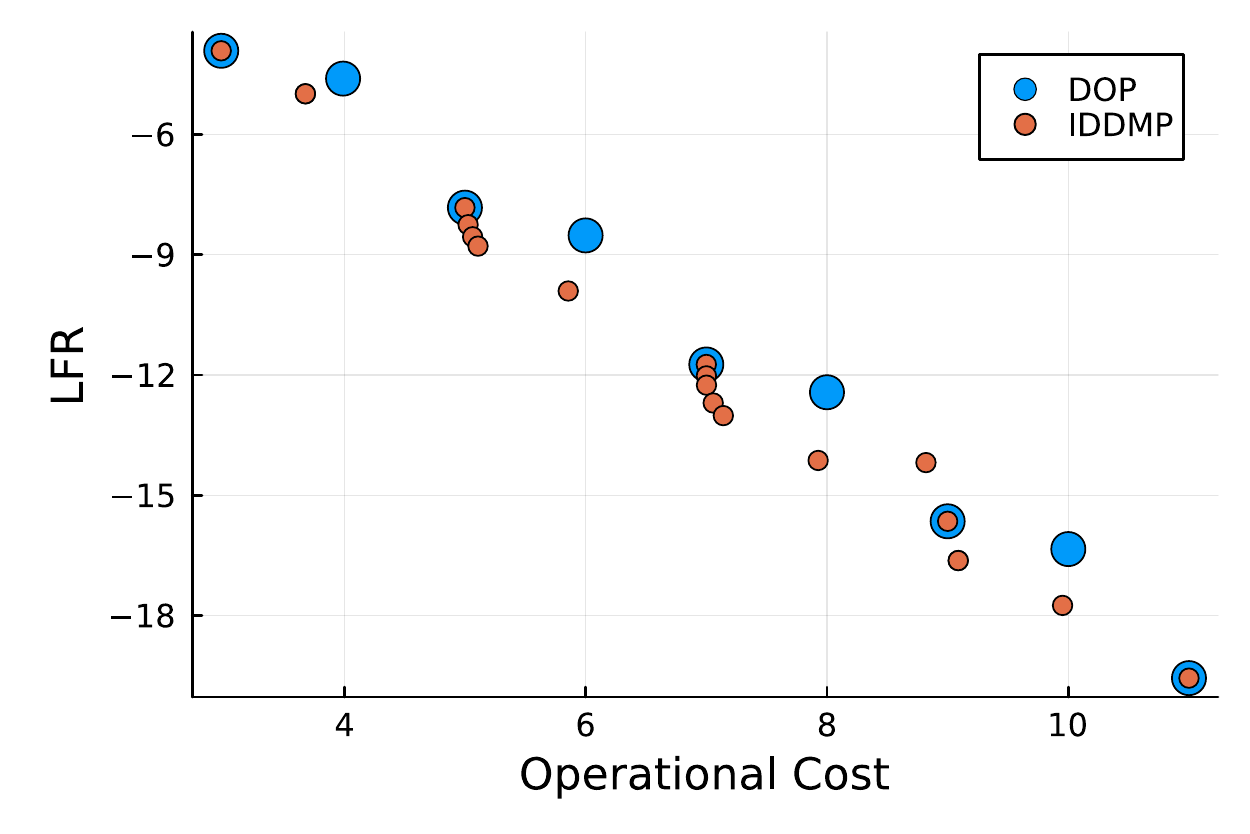}
    \subcaption{$10.0,10.0$}
\end{subfigure}
\label{fig:rateGrid}
\end{sidewaysfigure}

In this subsection we explore the effects of `event rate scaling', which refers to the simultaneous scaling of the $\alpha_i$ and $\tau_i$ parameters of individual component types. We again perform our analysis on instance 6-20, and further explore the $(r_1,r_2) = (300,100)$ case. If we multiply the values $\alpha_i$ and $\tau_i$ by some constant factor $m_i$, the reliability $p_i$ remains the same and therefore the underlying designs found by BO-DOP remain the same; however, the dynamic policies may yield different results. The effect of this multiplier is twofold: both failure and repair events happen faster, and the expected total cost of a single repair is affected, as this takes the value $r_i/(m_i\tau_i)$.

We consider multiplier values of 1.0, 5.0, and 10.0, and all combinations thereof, applied to components 1 and 2. As in the previous subsection, components 3 and 4 are ignored for this problem. \autoref{fig:rateGrid} shows the Pareto fronts for all combinations of the multipliers applied to components 1 and 2. It is immediately apparent that these fronts are not all the same, so the simultaneous scaling of these rates does indeed affect the outcomes of dynamic policies, despite not affecting the reliability value $p_i$. As previously mentioned, the design solutions are unaffected by the scaling of the rates, so the DOP solutions are the same across all cases. The fronts are identical for combinations (1,1), (5,5), and (10,10). This is because the ratio of the multipliers is equivalent across these cases, so these multipliers simply correspond to the re-scaling of time, and hence the solutions are unaffected. We also note that the fronts for combinations (5,10) and (10,5) are not dissimilar to the (1,1) front. Intuitively, this is because, under the previously described invariance due to scaling, these combinations are the same as (1,2) and (2,1) respectively. Such a difference may not be drastic enough to yield large changes to the Pareto fronts.

We recall from the main paper that for instance (6,20) the static solutions can be split into two groups: those without a copy of component 1, and those with a copy. We remarked previously that the static solutions with a copy of component 1 all became dominated by some dynamic policy. \autoref{fig:rateGrid} shows that this remains true across all cases for the multipliers. However, in some cases we see other DOP solutions become dominated as well. Case (1,10) sees two more DOP solutions become dominated by dynamic policies. Case (5,1) sees one extra DOP solution become dominated, which is especially interesting as this solution must use a design which includes component 1. So in this case, not only do we see DOP solutions with component 1 become dominated by dynamic policies applied to designs without that component, but we also see the opposite; that is, previously non-dominated DOP solutions without a copy of component 1 become dominated by a dynamic policy applied to a design that includes component 1. Case (10,1) takes this to the extreme, and sees all except three DOP solutions become dominated by dynamic policies. This suggests that the benefits of using dynamic policies may become stronger in situations where component types differ greatly with respect to frequencies of events occurring.

\end{document}